\documentclass[11pt,a4paper]{amsart}
\usepackage{amssymb,amsmath}
\usepackage{color}

\textwidth16cm \numberwithin{equation}{section} \vspace{8cm}
\newtheorem{theorem}{Theorem}[section]
\newtheorem{proposition}[theorem]{Proposition}
\newtheorem{corollary}[theorem]{Corollary}
\newtheorem{lemma}[theorem]{Lemma}
\newtheorem{definition}[theorem]{Definition}
\newtheorem{remark}[theorem]{Remark}
\newtheorem{hypothesis}[theorem]{Hypothesis}
\newcommand{\R}{{\mathbb R}}
\newcommand{\RN}{{\mathbb R}^N}

\newcommand{\vep}{\varepsilon}

\newcommand{\vfi}{\varphi}

\def\dys{\displaystyle}

\newcommand{\be}{\begin{equation}}
\newcommand{\ee}{\end{equation}}
\newcommand{\rife}[1]{{(\ref{#1})}}

\newcommand{\la}{\lambda}
\newcommand{\al}{\alpha}

\def\m{\noalign{\medskip}}
\def\hh{\vskip 1mm \noindent}

\def\de{\delta}

\def\vep{\varepsilon}

\def\vfi{\varphi}
\def\qed{{\unskip\nobreak\hfil\penalty50
          \hskip2em\hbox{}\nobreak\hfil\mbox{\rule{1ex}{1ex} \qquad}
   \parfillskip=0pt
   \finalhyphendemerits=0\par\medskip}}
 \def\la{\lambda}
 \def\ga{\gamma}
\def\rife#1{(\ref{#1})}

\begin{document}

\title[Lipschitz
regularizing effects for parabolic PDEs] { Global Lipschitz
regularizing effects for linear and  nonlinear parabolic equations
}

\bigskip

\author {A. Porretta}
\address{Dipartimento di Matematica\\Universit\`a di Roma ``Tor Vergata'' \\
Via della Ricerca Scientifica 1 \\
00133 Roma} \email{
porretta@mat.uniroma2.it}

\author{E. Priola}
\address{Dipartimento di Matematica\\Universit\`a di Torino
\\Via Carlo Alberto
10\\10123 Torino\\Italy } \email{ enrico.priola@unito.it}

\thanks {This work has been supported by the Indam
GNAMPA project 2008 ``Problemi di Diffusione degeneri'' and GNAMPA
project 2010 ``Propriet\`a di regolarit\`a in Equazioni alle
Derivate Parziali nonlineari legate a problemi di controllo''.}

%\centerline{Alessio Porretta  \&  Enrico Priola
%}

 \maketitle
 %\bigskip
 %\vskip1em

% \begin{abstract}
\centerline{\bf Abstract}

 In this paper
 we prove  global bounds on the spatial  gradient
   of viscosity solutions to
    second order linear and nonlinear parabolic
 equations   in   $(0,T) \times \R^N$.
Our assumptions include the case that the  coefficients  be
  both
 unbounded and with very mild local regularity (possibly  weaker than
the  Dini continuity), the  estimates only depending on the
    parabolicity  constant   and  on
   the modulus of continuity of coefficients
   (but  not on their  $L^{\infty}$-norm).  Our proof
    provides  the analytic
 counterpart to  the probabilistic proof used in Priola  and Wang
\cite{PW} (J. Funct. Anal.  2006)
 to  get this type of gradient estimates in the linear case.
  We actually extend such estimates to the case of possibly
  unbounded data and solutions as well as to the case of
 nonlinear operators including Bellman-Isaacs equations.
  We investigate both the classical regularizing effect
(at time $t>0$) and the possible conservation of Lipschitz
regularity from $t=0$, and  similarly we prove  global H\"older
estimates under weaker assumptions on the coefficients.  The
estimates we prove  for unbounded data and solutions
 seem to be new even in the classical case of linear  equations
 with bounded and H\"older continuous coefficients.
  Applications to Liouville type theorems are also
  given in the paper.
 Finally, we compare in an appendix the analytic and the
probabilistic approach  discussing the analogy between the doubling
variables  method of viscosity solutions and the probabilistic
coupling method.
%\end{abstract}

\vskip1em


\tableofcontents

% \vskip1em

\section{Introduction}

It is well known that the bounded solution $u(t,x)$ of the heat
equation in $\R^N$, even if it is only bounded at time $t=0$,
becomes Lipschitz as soon as $t>0$ with a  global Lipschitz bound of
the type $\|Du(t, \cdot)\|_\infty\leq \frac{\|u_0\|_\infty}{\sqrt
t}$.

In this paper we prove a similar
 global gradient bound
for solutions of the  Cauchy problem concerning both linear and
nonlinear
parabolic operators. In particular, our main goal is to obtain this
global gradient estimate independently of the $L^\infty$-bounds of
the coefficients of the operator but only depending on their modulus
of continuity. In case of solutions of linear diffusions, such
bounds are  well-known whenever either the coefficients are
uniformly continuous and bounded (see \cite{St}) or the coefficients
are unbounded but at least $C^1$ with a suitable control on the
derivatives (see \cite{EL}, \cite{Ce0}, \cite{L1}, \cite{Ce},
\cite{BF}, \cite{KLL} and the references therein).
   Note that linear parabolic equations with unbounded
 coefficients also arise
 in some models of
 financial mathematics  (see, for instance, \cite{HR}
 and the references therein).
 Recently, in \cite{PW}
  E. Priola and F.Y. Wang
 obtained  uniform gradient estimates
   when
the coefficients are unbounded and singular at the same time. More
precisely, they proved that the uniform estimate
\be\label{est-base}
\|   D u(t, \cdot)\|_\infty \leq  \frac K{\sqrt{t \wedge
1}}\,\|u_0\|_\infty, \;\;\; t>0,
 \ee holds for the  diffusion semigroup
solution $u= P_t u_0$  of the Cauchy problem
%corresponding to  the linear operator
\be \label{pwpw} \begin{cases}\partial_t u - A(u)=0 \quad
\mbox{in}\,\, (0,T)\times \R^N \\  u(0, \cdot)=u_0
\end{cases}  \text{where} \;\;\;
 A=
 \sum\limits_{i,j=1}^N
q_{ij}(x)\partial^2_{x_i x_j}+ \sum\limits_{i=1}^N
b_i(x)\partial_{x_i},
% - V(x)f
\ee only assuming that $q_{ij}$ and   $b_i$
%and $V$
are continuous and
that $q(x)$ is elliptic, more precisely  if
 \be \label{qq}
 q(x)=\la \,I+ \sigma(x)^2, \;\; \;\; x \in \R^N,
 \ee  for some $\la>0$
and some symmetric $N\times N$ matrix $\sigma(x)$ satisfying,
jointly with the drift  $b(x)$, the following condition:
\be\label{cond-base}
 \|\sigma(x)-\sigma(y)\|^2+
(b(x)-b(y))\cdot(x-y)\leq g(|x-y|)\, |x-y|,\quad  \,x,y\,\in \R^N,\,
0<|x-y|\leq 1, \ee for some nonnegative function $g$ such that
$\int_0^1 g(s)\,ds<\infty$.
In this case, the constant $K$ in
\rife{est-base} only depends on $\la$ and $g$. Observe that, when
$b\equiv 0$, this is  a slightly weaker assumption than asking
$\sigma(x)$ to be Dini-continuous. Moreover, both the diffusion and
the drift could be unbounded, in particular $b$ could be the
gradient of whatever concave $C^1$-function   (as in standard
Ornstein-Uhlenbeck operators).
% \vskip1em

The proof of the result in \cite{PW} stands on  probabilistic
arguments and relies in particular on the extension to the case of
non locally Lipschitz coefficients of the coupling  method for
diffusions (see e.g. \cite{Lin-Rog}, \cite{Chen-Li},
\cite{Cranston}, \cite{Cranston1}).  Possible extensions of
this regularity result to infinite dimensions  in the setting of \cite{DZ} are given in \cite{WZ}.

The aim of our work is twofold. On one hand we give a new proof,
using entirely analytic arguments, of the result by Priola and Wang;
the only difference here is that, in place of the assumption that
the underlying stochastic process is nonexplosive, we  assume
 directly the existence of a smooth Lyapunov function, which is  a
 stronger condition though more adapted to our analytical approach.
On the other hand, we extend the same result in two main features,
including in our results both the case of {\em unbounded data and
solutions}  (as well as possibly larger classes of unbounded
coefficients) and  the case of {\em nonlinear operators}. Further
minor improvements with respect to \cite{PW} concern  the fact that
we consider  time dependent coefficients or the case of H\"older
estimates or unbounded potentials  (see e.g.  Theorem
\ref{PWc}).

In order to allow for unbounded data and solutions, we review the
gradient bound \rife{est-base} by distinguishing two main steps in
the a priori estimate.   In a first step, one  obtains a global
gradient bound depending only on the oscillation of the solution $u$
(see Theorem \ref{PW}). In a second step, one looks for conditions
which allow to estimate the oscillation of $u$. In case of bounded
data, the second step is easy since solutions turn out to be
globally bounded. However, it is also possible to consider unbounded
data, though yet with bounded oscillation, and prove that solutions
have bounded oscillation  as well, and then, by the first
step, be globally Lipschitz for $t>0$. Let us stress, however, that
this latter estimate on the oscillation of $u$ requires
\rife{cond-base} to hold \emph{for every $x,y \in \R^N$} and  a
linear growth assumption on the function $g$ at infinity is
required.

 Moreover, recall that our estimates rely on the assumption
that a  Lyapunov function $\varphi$ exists for $A$, which may
implicitly be linked to a compatibility condition in the growth at
infinity of the drift and diffusion coefficients.
For instance, the requirement that $\varphi (x) = 1 + |x|^2$, $x \in
\R^N$, is a  Lyapunov function for $A$ corresponds to the
hypothesis
 \be \label{lia}
 \text{tr}(q(x)) +  \, b(x) \cdot x \le C (1+ |x|^2),\;\;\; x\in
 \R^N,\,\,
\ee
for some $C\geq 0$.

In order to fix the ideas, our global result in
the linear case  (at least when the equation has no source
terms and coefficients are independent on time as in \rife{pwpw})
 reads as follows:

\begin{theorem}
   Assume \rife{qq} and  that \rife{cond-base} holds true for
every $x,y\in \R^N$,   where $g\in L^1(0,1)\cap
C(0, +\infty)$, $g(r)r\to 0$ as $r\to 0^+$ and $g(r)=O(r)$ as $r\to
+\infty$,  that a Lyapunov function $\vfi$ exists  satisfying \rife{lyap}  and that $u_0 \in
C(\R^N)$ satisfies
\be\label{u00}
|u_0(x)-u_0(y)|\leq k_0+k_\alpha|x-y|^\alpha+k_1|x-y|\,, \qquad
x,y\in \R^N,
\ee
 for some $ \alpha \in (0,1)$ and $k_0, k_{\alpha}, k_1 \ge
0$.  Then any viscosity solution $u\in C([0,T]\times \R^N)$ of the
Cauchy problem \rife{pwpw}
%$$
% \partial_t u - A(u)=0 \quad \mbox{in}\,\, (0,T)\times
%\R^N\,,\qquad u(0)=u_0,
%$$
such that  $u=o(\vfi)$ as $|x|\to \infty$  (uniformly in $[0,T])$
satisfies that $u(t):= u(t, \cdot)$ is Lipschitz continuous  for $t\in (0, T)$
and \be\label{000} \|   D u(t)\|_\infty \leq  c_T \Big\{ \frac
{k_0} {\sqrt{t \wedge 1}}\, + \frac { k_{\alpha}} {(t \wedge
1)^{\frac {1-\al}2}} + k_1\Big\},\qquad \quad \ee for some $c_T,$
 possibly depending  on $\alpha, \la$ and $g$.
\end{theorem}

%\vskip0.5em
We refer the reader to Sections 3.2 and 3.3 for more general results of this type  (see Theorem \ref{PWc},
Lemma \ref{osc-u} and Theorem \ref{grow}).  Due to the fact that $u_0$ can be unbounded  the previous result  seems  to be new
even in the classical case of linear equations  with bounded and H\"older continuous coefficients (note that, in particular, if $k_0=k_{\alpha}=0$ the above estimate \rife{000} implies that  unbounded Lipschitz data, as $u_0=|x|$, yield globally Lipschitz solutions).
%, even  if the coefficients are unbounded and only %mildly continuous).
We also mention that  the same  approach is used to get
H\"older estimates, in which case we relax  the
assumptions on the coefficients (the function $g$ needs only satisfy $g(r)r\to 0$ as $r\to0^+$)  including, for instance, merely
uniformly continuous  coefficients in the diffusion part.
\vskip0.2em
The extensions of the previous results to nonlinear operators
include both the case of Bellman-Isaacs' type  operators, namely sup
or inf\,(sup) of linear operators, which arise naturally from
stochastic control problems as well as from differential game theory
(see e.g. \cite{Kr80}, \cite{Fle-Soner}, \cite{DaLio-Ley}, \cite{K}), and the case of operators
with general nonlinear first order terms, including possibly
superlinear growth in the gradient. Similar extensions are possible
since we prove  the Lipschitz bound
%through
%the method and
in the framework of viscosity solutions (this notion
 also generalizes, in the linear case,  the notion of
  probabilistic solution,
 see \cite{Fle-Soner}),   combining the method introduced by  H. Ishii and P.L.
Lions (\cite{IL})  with  the coupling idea used in \cite{PW}.
The Ishii and Lions method is by now classical to obtain gradient
bounds for viscosity solutions though mainly used  in bounded
domains or for local estimates (see e.g. \cite{Ba1}, \cite{Ba2},
\cite{Chen93} and references therein).
  The main novelty here consists in getting {\em global} Lipschitz estimates on $\R^N$ for possibly unbounded coefficients. To this purpose our assumptions  (see e.g. Hypothesis \ref{F} below) take into account both  the possible dissipation at infinity  of  the drift term $b$  and  the minimal (local) regularity of the coefficients. In this respect, we mention that \rife{cond-base} gives  a slight
improvement, at least in the model linear case,
of the assumptions used in \cite[Section III.1]{IL} to get local
Lipschitz continuity (see also Remark
\ref{cfr-IL} and \cite[Section III.1]{Ba1} for similar improvements).
Let us  point out that some dissipation condition on the drift term is somehow necessary when this is unbounded.
 A counterexample in this sense is given in \cite{BF},  even in the linear case and for   $C^1$-coefficients.

 Eventually, for the nonlinear problem
\be \label{pwpw-nl}
\begin{cases}\partial_t u + F(t,x,Du,D^2 u)=0 \quad
\mbox{in}\,\, (0,T)\times \R^N \\  u(0, \cdot)=u_0
\end{cases}
\ee
we give similar results as those mentioned before for the linear case, under two main structure assumptions. The first one generalizes   to a nonlinear setting \rife{qq}--\rife{cond-base}. To give an idea,   in the simplest case  it  reads as follows.

\begin{hypothesis}\label{F}
There exist
$\lambda>0$ and  non-negative functions
  $\eta(t,x,y)$   and
  $g\in C((0, 1); \R_+ ) \cap L^1(0,1)$ such that
  \be\label{feq0}
  \begin{array}{c}
 \!\!\!\!\!\!\!\!\!\!\!\! F(t, x, \mu(x-y), X)- F(t,y,\mu(x-y),Y)
  \\
  \m
  \qquad\qquad \qquad\quad
  \geq
 - \la {\rm tr}\left(X-Y\right)- \mu |x-y| \, g(|x-y|)
  - \nu \, \eta(t,x,y)\,,
  \end{array}
\ee for any $\mu> 0$, $\nu \ge 0$, \  $ x,y\in \R^N\,,
  \;\; 0<|x-y|\le 1,
  \;\;\; t\in (0,T),\,\quad    X,Y\in {\mathcal S}_N$:
\begin{align} \label{ine0}
  \begin{pmatrix}
 X & 0    \\ 0 & -Y
 \end{pmatrix}
 \leq  \mu \, \begin{pmatrix}
 I & -I    \\ -I & I
\end{pmatrix}  +  \nu \begin{pmatrix}
 I & 0    \\
  0 & I
\end{pmatrix}.
%\qquad \hbox{for some matrix $B$: $B\leq \mu I$.}
 \end{align}
\end{hypothesis}

Actually, Hypothesis \ref{F} is only  a simplified version of the
one that we will use (see Hypothesis \ref{pw-nl}) in order to
include more general lower order terms. However, \rife{feq0} already
contains the main  features of operators which are possibly
nonlinear in the second order terms.

The second basic structure assumption is meant to get rid of the
infinity, again requiring the existence of a Lyapunov function. To
this extent, a general condition will be introduced later (see
Hypothesis \ref{pert-fi}), whose simplest  (more readable) version
is the following.
\begin{hypothesis} \label{L}
 There  exist $
 \varphi \in C^{1,2}(\bar Q_T)\,, \;    \vep_0 > 0\,$:
$$
\begin{array}{l}
\begin{cases}
\vep \partial_t \varphi + F(t,x, p+ \vep D\varphi, X+ \vep
D^2\varphi)- F(t, x,p, X)\geq  0 &
\\
\m \qquad \hbox{for every $(t,x)\in Q_T$, $p\in \RN$, $X\in {\mathcal S}_N$, and every $\vep\leq
\vep_0$} &
 \\
 \m
\varphi(t,x)\to +\infty \quad \hbox{as $|x|\to \infty$, uniformly
for $t\in [0,T]$.} &
\end{cases}
\end{array}
$$
\end{hypothesis}

As in the linear case, the statements that we prove for the nonlinear problem \rife{pwpw-nl}
can be summarized in the following two types of estimates:

\begin{theorem}\label{semin} Assume that Hypotheses  \ref{F} and \ref{L} hold.
\vskip0.5em
(i) Any viscosity solution  $u$  of \rife{pwpw-nl} which has  bounded oscillation is Lipschitz continuous in $(0,T)$ and satisfies
$$
\|   D u(t)\|_\infty \leq  \frac {C}{\sqrt{t \wedge 1}},
$$
where $C= C\left({osc}_{(\frac t2\, , \, T\wedge \frac{3}{2}t)}(u),
g,\lambda\right)$. \vskip0.3em (ii) Assume in addition that
\rife{feq0} holds for every $x,y\in \R^N$, and the function $g$
satisfies  $g(s)s\to 0$ as $s\to 0^+$ and $g(s)= O(s)$ as $s\to
+\infty$. If  $u_0$ satisfies \rife{u00}, then any viscosity
solution $u$ such that   $u=o(\varphi)$ as $|x|\to \infty$
(uniformly in $[0,T])$ has bounded oscillation and  satisfies
$$
\|   D u(t)\|_\infty \leq  c_T\left\{ \frac
{k_0} {\sqrt{t \wedge 1}}\, + \frac { k_{\alpha}} {{(t \wedge
1)}^{\frac{1-\alpha}2}} + k_1\right\}
$$
for some  $c_T(T, \lambda, g, \alpha)$.
\end{theorem}

Theorem \ref{semin} is the typical statement that we prove, with several possible variations; for instance, as mentioned before, we obtain similar results replacing the Lipschitz bound with a H\"older global  estimate under a slightly different setting of assumptions.

 Let us notice that, in the viscosity solutions approach, it is
quite clear that  a general gradient estimate should depend firstly
on the oscillation of $u$  (statement (i) above), while a global estimate in turns depends on the
growth of the solution and of the initial data   (statement (ii) above), requiring some extra condition on the long range oscillation of the coefficients (i.e. when $|x-y|\to \infty$).  A similar two-step approach, even if in a  different context, can be found in  \cite{GGIS}.  In this way,  we obtain estimates
and regularity results for viscosity solutions which are possibly
unbounded, with the only condition that their growth be dominated by
the Lyapunov function's growth. Unfortunately, in the nonlinear case
the existence of such  a Lyapunov function remains as an implicit
condition  to be   possibly checked on
the particular examples of equations. In   Section 5 we
give simplified, more readable versions of this condition in at
least two cases;   for  Bellman-Isaacs operators and
for  operators with linear second order part and
nonlinear first order terms. In particular, our result for
Bellman-Isaacs equations of the type
\be \label{isa}
 \partial_t u + \inf\limits_{\beta\in \mathcal B}\sup\limits_{\alpha\in \mathcal
A}\, \left\{-{\rm tr}\left(q_{\alpha,\beta}(t,x)D^2u\right)-
 b_{\alpha,\beta}(t,x)\cdot   D u  -f_{\alpha,\beta}(t,x)\right\}
 =0
\ee reads very similar as in the linear case (see in particular
Corollary \ref{fin2}). We obtain indeed an estimate like \rife{000}
by requiring that the coefficients $q_{\alpha, \beta}(t,x)$ and
$b_{\alpha,\beta}(t,x)$ satisfy hypotheses \rife{qq} and
\rife{cond-base}  uniformly with  respect to $t \in (0,T)$,
$\alpha$, $\beta$ and that  a Lyapunov function exists, again
uniformly in $\alpha, \beta$.
It should  be noticed that in  some
examples, if the coefficients have a controlled growth at infinity,
Lyapunov functions with polynomial growth at infinity can be
constructed at hand (see e.g. \cite{BBBL}). For example, in case of
Bellman-Isaacs operators,
%we notice that
$\vfi=1+|x|^2$  is a Lyapunov function provided  (as in \rife{lia})
$$
{\rm tr}\left(q_{\al,\beta}(t,x) \right)+ b_{\alpha,\beta}(t,x)\cdot x
\le C(1+ |x|^2)\quad (t,x)\in (0,T)\times \R^N\,,\qquad \alpha\in {\mathcal A},\beta\in \mathcal B\,.
$$
Of course, the Lyapunov function condition is no
more needed if one is only interested in \emph{local} Lipschitz
estimates (see Theorem \ref{liploc}).
Moreover, we  point out that the continuity assumptions made on the
coefficients, see e.g.  \rife{cond-base}, are far beyond the assumptions currently used for the
comparison principle of viscosity solutions  to hold.
Therefore, the
Lipschitz, or H\"older estimates proved, allowing for local
compactness in the  uniform topology, can also be used to obtain the
existence of continuous viscosity solutions
(possibly globally Lipschitz for $t>0$), whereas the Perron's
method is not usable since the comparison principle is missing.
\vskip0.4em

Let us further mention that,   if one takes care of the dependence on $T$ in the estimates of Theorem \ref{semin}, one can deduce some new Liouville type result  for the stationary problem
\be\label{liustat}
F(x,Dv,D^2 v)=0
\ee
in the same spirit as one deduces that bounded harmonic functions are constant from the estimate $\|Du(t, \cdot)\|_\infty\leq \frac{\|u_0\|_\infty}{\sqrt
t}$ by letting $t\to +\infty$. A statement of such results is as follows.

\begin{theorem} \label{liu0} Assume that  $F$
is independent of $t$ and    satisfies Hypothesis \ref{L} and
Hypothesis \ref{F} for every $x,y\in \R^N$ with some function $g\in
C((0, +\infty); \R_+ ) \cap L^1(0,1)$.
\begin{itemize}
\item[(i)] If  $g$   satisfies
$$ \int_0^{+\infty} e^{-\frac{1}{4 \lambda} \int_0^r g(s)ds} dr =
+\infty, $$ then any bounded  viscosity solution  $v \in C_b(\R^N)$
of \rife{liustat}  is constant.
\item[(ii)] If $g$ satisfies
$$\limsup_{s\to \infty} \, g(s) \,s < 4\la(1-\alpha)\,, $$ then any
viscosity solution  $v $ of \rife{liustat}
 which verifies,
 for some $\alpha \in (0,1)$, $k_0 > 0$
 \be \label{frgr1}
|v(x) - v(y)|\le k_0 (1+  |x-y|^{\alpha}),\;\;\; x,y
\in \R^N,
\ee
must be constant.
\end{itemize}
\end{theorem}
  Apparently, the case (ii) seems to be new even
 in the linear case.
 Note that condition \eqref{frgr1} is optimal with respect to
$\alpha$. Indeed, any affine function is harmonic on $\R^N$.
\vskip0.4em

Last but not least, revisiting the method of \cite{IL}
we also show the common denominator which exists between the
 probabilistic coupling method  and the analytic approach to such gradient estimates,
which we closely compare in Section \ref{prob-anal} in the model linear case. Indeed, when
the doubling variables method of viscosity solutions is specialized
to linear diffusions, it offers  a truly analytic counterpart of the
probabilistic approach; this latter one, in turn, gives an
interesting interpretation of the viscosity techniques in terms of
coupling of probability measures. Although the links between the viscosity solutions theory with  stochastic processes are well known, especially in the theory of controlled Markov
 processes (see e.g. \cite{Kr80} and \cite{Fle-Soner}),
it seems that a direct comparison between the coupling method  and
the doubling variable technique was not addressed before, and  may
be interesting in itself.

\vskip0.4em The plan of the paper is the following. We start with
some notations and preliminary notions about viscosity solutions in
Section \ref{prel}. In Section \ref{lin} we deal with the linear
case, offering, in particular,  an analytical proof of the  main
result of \cite{PW}. Section \ref{nonlin} contains the extension to
nonlinear operators and further comments.  Section 5 contains examples and
 applications to Liouville type theorems. Let us stress here that
{\em our approach can be extended to deal with more general
quasilinear operators}, but we have chosen  to postpone this
extension eventually to future work, in order to avoid an increasing
number of technicalities. Finally, Appendix \ref{prob-anal} is
devoted to a possible comparison between the probabilistic and the
analytical approach, which we hope be of some interest to experts of
different techniques. In particular, restricting   to the linear
case, we rephrase the viscosity solutions approach in a form which
is closer to the  probabilistic coupling method.

\section{Preliminaries and notations}\label{prel}

In the following, we set $Q_T=(0,T)\times \RN$ and indicate by
 $\bar
 Q_T$ its
 closure.
 We denote by $C(Q_T)$ the set of real continuous functions defined
 on $Q_T$ and by  $C^{1,2}(Q_T)$ the subset of
 functions which are $C^1$
in $t$ and $C^2$ in $x$.  If $u \in C(Q_T)$ we often set $u(t) =
u(t, \cdot)$, $t \in (0,T)$. By $W^{1,\infty}(\RN)$ we denote, as
usual, the set of real Lipschitz functions on $\RN$; for $f \in
W^{1,\infty}(\RN)$, $\| f\|_{\infty}$ is the supremum norm of $f$ on
$\R^N$ and  Lip$(f)$  the smallest Lipschitz constant of $f$.

Given a function $v : I \times \R^N \to \R$, where $I \subset \R$ is
an interval, we say that $v$ has  \textit{bounded
 oscillation} in $I$ if
\begin{align} \label{osci}
{\rm osc}_{I}(v) =\sup_{x,y \in \R^N,\, |x-y| \le 1,\, t \in I}
|v(t,x) - v(t,y)| < + \infty.
\end{align}
If the function $v$ only depends on $x$, we will omit the subscribe $I$.
 It is easy to see that $v$ has bounded oscillation if and only if
 there exists $\delta_0 \in (0,1]$ such that
\begin{align} \label{osci1}
{\rm osc}_{I,\,  \delta_0}(v) =\sup_{x,y \in \R^N,\, |x-y| \le
\delta_0,\, t \in I} |v(t,x) - v(t,y)| < + \infty
\end{align}
(note that  ${\rm osc}_{I,\,  1}(v) \le \frac{1}{\delta_0} {{\rm osc}}_{I,\,
\delta_0}(v)$). The sum of a bounded function and  a uniformly
continuous one is an example of  function having bounded oscillation.
Functions with bounded oscillation have always at most  linear growth
in $x$ (uniformly in $t$). In particular,  $v$ has bounded
oscillation if and only if
 \be\label{lgro} |v(t,x) - v(t,y)| \leq
k_0+k_1 |x-y|, \qquad  x,y,\in \R^N\,,\,\,t\in I,
 \ee for some
constants $k_0$, $k_1\geq 0$.  Note that this inequality is not
restricted to $x$, $y$ such that $|x-y|\leq 1$. However, one can
easily show that \rife{osci} implies \rife{lgro} with  $k_0=k_1={\rm
osc}_I(v)$; and conversely, \rife{lgro} obviously implies
\rife{osci}. \vskip0.2em In the sequel, given two functions $\vfi$,
$v : [0,T ]
 \times \R^N \to
\R$, we will write that   $v=o_{\infty}(\varphi)$ in $\bar Q_T$
 if, for any $\epsilon>0$, there exists
$C_{\epsilon}>0$
  such that
 $|v(t,x)|$ $\le \epsilon \varphi (t,x)$,
 for $|x| > C_{\epsilon}$,  $t \in [0,T]$.
\vskip0.4em
Given two non-negative symmetric $N \times N $ (real)
 matrices
 $A,B$, we write $A \ge B$ if $ A \xi \cdot \xi $
  $\ge  B \xi \cdot \xi $, for any $\xi \in \R^N$ ($ \cdot
  $ denotes
  the inner product and  $|\cdot|$ the Euclidean norm in
   $\R^N$). Moreover, we always indicate with  $\| A \|$ the Hilbert-Schmidt
norm of a matrix $A$, defined as $\|A\|= \sqrt{{\rm tr} \left(A
A^*\right)}$ (here $A^*$ denotes the adjoint or transpose matrix of
$A$).  We also use $a \wedge b = \min(a,b)$, for $a,b \in \R$.

 \vskip1em Let us briefly recall the notion of
viscosity solution (in the standard continuous setting) for a
general parabolic fully nonlinear equation \be\label{fully}
\partial_t u + F(t, x, u,   D u, D^2 u)=0 \ee where $F(t,x,u,p,X)$
is a continuous function of $(t,x)\in Q_T$, $u\in \R$, $p\in \RN$,
$X\in {\mathcal S}_N$, being ${\mathcal S}_N$ the set of symmetric
$N\times N$ matrices.   We will always assume that $F$ is proper, i.e.,
\be\label{monou}
F(t,x,u,p,X)\geq F(t,x,v,p,Y),\quad   \forall u\geq v\,,\quad X \le Y
\ee
where $(t,x)\in Q_T, u,v\in \R, p\in \R^N, X, Y\in {\mathcal S}_N$.

By $USC(Q_T)$ and $LSC(Q_T)$ we denote the sets of upper
  semicontinuous, respectively lower semicontinuous, functions in
$Q_T$. Given a point $(t_0,x_0)\in Q_T$,  we denote by
$P^{2,+}_{Q_T} u(t_0,x_0)$ the set of ``generalized derivatives''
$(a,p,X)\in \R\times \RN\times {\mathcal S}_N$ such that
$$
u(t,x)\leq u(t_0,x_0)+ a(t-t_0)+ p\cdot(x-x_0)+ \frac12 X(x-x_0)\cdot(x-x_0)+ o(t-t_0)+ o(|x-x_0|^2)$$$$
\qquad \hbox{as $(t,x)\in Q_T$ and $t\to t_0$, $x\to x_0$.}
$$
The set $P^{2,-}_{Q_T} u(t_0,x_0)$ is the set of $(a,p,X)$
satisfying the opposite inequality. Note that
\begin{align} \label{meno}
P^{2,-}_{Q_T} u(t_0,x_0) = - P^{2,+}_{Q_T} (- u)(t_0,x_0).
\end{align}
Of course, $P^{2,+}_{Q_T} u(t_0,x_0)$ as well as $P^{2,-}_{Q_T}
u(t_0,x_0)$ are reduced to the unique triplet $(\partial_t
u(t_0,x_0), $ $  D u(t_0,x_0),$ $ D^2 u(t_0,x_0))$  whenever $u\in
C^{1,2}(Q_T)$, by simply applying Taylor's expansion. Moreover, for
any $\psi\in C^{1,2}(Q_T)$ and any $(t_0,x_0)\in Q_T$ such that
$u-\psi$ has a local maximum (respectively, a  local minimum) at
$(t_0,x_0)$ we have that $(\partial_t \psi(t_0,x_0),   D
\psi(t_0,x_0), D^{2}\psi(t_0,x_0))$ belongs to $P^{2,+}_{Q_T}
u(t_0,x_0)$ (respectively, to $P^{2,-}_{Q_T} u(t_0,x_0)$). Finally,
the set $\overline{P}^{2,+}_{Q_T} u(t_0,x_0)$ is defined as
containing limits of elements of $P^{2,+}_{Q_T} u(t_n,x_n)$ whenever
$(t_n, x_n) \in Q_T$,  $t_n\to t_0$, $x_n\to x_0$ and $u(t_n,x_n)\to
u(t_0,x_0)$ (and similarly for ${\overline P}^{2,-}_{Q_T}
u(t_0,x_0)$). We refer to the key-reference \cite{CIL} for more
details and comments.

\begin{definition}\label{visc}
A  function $u\in USC(Q_T)$ is a  viscosity subsolution
 of \rife{fully} if, for
 every $(t_0,x_0)\in Q_T$, we have
$$
a+ F(t_0,x_0, u(t_0,x_0), p, X)\leq 0, \quad  (a,p,X)\in
P^{2,+}_{Q_T} u(t_0,x_0)\,.
$$
A function $u\in LSC(Q_T)$ is a supersolution if, for every $(t_0,
x_0)\in Q_T$, we have
$$
a+ F(t_0,x_0, u(t_0,x_0), p, X)\geq 0,\quad  (a,p,X)\in
P^{2,-}_{Q_T} u(t_0, x_0)\,.
$$
A continuous function $u\in C(Q_T)$ is  a viscosity solution of
\rife{fully} if it is both a  subsolution and a  supersolution.
\end{definition}
By the above remarks, one can alternatively define  a viscosity
subsolution (respectively, supersolution) requiring that for any
$\psi\in C^{1,2}(Q_T)$ and any $(t_0,x_0)\in Q_T$ such that $u-\psi$
has a local maximum (respectively, a  local minimum)   at
$(t_0,x_0)$ we have
$$
\partial_t  \psi (t_0,x_0)+ F(t_0, x_0, u(t_0,x_0),   D \psi(t_0,x_0),
D^2\psi(t_0,x_0))\leq 0
$$
(respectively, $\partial_t  \psi (t_0,x_0)+ F(t_0, x_0, u(t_0,x_0),
  D \psi(t_0,x_0), D^2\psi(t_0,x_0))\geq 0$).
 Moreover, thanks to the continuity of $F$, the inequalities required
in the above definition remain true for elements of
$\overline{P}^{2,+}_{Q_T} u(t_0,x_0)$   (respectively, ${\overline
P}^{2,-}_{Q_T} u(t_0,x_0)$) in case of subsolutions (respectively,
supersolutions).

  \vskip0.5em We
remind that such a formulation is consistent (i.e., classical
solutions $u\in C^{1,2}(Q_T)$ are viscosity solutions) since  by \eqref{monou}  $F(t,x,u,p, X)$ is nonincreasing with respect to the matrix $X \in
{\mathcal S}_N$ (such  an assumption
 corresponds to a possibly degenerate ellipticity condition).

\vskip 1mm
 We work for simplicity of notation
 only on  $Q_T$. However   all the
results can be easily extended to the case in which $Q_T$ is
replaced by the more general domain $Q_{s,T}= [s,T] \times \R^N$,
with $s<T$.  Indeed  if $v \in C(Q_{s,T})$ is a viscosity solution
to  \rife{fully}  on $Q_{s,T}$ then introducing $u(t,x) = v(t+s,x)$
we get a viscosity solution on $Q_{T-s}$ for
$
\partial_t u $ $ + F(\cdot \, +  s, x, u,   D u, D^2 u)=0. $
 In addition, for example,  gradient estimates for $u$ like
 $ \|   D u(t, \cdot )\|_\infty \leq \frac{C_T} {\sqrt{t}}
 $ $\| u(0,
  \cdot)\|_{\infty}  $, $t \in (0,T-s)$
  (cf. Corollary \ref{fulbound} and
  Theorem \ref{pw-fully}) become
$
\|   D v(t, \cdot )\|_\infty \leq \frac{C_T} {\sqrt{t-s}}
  \| v(s,
  \cdot)\|_{\infty},$ $  t \in (s,T).$
All the regularity results in the paper concerning
  viscosity solutions $u$ on $Q_T$ can be transferred in an obvious
 way to viscosity solutions $v$ on $Q_{s,T}$ for the corresponding
  ``translated in time'' parabolic equations.
%}
%\end{remark}

 \smallskip Let us
 explicitly recall the following fundamental   result, which is a
special case of   \cite[Theorem 8.3]{CIL}.

\begin{theorem} \label{key} Let $u_i$, $i=1, \ldots, k,$
 be real continuous functions  on $Q_T$.
 Assume that, for any $i =1, \ldots, k$,  either $u_i$ is
 a subsolution
 or $-u_i$ is a supersolution  to the parabolic equation
\rife{fully}.
% Let $D_i$ be  open  sets in $\R^N$, $i=1, \ldots,
%k$, and
Let  $z (t, x_1, \ldots, x_k)$ be a $C^{1,2}$-function on
 $(0,T) \times (\R^N)^k $
  such that the mapping
$$
 u_1(t, x_1) +
   \ldots + u_k (t, x_k) -
 z (t, x_1, \ldots , x_k)
$$
has a local
  maximum in $(\hat t , \hat x_1, \ldots, \hat x_k) \in
 (0,T) \times (\R^N)^k$.

Then,
  for every $n>0$ there exist matrices $X_i
  \in { \mathcal S}_N $
and $a_i  \in \R$, such that
$$
a_1 + \ldots + a_k = \partial_t z(\hat t, \hat x_1,
 \ldots, \hat x_k)\,, \quad
 (a_i,D_{x_i} z(\hat t, \hat x_1,
 \ldots, \hat x_k   ), X_i)
\in {\overline P}^{2,+}_{Q_T} u_i(\hat t, \hat x_i),
$$
\begin{align}\label{viva}  -  (n+ c_N\| A\|)I \leq
\left(
                        \begin{array}{ccc}
                          X_1 & \ldots & 0 \\
                          \ldots & \ldots & \ldots \\
                          0 & \ldots & X_k \\
                        \end{array}
                      \right)
 \leq A+
 \frac1n \,A^2
\end{align}
where $A=D^2_x z(\hat t, \hat x_1,
 \ldots, \hat x_k)$ (and the constant $c_N$
  appears since we are using the Hilbert-Schmidt norm of $A$).
\end{theorem}
\begin{remark} \label{rit}
\rm  In the statement of Theorem \ref{key}, the matrices $X_i\in { \mathcal S}_N $
and the values $a_i  \in \R$  depend on $n$, but usually one can apply a compactness argument to get rid of $n$ and obtain $a_i$, $X_i$ independent of $n$.

To show this fact, we concentrate on the standard case (used in this
paper) when $k=2$, $x_1=x$, $x_2=y$ and $u_1=u(t,x)$, $u_2=-v(t,y)$,
where $u$, $v$ are  respectively  a sub and supersolution; we set
$X_1=X_n$, $X_2=-Y_n$, $a_1 = a_n$ and $a_2= -b_n$, which
are given by Theorem \ref{key}.  Observe first that inequality
\eqref{viva} implies that $X_n \le B$ and $Y_n \ge C$, for some $N
\times N$-symmetric matrices $B$ and $C$ possibly depending on
$(\hat t ,\hat x, \hat y)$ but not on $n$. Moreover, since $u$ is a
subsolution we have \be\label{subn} a_n + F(\hat t ,\hat x,
u(\hat t,\hat x) ,D_{x} z(\hat t, \hat x,\hat y ), X_n)\leq 0 \ee so
that  using that $X_n\leq B$ and \eqref{monou} implies that $a_n
$ is bounded from above independently of $n$. Similarly we deduce
that $b_n$ is bounded from below and since $a_n-b_n=
\partial_t z(\hat t, \hat x, \hat y)$ this means that both $a_n$
and $b_n$ are uniformly bounded with respect to $n$.

\noindent Now, suppose that we know that
\be\label{suppose}
\begin{array}{c}
F(\hat t,\hat x, u(\hat t,\hat x), D_{x} z(\hat t, \hat x, \hat y), X_n) - F(\hat t, \hat y, v(\hat t, \hat y),- D_{y} z(\hat t, \hat x, \hat y), Y_n)
\\
\m
\qquad \qquad
\geq  -\la {\rm tr}(X_n- Y_n) - g (\hat t, \hat x, \hat y)
\end{array}
\ee
for some finite quantity $g$, possibly depending on $(\hat t ,   \hat x, \hat y)$ but not on $n$. Then we deduce
\be\label{paris}
\partial_t z(\hat t, \hat x, \hat y)  -\la {\rm tr}(X_n- Y_n) - g (\hat t, \hat x, \hat y)\leq 0\,.
\ee
Using $X_n \le B$ and $Y_n \ge C$, we deduce from \rife{paris}  that $ {\rm tr}(X_n)$ is bounded from below  and ${\rm tr}(Y_n)$ is bounded from above. Both matrices are therefore uniformly bounded in ${\mathcal S}_N$. We have proved so far that $(X_n)$ and $(Y_n)$, respectively $a_n$ and $b_n$, are relatively compact  in ${\mathcal S}_N$, respectively in $\R$,   and we can pass to the limit as $n\to \infty$, at least for  a subsequence.
\vskip0.5em
We conclude that, if \rife{suppose} holds, then there exist
$a, b\in \R$ and $X, Y\in {\mathcal S}_N$ such that we have
$(a,D_{x} z(\hat t, \hat x,
\hat y   ), X)
\in {\overline P}^{2,+}_{Q_T} u(\hat t, \hat x)$, $(b,-D_{y} z(\hat t, \hat x,
\hat y   ), Y)
\in {\overline P}^{2,-}_{Q_T} v(\hat t, \hat y)$,    $a-b =
 \partial_t z (\hat t, \hat x, \hat y)$ and
$$
\begin{pmatrix}
 X & 0    \\
\noalign{\medskip} 0 & -Y
 \end{pmatrix}
 \leq     A.
$$
In particular, we stress that \rife{suppose}  holds whenever, for instance,   $F$ is a  linear  nondegenerate  parabolic operator, or more generally if  there exists $\lambda >0$ and some locally bounded function $H$ such that
\begin{align} \label{elo}
F( t,x, u, p, X)-F( t, x, u,p, B)
\ge  -\la \, \text{tr}(X- B) + H ( t,  x,u, p), \;\; X, B \in { \mathcal S}_N ,\;\; X \le B,
\end{align}
for every  $t \in [0,T],\; u\in \R\,,\, x,p \in \R^N$. In this case we may reason separately on the equations of $u$ and $v$; indeed, since $X_n\leq B$, inequality \rife{subn} together with  \rife{elo} yield a bound from below on tr$(X_n)$, since $a_n$ is bounded.
Similarly one gets a bound from above on tr$(Y_n)$ from the condition of supersolution and the fact that $b_n$ is bounded.
\end{remark}
We collect now some simple matrix inequalities
which play an important  role in obtaining our estimates.
\begin{proposition} \label{refe} Let ${\tilde X}, {\tilde Y}, A\in {\mathcal S}_N$ be such that
\be\label{ineqbase}
 \begin{pmatrix}
 {\tilde X} & 0    \\
\noalign{\medskip} 0 & {\tilde Y}
 \end{pmatrix}
 \leq    \begin{pmatrix}
 A & -A    \\
\noalign{\medskip} -A& A
\end{pmatrix}.
\ee
Then:

\noindent (i) For every $\sigma_1, \sigma_2\in {\mathcal S}_N$, we
have \be\label{rep1} {\rm tr}\left(\sigma_1^2\, {\tilde X} +
\sigma_2^2\,{\tilde Y}\right) \leq {\rm
tr}\left((\sigma_1-\sigma_2)^2\,A \right) \ee (ii) For every $P\in
{\mathcal S}_N$ such that $0 \le P \le I$, we have \be\label{rep2}
{\rm tr}\left({\tilde X} + {\tilde Y}\right) \leq 2t\, {\rm
tr}\left(P\,A \right), \qquad  t\in [0,2]\,. \ee
In particular, for
any matrix $C\in {\mathcal S}_N$: $A\leq C$ and for any $\tilde
\la>0$, we have \be\label{rep3} \tilde  \la {\rm tr}\left({\tilde X}
+ {\tilde Y}\right)+ {\rm tr}\left(\sigma_1^2\, {\tilde X} +
\sigma_2^2\,{\tilde Y}\right) \leq 4\tilde  \la\, {\rm tr}\left(P\,A
\right)\\ + \|\sigma_1-\sigma_2\|^2\, \sup\limits_{|e|=1} | Ce \cdot
e|. \ee
\end{proposition}
Let us point out that estimate \rife{rep1} is usual in the viscosity
solutions approach, it is  somehow implicitly  contained already in
\cite{Ishii89}. Estimate \rife{rep2} is the crucial point in the
method introduced by Ishii and Lions (\cite{IL}) to get estimates
when ellipticity holds. Here we give a different, possibly simpler,
proof of this estimate inspired by  the probabilistic coupling
method (see also Section \ref{prob-anal}).

\proof   Statement \rife{rep1} follows multiplying \rife{ineqbase}
by the matrix
$$
 \begin{pmatrix}
 \sigma_1^2 & \sigma_1\sigma_2    \\
\noalign{\medskip} \sigma_2\sigma_1   & \sigma_2^2
 \end{pmatrix}
$$
and taking traces.  Similarly, multiplying \rife{ineqbase} by the matrix
$
 \begin{pmatrix}
 I & I-t P    \\
\noalign{\medskip}  I-t P & I
 \end{pmatrix}
$
we get \rife{rep2}.  Note that the above matrix is nonnegative for every $t\in [0,2]$; to this purpose note that   $0 \le P \le I $  implies  $0 \le P^2 \le P$.
 Finally, if $A\leq C$ we have ${\rm tr}\left((\sigma_1-\sigma_2)^2\,A \right) \le
 {\rm tr}\left((\sigma_1-\sigma_2)^2\,C \right)$; using \rife{rep1} we get
 $$
 {\rm tr}\left(\sigma_1^2\, {\tilde X} + \sigma_2^2\,{\tilde Y}\right) \leq {\rm tr}\left((\sigma_1-\sigma_2)^2\,C \right) \leq  \|\sigma_1-\sigma_2\|^2\, \sup\limits_{|e|=1} | Ce \cdot e|\,.
 $$
Summing up with \rife{rep2} - multiplied by $\tilde  \la$ - we deduce \rife{rep3}.
\qed

\section{Linear equations}\label{lin}

\subsection{The main gradient estimate
 in terms of the oscillation of the solution}

In this section we  deal with  the linear equation
 \be\label{eq}
  \partial_t u -{\rm tr}\left(q(t,x)D^2u\right)-
b(t,x)\cdot   D u  =h(t,x)\qquad  \hbox{in} \; Q_T.
 \ee
 Throughout the whole Section \ref{lin}, we will always
 assume the following basic continuity hypothesis.

\begin{hypothesis} \label{hy1}
 { \em We suppose that $q(t,x)=(q_{ij}(t,x))$
 is a symmetric matrix depending continuously on $(t,x) \in   Q_T$,
    $b(t,x)$ is
a continuous vector field  defined on $   Q_T$ with values in
$\R^N$, and  $h :   Q_T \to \R$ is a  continuous function.
 \qed

%\vskip 1mm
We also suppose  that the matrix $q(t,x)$ satisfies the parabolicity
(or ellipticity) condition \be\label{q1} q(t,x)\xi\cdot\xi\geq
\la\,|\xi|^2, \;\; \,\; (t, x ) \in Q_T, \;\; \xi \in
\R^N,
 \ee
 for some $\lambda>0$
  and that
  the following condition (existence
of a Lyapunov function) holds: \be\label{lyap} \exists \,\, \vfi\in
C^{1,2}({ \bar Q_T})\,, \;  M\geq 0\, \,:\, \;
\begin{cases}
A_t( \vfi) \leq  M \vfi+ \partial_t \vfi, \,\;  (t,x) \in Q_T,
 \\
\vfi(t,x)\to +\infty \quad \hbox{as $|x|\to \infty$, uniformly in $t\in (0,T)$, }
\end{cases}
\ee
where we have  set \be\label{At} A_t (u) =
{\rm tr}\left(q(t,x)D^2u\right) + b(t,x)\cdot   D u. \ee }
 \end{hypothesis}

\begin{remark} \label{bene}
{\em
 There is no loss of generality if one requires that \rife{lyap} is satisfied with   $M=0$, since
we can always replace  $\vfi$ with $e^{M \,t}\vfi$.  Thus, \rife{lyap} is actually equivalent to
\be\label{lyap0}
\exists \,\, \vfi\in
C^{1,2}({ \bar Q_T})\, :\, \;\begin{cases}
A_t( \vfi) \leq  \partial_t \vfi, \,\;  (t,x) \in Q_T,
 \\
\vfi(t,x)\to +\infty \quad \hbox{as $|x|\to \infty$, uniformly in
$t\in [0,T]$.}
\end{cases}
\ee  Moreover, possibly adding a constant to $\varphi$, we may
suppose that $\varphi >0$.}
%Note  that \rife{lyap} implies that the associated %stochastic process does not explode.
\end{remark}

 \vskip 1mm
 Following \cite{PW}, for any $\lambda$ such that \rife{q1} is
 satisfied we  denote by
 $\sigma(t,x)$  the
symmetric $N \times N$ nonnegative matrix such that
\begin{align} \label{sig6}
\sigma^2(t,x)=q(t,x)-\lambda\, I,\;\;\; (t,x) \in Q_T.
 \end{align}
Note that $\sigma(t,x)$ actually depends on $\lambda$ but for
simplicity we drop in our notation this dependence.

\begin{theorem}\label{PW} Assume that
 \rife{q1},
 %\rife{hva},
  \rife{lyap}
 hold true, and,
  in addition, that there exists  a nonnegative
  $g\in C(0, 1)\cap L^1(0,1)$ such that
   \be\label{pw}
\begin{array}{c}
\dys \frac1{|x-y|}\Big( \|\sigma(t,x)-\sigma(t,y)\|^2+
(b(t,x)-b(t,y))\cdot(x-y) \Big)\leq g(|x-y|),
\\
\m
  \quad
\, x,y\,\in \R^N, \, 0<|x-y|\leq 1,\quad t \in (0,T)\,.
\end{array}
\ee
 Let  $u \in C(Q_T)$
  be a
viscosity solution of \rife{eq}
 and assume that $u$ and $h$ have bounded oscillation  in $Q_T $ (see
(\ref{osci})).  Then   $ u(t)$ is Lipschitz continuous and we have
\begin{align}\label{general}
& \|   D u(t)\|_\infty \leq  \frac {\hat c_1}{\sqrt{t \wedge 1}} \, \,
\omega(t,u)+ \, \, \hat c_2
\sqrt{t \wedge 1}\, \, \omega(t,h)\, ,\,\,\;\; t \in (0,T),
\end{align}
where $ \hat c_1 = \frac{1+ 2 \lambda}{\lambda} \, e^{\frac{1}{4
\lambda}\int_0^1 g(s)ds}, \; \hat c_2 = \frac{1}{4\lambda} \,
e^{\frac{1}{4 \lambda}\int_0^1 g(s)ds}$ and (according to
(\ref{osci}))
$$
\omega(t,u)=   {{\rm osc}}_{(\frac t2\, , \, T\wedge \frac{3}{2}t)}(u)
\,,\, \quad \omega(t,h)=   {{\rm osc}}_{(\frac t2\, , \, T\wedge
\frac{3}{2}t)}(h).$$
\end{theorem}

Before the proof we make some comments on hypothesis \rife{pw}
 and on related conditions.

\begin{remark}\label{bb} {\em
% Let us comment on condition \rife{pw}.
 If $b=0$ and if we assume that
 $s g(s) \to 0$ as $s \to 0^+$
 (which is reasonable, being $g\in L^1(0,1)$),
 then \rife{pw} implies that
  \begin{align}\label{sig1}
\sigma(t, \cdot) \;\;\; \text{is uniformly continuous on $\R^N$,
uniformly with respect to $t.$}
\end{align}
Note that this does not require $q(t,x)$ to be uniformly continuous
(e.g. take $N=1$ and  $q(x) = x^2 + 1 $).

In the general case that $b\neq 0$, assumption \rife{pw} does not
imply any growth restriction  on $\sigma(t,x)$, since the growth of
the drift term and of the diffusion matrix may compensate
adequately. As an example, consider
\begin{align}\label{ex1}
A_t= A = (1+ |x|^4) \triangle u - 4N |x|^2  x \cdot Du.
\end{align}
Here taking $\la=1$ and $\sigma= |x|^2 I$ we have that \rife{pw} holds trivially with $g(r) =r$,
and  for {\em every}  $x, y \in \R^N$ (not only for $x,y$ such that $|x-y|\le
  1$). Notice also that in  this case
  $\varphi(x) = 1+ |x|^2$
  is a Lyapunov function  satisfying \rife{lyap}.
}
\end{remark}

\begin{remark} {\em
If the matrix $q(t, x)= \Sigma(t,x)\Sigma(t,x)^*$ with
 $\Sigma(t,
\cdot) \in \R^N\times \R^k$ (as it happens if the operator $A_t$ is
associated to the stochastic differential equation  $dX_t =
b(t,X_t)dt + \Sigma(t,X_t)dB_t$, being $B_t$ a $k$-dimensional
Brownian motion), then, assuming that
 \be\label{Sig}
\begin{cases}
\la(t,x) I \leq q(t,x)\leq \Lambda(t,x) I & \\
\la(t,x)\geq \la >0\,, \qquad \frac{\Lambda(t,x)}{\la(t,y)} \leq
M\quad
\hbox{for any $x,y\in \R^N:|x-y|\leq 1$,}\;\; t \in (0,T),\\
\end{cases}
\ee
one can  check  assumption \rife{pw} directly on
$\Sigma(t,x)$.

Indeed, if we take $\sigma (t,x) =
\sqrt{q(t,x) - \frac{\lambda}{2} I}$,   we have $\sigma(t,x)\ge \sqrt{\frac{\la(t,x)}{2}} I$ and the same for $\sigma(t,y)$. Recall that
 if $A$  and $B$ are positive symmetric $N \times N$ matrices such that $A \ge \lambda' I$ and $B \ge \lambda' I$ for some
 $\lambda' >0$ then $\| \sqrt{A} - \sqrt{B} \|$ $
 \le \frac{1}{2 \sqrt{\lambda'}} \| A- B \|$ (see e.g. \cite[Lemma 3.3]{PW}).
  Hence
\begin{align}\label{frat}
 \| \sigma(t,x) - \sigma(t,y) \| \le  \frac{1}{\sqrt{2}\sqrt{
\la(t,x) \wedge \la(t,y) }} \,\| q(t,x) - q(t,y) \|, \;\;\; t \in
(0,T),\; x,y \in \R^N.
\end{align}
Since we also have
 \be\label{prova}
\| q(t,x) - q(t,y) \|
\le \sqrt{2N} \sqrt{\max\{\Lambda(t,x), \Lambda(t,y)\}}
\|\Sigma(t,x) - \Sigma(t,y) \|,
 \ee
we deduce that, under condition \rife{Sig},
$$
\| \sigma(t,x) - \sigma(t,y) \| \le
\frac{\sqrt{N}\sqrt{\max\{\Lambda(t,x),
\Lambda(t,y)\}}}{\sqrt{\min\{ \la(t,x), \la(t,y)\} }}\,
\|\Sigma(t,x) - \Sigma(t,y) \| \le \sqrt{M \, N} \,\| \Sigma(t,x) -
\Sigma(t,y) \|
$$
if $|x-y|\le 1$.
Therefore, in this case the regularity of
$\Sigma(t,x)$ can be used to verify \rife{pw}.
  Note that \rife{Sig}
allows $\la(t,x)$ and $\Lambda(t,x)$ to have (the same) polynomial
growth in $x$ of any order (uniformly in $t$).
}
\end{remark}
\vskip0.5em

A key point in the proof of Theorem \ref{PW}   is the use of  the
auxiliary function $f(r)$ which solves the ODE \be\label{ode1}
\begin{cases}
4\la\, f''(r)+ g(r)f'(r) = -1\qquad r\in (0,\de]\,,
\\
\m f(0)=0\,,\,\, f'(\de)=0\,,
\end{cases}
\ee
  for $\delta \in (0,1] $.
   Notice that $f$  is Lipschitz continuous if and only if  $g\in
L^1(0,1)$. More precisely, we have \be\label{effe} f(r)=
\frac1{4\la} \int_0^r e^{-\frac{G(\xi)}{4\la}}\int_\xi^{\de}
e^{\frac{G(\tau)}{4\la}}d\tau d\xi\,, \qquad \hbox{where $G(\xi)=
\int_0^\xi g(\tau)d\tau$}\,. \ee Observe that  $f\in C^2((0,\de])$,
$f$ is  increasing and  concave in $(0,\de)$. In particular, we have
$f(r)\leq f'(0)r$ for any   $r \in (0,\de)$, whereas \be\label{fde2}
 f'(0)\leq   \frac{\de}{4\lambda} \, e^{ \frac{1}{4 \lambda}
 \int_0^1 g(s)ds}\qquad \quad \forall \de\leq 1\,.
\ee Moreover, we estimate \be\label{fde0}  f(r) \geq \frac1{4\la}
\int_0^r (\de-\xi)\, d\xi \geq \frac\de{8\la} r\,, \ee hence
\be\label{fde1} f(\de)\geq \frac{\de^2}{8\la }.\ee
\vskip0.5em

\noindent {\it Proof of Theorem \ref{PW}.}\quad
 According to Remark \ref{bene}, we assume that \rife{lyap0}
holds true.
Let us fix $t_0\in (0,T)$ and consider $\de \in (0,1]$. We define
  the open set
$$
\Delta = \Delta (t_0, \de) = \Big \{ (t,x,y)\in (0,T)\times
\RN\times \RN\,:\, |x-y|<\de\,,\,\,
 \frac{t_0}2 < t < (T\wedge \frac{3}{2}t_0) \Big\}
$$
and the function
$$
 \Phi_{\epsilon}(t,x,y)= u (t,x)-u (t,y)
  -  K \psi(x-y) -
\vep \,  (\vfi(t,x)+\vfi(t,y))
  - C_0 (t-t_0)^2\,-\frac\vep{T-t},
$$
where $K$, $C_0$ will   be chosen later (depending also on $t_0$),
$\vfi$ satisfies \rife{lyap0} and  $\psi(x-y)=f(|x-y|)$, where $f\in
C^2((0,\de])$ is the solution of \rife{ode1}.

The goal is to prove that, if $K$ and $C_0$ are large enough
(independently of $\vep$), then
 \be\label{goal11}
\Phi_{\epsilon}(t,x,y)\leq 0,\quad
 (t,x,y)\in \Delta\,.
 \ee
Once  \rife{goal11} is proved, then choosing $t=t_0$ and
letting $\vep\to 0$ one gets
 \be\label{goal2} u(t_0,x)-u(t_0,y)\leq
K \psi(x-y)= K f(|x-y|), \quad  x\,,y\,:|x-y|<\de\,.
 \ee Reversing
the roles of $x$, $y$ and using the Lipschitz continuity of $f$ will
allow  us to conclude.

In order to   prove \rife{goal11},  we argue by contradiction,
assuming that \be\label{abs11}
\sup\limits_{\Delta}\Phi_{\epsilon}(t,x,y)>0\,. \ee In the
following, let us set (cf. \rife{osci1})
$$
\omega_{0,\de}(u) =
 {{\rm osc}}_{(\frac {t_0}{2}\, , \, T\wedge \frac{3}{2}t_0), \de}(u)=
\sup \Big \{|u(t,x)-u(t,y)|\,,\,\, |x-y|\leq \de\,,\, t\in
\big(\frac {t_0}{2}\, , \, T\wedge \frac{3}{2}t_0 \big) \Big\}\,.
 $$
Since we have, for $(t,x,y)\in \Delta$,
$$
 \Phi_{\epsilon}(t,x,y) \leq
 % 2\|u\|_{L^\infty((\frac{t_0}2, T)
 %\times \RN)}
\omega_{0,\de}(u)
  -  K f(|x-y|)  - \vep \,  (\vfi(t,x)+\vfi(t,y))- C_0(t-t_0)^2
  - \frac{\epsilon}{T-t},
$$
we deduce the following assertions.

 1. Thanks  to the fact that  $\vfi$ blows-up at infinity, we have,
for any $\epsilon>0,$
$$
\Phi_{\epsilon} \to -\infty \qquad \hbox{as $|x|\to\infty$ or
$|y|\to \infty$}
$$
hence $\Phi_{\epsilon}$ has a global maximum on $\overline \Delta$.
\vskip0.5em 2.
Choosing $C_0\geq  \frac{4 \omega_{0,\de}(u)}{t_0^2}$ we have
$\Phi_{\epsilon} \leq 0$ if $t=\frac{t_0}2$ or
 $t=\frac32 t_0$, and, since  $
\lim_{t \to T^-}\frac\vep{T-t} = \infty$, we conclude that  the
positive maximum cannot be attained  on $\{t=\frac {t_0}2 \}\cup
\{t=T\wedge \frac{3}{2}t_0\}$.

\vskip0.5em 3. Choosing $K \geq  \frac{\omega_{0,\de}(u)}
%2\|u\|_\infty }
{f(\de)}$ we have $\Phi_{\epsilon} \leq 0$ if $|x-y|=\de$.
\vskip0.5em 4. Since the maximum value is positive, it cannot be
reached when $x=y$.

\vskip0.5em Therefore,  \rife{abs11} implies that
 the global  maximum of $\Phi_{\epsilon}$
  cannot be attained when $x=y$
 nor  on $\partial\Delta$ if $C_0$ and $K$ satisfy the above
conditions; in particular, the value of $C_0$ is fixed by now as
$$
C_0=\frac{4 \omega_{0,\de}(u)}{t_0^2}.
 $$
We deduce that $ u (t, x)-u (t,y)-
 z(t,x,y)$  has a local maximum (possibly depending on $\epsilon$)
 at some $(\hat t,
\hat x,\hat y) \in \triangle$, where
\begin{align} \label{zeta}
z(t,x,y)= K\,\psi(x-y)+\vep\, (\vfi(t,x)+\vfi(t,y))+C_0(t-t_0)^2\,
+\frac\vep{T-t}.
 \end{align}
 Then we apply  Theorem \ref{key} with
 $u_1(t,x)=u(t,x)$, $u_2(t,y)=-u(t,y)$; in view of Remark \ref{rit},  we eventually get  that  there exist matrices $X$, $Y\in {
\mathcal S}_N$  and $a$, $b\in \R$
 such that
$$
a - b= \partial_t z(\hat t, \hat x,\hat y)\,, \quad
 (a,D_x z(\hat t, \hat x,\hat y), X)
\in {\overline P}^{2,+}_{Q_T}u(\hat t, \hat x)\,,
$$
$$
\quad (b, -
D_y z(\hat t, \hat x,\hat y), Y)\in {\overline P}^{2,-}_{Q_T}u(\hat
t, \hat y)
% = - {\overline P}^{2,+}_{Q_T}(-u)(\hat t, \hat y)
$$
and $ \begin{pmatrix}
 X & 0    \\ 0 & -Y
\end{pmatrix}
 \leq D^2 z (\hat t,\hat x,\hat y)$,  that is
 \be\label{matri1}
\begin{pmatrix}
 X - \vep\,
 D^2\vfi (\hat t,\hat x)    \\
\noalign{\medskip}
0 & -Y - \vep\,
 D^2\vfi (\hat t,\hat y)
\end{pmatrix} \le K
\begin{pmatrix}
 D^2\psi(\hat x-\hat y) & -D^2\psi(\hat x-\hat y)\\
\noalign{\medskip}
- D^2\psi(\hat x-\hat y)&  D^2\psi(\hat x-\hat y)
\end{pmatrix}.
\ee
Using that
  $u(t,x)$ is a subsolution and $u(t,y)$
  a supersolution,
we obtain the two inequalities
$$
\begin{array}{c} a-{\rm tr}\left( q(\hat t, \hat x)
X\right)- b(\hat t , \hat x)\cdot D_x z(\hat t,\hat x, \hat y) \leq h(\hat t, \hat x)
\\
\m b-{\rm tr}\left( q(\hat t , \hat y) Y\right) +  b(\hat t , \hat
y)\cdot D_y z(\hat t,\hat x, \hat y)\geq h(\hat t, \hat y).
\end{array}
$$
Subtracting and using $a-b=\partial_t  z
 (\hat t, \hat x, \hat y)$
we get
\be\label{z1}
\partial_t  z(\hat t,\hat x, \hat y)
 -{\rm tr}
\left( q(\hat  t, \hat x) X-q(\hat t,  \hat y) Y\right)
 \leq  b(\hat t ,\hat x)\cdot D_x
z(\hat t,\hat x, \hat y) +b(\hat t ,\hat y)\cdot D_y z(\hat t,\hat
x, \hat y) + \omega_{0,\de}(h)\ee
where
$$
\omega_{0,\de}(h) =
 {{\rm osc}}_{(\frac {t_0}{2}\, , \, T\wedge \frac{3}{2}t_0), \de}(h)
= \sup\{|h(t,x)-h(t,y)|\,,\,\, |x-y|\leq \de\,,\, t\in (\frac
{t_0}{2}\, , \, T\wedge \frac{3}{2}t_0)\}\,.
$$
Recalling
 the definition of $z$ we compute and find
$$
 \begin{array}{c}
\frac{\epsilon}{(T-\hat t)^2} +
 2C_0(\hat t-t_0)+ \vep\,
 (\partial_t\vfi(\hat t,\hat x)+
 \partial_t \vfi(\hat t,\hat y))
 \\
 \m
 -{\rm
tr}\left( q(\hat  t, \hat x) X-q(\hat t,\hat y)Y\right) \leq
 \left( b(\hat t ,\hat x)-b(\hat t ,\hat y)\right)\cdot
D\psi(\hat x-\hat y)
\\
\m \qquad + \vep \, \left( b(\hat t ,\hat x) \cdot D\vfi(\hat t,\hat
x)+b(\hat t ,\hat y)\cdot D\vfi(\hat t,\hat y)\right) + \omega_{0,\de}(h)\,,
\end{array}
$$
which yields
$$
 \begin{array}{c}
%\frac{\epsilon}{(T-\hat t)^2} +
 2C_0(\hat t-t_0) -{\rm
tr}\left( q(\hat  t, \hat x) \tilde X-q(\hat t,\hat y)\tilde Y\right) \leq
 \left( b(\hat t ,\hat x)-b(\hat t ,\hat y)\right)\cdot
D\psi(\hat x-\hat y)
\\
\m \qquad +\vep\,\left( A_{\hat t}(\vfi)(\hat t,\hat x)-\partial_t
\vfi(\hat t,\hat x) + A_{\hat t}(\vfi)(\hat t,\hat y)-\partial_t
\vfi(\hat t, \hat y)\right)+ \omega_{0,\de}(h)\,,
\end{array}
$$
where $\tilde X = X - \vep\,  D^2\vfi(\hat t,\hat x)$,
 $\tilde Y = -Y - \vep\,  D^2\vfi(\hat t,\hat x) $.
  Using \rife{lyap0}, we obtain
\be\label{equa}
 2C_0(\hat t-t_0) -{\rm
tr}\left( q(\hat  t, \hat x) \tilde X-q(\hat t,\hat y)\tilde Y\right) \leq
 \left( b(\hat t ,\hat x)-b(\hat t ,\hat y)\right)\cdot
D\psi(\hat x-\hat y) + \omega_{0,\de}(h)\,.
\ee
In the  next step we consider \eqref{matri1} and
apply Proposition \ref{refe}
with $\sigma_1 = \sigma (\hat t, \hat x)$, $\sigma_2 = \sigma (\hat t, \hat y)$, $\tilde \lambda = \lambda$,
 $A = K \,D^2\psi( \hat x-\hat y)$.
Since $\psi (x-y)=f(|x-y|)$,  we compute, for $x \not = y$,
 \be\label{hesspsi}
 D^2
\psi(x-y)=\frac{f'(|x-y|)} {|x-y|}\left(I-
\frac{x-y}{|x-y|}\otimes\frac{x-y}{|x-y|}\right) +
f''(|x-y|)\frac{x-y}{|x-y|}\otimes\frac{x-y}{|x-y|}.
\ee
and since $f$ is concave  we have
\be\label{A<C}
A= K\, D^2
\psi(\hat x- \hat y) \le C:= K\, \frac{f'(|\hat x-\hat y|)} {|\hat x-\hat
y|}\left(I-
\frac{\hat x-\hat y}{|\hat x-\hat y|}\otimes\frac{\hat x-\hat y}{|\hat x-\hat y|}\right)\,.
\ee
So, applying   \eqref{rep3} with
 \be \label{pp1}
 P = \frac{\hat x- \hat y}{|\hat x- \hat y|}\otimes\frac{\hat x- \hat y}{|\hat x- \hat y|},
 \ee
we obtain (recall that \rife{sig6} holds)
\be\label{pre-opt}
\begin{array}{c}
{\rm tr}\left( q(\hat t , \hat x) \tilde X -q( \hat  t,
 \hat y)\tilde Y\right)
 \leq 4\la K\, f''(|\hat x-\hat
y|)+ K \frac{f'(|\hat x-\hat y|)}{|\hat x-\hat y|}\, \|\sigma(\hat
t, \hat x)-\sigma(\hat t, \hat y)\|^2\,,
\end{array}
\ee
 since $P \left(I-
\frac{\hat x-\hat y}{|\hat x- \hat y|}\otimes\frac{\hat x-\hat y}{|\hat  x- \hat y|}\right)=0$.
 Combining this estimate with \eqref{equa} we find
 $$
 \begin{array}{c}
%\frac{\epsilon}{(T-\hat t)^2} +
2C_0(\hat t-t_0)
 \\ \leq K \left( 4\la\,f''(|\hat
x-\hat y| ) +\frac{f'(|\hat x-\hat y|)}{|\hat x-\hat y|}\left(
\|\sigma(\hat t, \hat x)-\sigma(\hat t,\hat y)\|^2+ (b(\hat t ,\hat
x)-b(\hat t ,\hat y))\cdot(\hat x-\hat y)\right)\right) +\omega_{0,\de}(h)
\end{array}
$$
and using   \rife{pw}
\be\label{finegiro}
\begin{array}{c} 2C_0(\hat t-t_0) \leq K \left(
4\la\,f''(|\hat x-\hat y| ) + f'(|\hat x-\hat y|)\, g(|\hat x-\hat
y|)\right) +\omega_{0,\de}(h).
\end{array}
\ee
Since $f$ solves \rife{ode1},
 using also that $\hat t > t_0/2$, we get
$$
 K\leq \omega_{0,\de}(h) +  2C_0(t_0- \hat t) <\omega_{0,\de}(h) +  C_0 {t_0}\,.
$$
Since we have chosen $C_0= \frac{4\omega_{0,\de}(u)}{t_0^2}$, we deduce
$$
K < \frac{4 \omega_{0,\de}(u)}{t_0} +\omega_{0,\de}(h).
$$
Therefore, if $K$ is larger than this bound we reach  a
contradiction and \rife{abs11} cannot hold true. We conclude that
\rife{goal11} is verified provided that
 $$
 K \ge  \max \left\{ \frac{4 \omega_{0,\de}(u) }{t_0}
 + \omega_{0,\de}(h),
  \, \frac{\omega_{0,\de}(u) }{f(\de)}\right\}\,.
  $$
  From \rife{goal11},
choosing $t=t_0$ and letting $\vep \to 0$,  we obtain \rife{goal2}
and reversing the roles of $x$ and $y$, we deduce
\be\label{preliouv}
\begin{array}{c}
|u(t_0,x)-u(t_0,y)|\leq  K\, f(|x-y|),   \qquad  |x-y|<\de
\\
\m
\m
\hbox{where $K =\frac{4 \omega_{0,\de}(u) }{t_0} + \omega_{0,\de}(h) +
  \, \frac{\omega_{0,\de}(u) }{f(\de)}$\,.}
  \end{array}
\ee
Using that  $f$ is concave we then obtain
\begin{align} \label{conc}
|u(t_0,x)-u(t_0,y)|\leq K\, f'(0)   |x-y|,  \qquad |x-y|<\de,
\end{align}
and on account of \rife{fde2}, \rife{fde1} and the choice of $K$  we
get
$$
 |u(t_0,x)-u(t_0,y)|\leq \hat c_2 \delta   \,  \Big(
\frac{4\omega_{0,\de}(u)}
 {{t_0\wedge 1}} + \omega_{0,\de}(h) + \frac{8\la \omega_{0,\de}(u)}
 { \delta^2} \Big)\,   |x-y|,
$$
for every $x$, $y$: $ |x-y|\leq \de$, with $\hat c_2 =  \frac{1}{4\lambda} \, e^{ \frac{1}{4 \lambda}\int_0^1 g(s)ds}$.  Choosing $\de=\sqrt{t_0\wedge
1}$, we deduce
$$
|u(t_0,x)-u(t_0,y)|\leq \hat c_2 \sqrt{t_0\wedge 1} \,  \Big(
{\omega(t_0,u)} \frac{4 + 8\la}
 {{t_0\wedge 1}}  + \omega(t_0,h) \Big)\,   |x-y|,
$$
for every $x$, $y$: $ |x-y|\leq \sqrt{t_0\wedge 1}$, and then for
every $x$, $y$ such that $|x-y|\leq 1$. The inequality then extends,
in a standard way,  to any $x$, $y \in \R^N$. This shows
\rife{general} and finishes the proof.\qed

 We point out  from the above proof
 the estimate
 on the Lipschitz constant of $u(t_0)$ which is
\be\label{lip-delta}
Lip(u(t_0))\leq  C_\la\, \delta   \,  \Big( \frac{\omega_{0,\de}(u)}
{t_0 } + \omega_{0,\de}(h) + \frac{\omega_{0,\de}(u)}
 { \delta^2} \Big),
\ee
where $C_\la$ only depends on $\la$
and $g$.   In particular, {\it any additional estimate on the
oscillation of $u$ and $h$ can provide further specifications  of
such estimate} (see Theorem \ref{grow} below).

\begin{remark}{ \rm  A global version of the previous result up to $t=0$ is
possible if $u_0$ is assumed to be Lipschitz continuous. In this
case one can take $C_0=0$  and modify suitably the previous proof.
One  obtains
$$
|u(t,x)-u(t,y)|\leq K |x-y|, \;\;\; t \in [0,T],
$$
where $K= C(\la, g)\, \cdot \, \max\{Lip(u_0), T {\rm osc}_{[0,T]}(h),
{\rm osc}_{[0,T]}(u)\}$.
Note that here the global Lipschitz continuity of $u$ is
established in terms of the (bounded) oscillation of $u$. Later on
(cf. Section 3.3), assuming  a further  condition (at infinity) on
the function $g$ appearing in \rife{pw}, we will able to prove
estimates on the oscillation of $u$ in order to get a complete
Lipschitz bound of $u(t)$ only depending on the data $u_0$ and $h$.}
\end{remark}

Theorem \ref{PW} also provides an analytical proof of other
regularity  results existing in the literature which were only
proved by probabilistic techniques. In particular  we have

\begin{proposition}  As in Theorem \ref{PW} assume that
 \rife{q1} holds and  let $u$ be a viscosity solution of
  \rife{eq} such that $u$ and $h$ have bounded oscillation.

Assume that
  $q_{ij}(t, \cdot), b_i(t, \cdot) \in C^1(\R^N)$, $t \in (0,T)$,
  and, setting $\sigma_0(t,x)= \sqrt{q(t,x)}$,
\begin{align}\label{cer}
2\| D_h \sigma_0(t,x)\|^2+ Db(t,x)h \cdot h
\leq  c
  {|h|^2}, \;\;\; x \in \R^N, \; t \in (0,T),\; h \in \R^N.
\end{align}
(for a smooth $N \times N$ matrix $a(t,x)$, we write
$D_h a(t,x)$ to denote the matrix $(D_h a_{ij}(t,x))$, where $D_h a_{ij}$ stands for the
directional derivative along the
direction $h$). Then the gradient estimate \eqref{general}
 holds for $u$ with $g(s)= c s$.
\end{proposition}
A time-independent  version of
   \rife{cer}, with some additional assumptions
of polynomial growth  on $b$ and $\sigma_0$,  is used in  \cite{Ce0}
and
 \cite[Chapter 1]{Ce}
 to prove gradient estimates for the diffusion semigroup,
 %yet
 with a probabilistic approach.
\begin{proof}
To see that \rife{cer} allows to apply our Theorem \ref{PW},
 we first
note that we can take $\varphi(x) = 1+ |x|^2$ as Lyapunov function.
Moreover, we set $\sigma_{\mu}(t,x)= \sqrt{q(t,x) - \mu I}$, $\mu =
\lambda/2$, and  we check  that  \rife{pw} holds for $\sigma_\mu$
with $g(r)= c r$.
To this purpose, since for every $x,h\in \R^N$, we have
$$
\| \sigma_{\mu}(t,x+h) - \sigma_{\mu}(t,x) \|^2
\le \sum_{i,j =1}^N \int_0^1 \langle D \sigma_{\mu}^{ij}
(t,x+ r h), h \rangle^2 dr\,,
% \leq  \| D_h \sigma_{\mu}(t,x)\|^2,
$$
$$
(b(t,x+h)-b(t,x))\cdot h  = \int_0^1 Db
(t,x+ r h) h  \cdot h \, dr\,,
$$
it is enough to show that
\begin{align}\label{diff5}
\sum_{i,j =1}^N  \langle D \sigma_{\mu}^{ij}
(t,x), h \rangle^2 =
 \| D_h \sigma_{\mu}(t,x)\|^2
\le 2 \| D_h \sigma_0(t,x)\|^2,
\end{align}
so that \rife{cer} will imply the validity of \rife{pw}.

To prove \rife{diff5},
 we denote $D_h \sigma_{\mu}(t,x)$ by
$\sigma_{\mu}'(t,x)$ and, similarly,
 $D_h \sigma_{0}(t,x)$ by
$\sigma_{0}'(t,x)$ and $D_h q_{}(t,x)$ by
$q'(t,x)$. The next proof is inspired by the one
of \cite[Theorem 2.1 in Section 3.2]{Fri}.
 Note that the following
computations hold even if $\mu$ is equal to 0.
Since
$\sigma_{\mu}(t,x)^2 = q(t,x) - \mu I$, differentiating
with respect to $h$, we find
\begin{align} \label{ce}
q'(t,x) = \sigma_{\mu}'(t,x) \sigma_{\mu}(t,x) +
\sigma_{\mu}(t,x) \sigma_{\mu}'(t,x).
\end{align}
Take an orthogonal matrix $\theta(t,x)$ (independent of
$\mu$) such that
$$
\theta(t,x)\sigma_{\mu}(t,x) \theta^*(t,x) =
\Lambda_{\mu}(t,x),
$$
where $\Lambda_{\mu}(t,x)$ is the diagonal matrix, having
$\sqrt{\lambda_i(t,x) - \mu}$, $i =1, \ldots, N$, on the
 diagonal ($\lambda_i(t,x)$ are the eigenvalues of $q(t,x)$).
 Now multiplying both sides of \rife{ce}
 on the left by $\theta(t,x)$ and on the right by
 $\theta(t,x)^{*}$ and  setting
 $
a(t,x)= $ $\theta_{}(t,x) q'(t,x) \theta_{}(t,x)^*,$ we get
$$
 a(t,x) = Y_{\mu}(t,x) \Lambda_{\mu}(t,x) +
  \Lambda_{\mu}(t,x) Y_{\mu}(t,x), \;\;\; \text{where}\;
   Y_{\mu}(t,x)= \theta_{}(t,x) \sigma_{\mu}'(t,x)
   \theta_{}(t,x)^*.
$$
It follows that $ Y_{\mu}(t,x)$ has components
$$
Y_{\mu}^{ij}(t,x) = \frac{a_{ij}(t,x)} {
\sqrt{\lambda_i(t,x) - \mu}
+ \sqrt{\lambda_j(t,x) - \mu} }
$$
so that, using that $\lambda_i(t,x) - \mu
 \ge \lambda_i(t,x)/2$, $i = 1, \ldots, N$,
\begin{align}\label{yrt}
Y_{\mu}^{ij}(t,x)^2 \le 2 \frac{a_{ij}^2(t,x)}
{( \sqrt{\lambda_i(t,x) }
+ \sqrt{\lambda_j(t,x)} )^2}= 2( Y_{0}^{ij}(t,x))^2.
\end{align}
Now note that  $  \sigma_{\mu}'(t,x)
= \theta_{}(t,x)^* Y_{\mu}(t,x)
   \theta_{}(t,x).$

   Moreover, since $\theta(t,x)$
   is orthogonal,  we have
   $ \|\sigma_{\mu}'(t,x) \|^2 $ $=
  \| Y_{\mu}(t,x) \|^2 $ even for $\mu =0$. It follows that
$$
 \|\sigma_{\mu}'(t,x) \|^2 =
   \sum_{i,j=1}^N Y_{\mu}^{ij}(t,x)^2 \le
   2 \sum_{i,j=1}^N Y_{0}^{ij}(t,x)^2
   =   2\| Y_{0}(t,x) \|^2 =  2\|\sigma_{0}'(t,x) \|^2
$$
which proves \rife{diff5}.
\end{proof}

 Finally, let us provide a similar result concerning the case of
H\"older regularity in which we relax the assumption
  $g \in C(0,1) \cap L^1(0,1)$ used in Theorem \ref{PW}.
We adopt the same notations as in the previous statement.

\begin{proposition}\label{hol-lin}
 Assume the same conditions of Theorem \ref{PW}, only replacing the
 hypothesis $g \in C(0,1) \cap L^1(0,1)$ with the following
 condition
\begin{align} \label{gg}
 g \in C(0,1) \;\;\; \mbox{and} \;\;\; \lim_{s \to 0^+} s g(s) =0.
\end{align}
Let  $u\in C(Q_T)$ be a viscosity solution of \rife{eq} as in Theorem \ref{PW}.
Then  $u(t)$ is $\alpha$--H\"older continuous
 for any $\alpha \in (0,1)$ and verifies
\begin{align} \label{ci1}
|u(t,x)-u(t,y)| \leq
   \frac {\tilde c} { \lambda \alpha (1 - \alpha)}
\Big( \frac{1}{({t \wedge 1})^{\alpha/2}} \, \omega(t,u) + ({t \wedge 1})^{1- \alpha/2}
\omega(t,h)\Big) |x-y|^{\alpha},
\end{align}
 for all $ x,y \in \R^N, \; |x-y|\le 1,$  $t\in (0,T) $,
where $\tilde c $ only depends on $\alpha$, $\lambda$ and the
modulus of continuity of $sg(s)$.

 Moreover, if we replace \rife{gg}
  with the condition $g  \in C(0,1)$ and
 $\limsup_{s \to 0^+} s g(s)<
4 \lambda$, then there exists $\alpha = \alpha (g, \lambda) \in
(0,1)$ such that
 \rife{ci1} holds.
\end{proposition}

\proof We proceed  as in the proof of Theorem \ref{PW} with   few
changes. First we fix $\de_0 \in (0,1)$ such that $sg(s)<2\la
(1-\alpha)$ if $s<\de_0$ and then we consider the set $\triangle
(t_0, \delta)$ with $\delta < \delta_0$.

We also choose from the beginning
 $f(s)= s^{\alpha}$, and we take  $K \ge \frac{\omega_{0,\de}(u)}{\delta^{\alpha}}$. Being $f$ increasing and
concave, we still arrive  at the inequality (see \rife{finegiro})
$$
\begin{array}{c} 2C_0(\hat t-t_0) \leq K \big(
4\la\,f''(|\hat x-\hat y| ) + f'(|\hat x-\hat y|)\, g(|\hat x-\hat
y|)\big) + \omega_{0,\de}(h)
\end{array}
$$
 which now becomes
$$
 2C_0(\hat t-t_0) \leq
\alpha \,K\, |\hat x-\hat y|^{\alpha-2}  \big( 4\la\,(\alpha-1) +
|\hat x-\hat y| g(|\hat x-\hat y|)\big) + \omega_{0,\de}(h).
$$
Since $|\hat x - \hat y| < \delta_0$, we get
$$
2C_0(\hat t-t_0) \leq 2\la\,\alpha(\alpha-1) \,K\, |\hat x-\hat
y|^{\alpha-2} + \omega_{0,\de}(h)
$$
and so $2\la\,\alpha(1- \alpha) \, K \delta^{\alpha -2} \leq
\omega_{0,\de}(h) +  2C_0(t_0- \hat t) <\omega_{0,\de}(h) +  C_0 {t_0}$. We get a
contradiction if
$$
 K \ge \frac{1}{2\la\,\alpha(1- \alpha) \delta^{\alpha -2}}
 \Big(\omega_{0,\de}(h) + \frac{4\omega_{0,\de}(u)}{t_0 }\Big).
$$
 Choosing again $\de =  \sqrt{t_0\wedge \de_0^2} $
  we get the desired estimate.

To prove the last assertion, we set $\gamma= \limsup_{s \to 0^+} s
g(s)$ and we proceed similarly, but first we choose $\alpha \in
(0,1)$ such that $4\la\,(\alpha-1) +\gamma <0$, which is possible
since $\ga<4\la$. \qed

We point out  that \rife{ci1} holds only for $x,y$ such that
$|x-y|\leq 1$. However,   replacing $|x-y|^{\alpha}$ with
$(|x-y|^{\alpha} + |x-y|)$, the estimate holds for every $x,y \in
\R^N$, $t \in (0,T)$.

\begin{remark}  {\em  We discuss here
 the previous result when $b=0 =h$.
 In such case assertion \rife{gg} is  equivalent to say that $\sigma$ is
uniformly continuous on $\R^N$ (uniformly in $t$). In general this
 is weaker than requiring that
  $q$
 is uniformly continuous on $\R^N$ (cf. Remark \ref{bb}).
 On the other hand, we recall that
  if $q$   is elliptic, {\em bounded, uniformly continuous  and   independent on time},
 then  gradient
 estimates hold  for $u(t, \cdot)$ (see \cite[Theorem 6]{St} whose
 proof uses the theory of analytic semigroups).
However, in our case we do not know if   gradient estimates hold
assuming only  \rife{gg} rather than $g\in L^1(0,1)$.

 Note that  condition \rife{gg}
 seems to be  almost
 sharp to get  H\"older continuity for any $\alpha
\in (0,1)$. Indeed, in cases when $sg(s)$ is only bounded, it is known  (see \cite[Section 4]{KS}
  for the more general situation with  $q_{ij}$ bounded and measurable)
that  any bounded regular solution $u(t,
\cdot)$ of \rife{eq} is $\alpha$-Holder continuous on $\R^N$ only  for {\it some}
$\alpha = \alpha (\lambda, N) \in (0,1)$, with
a universal H\"older constant  depending on $\| u\|_{\infty},
\lambda $ and $N$.

Finally note that the assertion $\limsup_{s \to 0^+} s g(s) < 4
\lambda$ holds in particular (with $b=h=0$) if  there exists $0< \lambda < \Lambda$ such
that
$$
\lambda |h|^2 \le  q(t,x) h \cdot h  \le \Lambda |h|^2,\;\; x, h
\in \R^N, \;\;\; t \in (0,T),
$$
and the following  Cordes type condition holds:
$\frac{\Lambda}{\lambda} < \frac{N+4}{N}.$  To this purpose take
$\mu $ such that $0< \mu < \lambda$ and $
 \frac{\Lambda}{\mu} < \frac{N +4}{N}$, and define a symmetric positive definite matrix $\sigma (t,x)$, such
that $\sigma (t,x)^2 + \mu I = q(t,x)$.
We have $ 0 \le \big( \sigma (t,x) - \sigma (t,y) \big)^2 \le
(\Lambda -
 \mu)I$
and so  $ \|  \sigma (t,x) - \sigma (t,y) \|^2
 = \text{tr} \big[( \sigma (t,x) - \sigma (t,y)  )^2 \big] \le
 N  (\Lambda - \mu)$.
 Note that we can take
  $g(r) = \frac{N  (\Lambda - \mu)}{r}$
  so that $\limsup_{s \to 0^+} s g(s)$
  $= N  (\Lambda - \mu) < 4
\lambda$ and
 the assertion follows.
}
\end{remark}

\subsection{The Cauchy problem with  bounded initial datum and  potential term}

We use Theorem \ref{PW}  and generalize \cite[Theorem 3.4]{PW}
concerning Lipschitz continuity of bounded solutions to the Cauchy
problem on $Q_T$. Due to the fact that  we are considering bounded
solutions, we can treat more general equations having also a
 \textit{ possibly unbounded
potential term $V$,} i.e.,
 \be\label{eq1}
  \partial_t u -{\rm tr}\left(q(t,x)D^2u\right)-
b(t,x)\cdot D u + V(t,x) u =h(t,x)\qquad  \hbox{in} \; \, Q_T. \ee
We always assume Hypothesis \ref{hy1} and that \textit {$V : Q_T
  \to \R$ is continuous and non-negative.}

We start with a quite standard   lemma.
\begin{lemma} \label{max} Let $h\in C(Q_T)\cap L^\infty(Q_T)$.
  Let $u\in C\left( \bar Q_T \right)$
  be a viscosity solution of \rife{eq1} such that
  $u_0 : = u(0, \cdot)$ is bounded on   $\R^N$ and moreover $u$ is $o_{\infty}(\varphi)$ in $\bar Q_T$,
   where $\vfi$ satisfies \rife{lyap}. Then
 $u$ is bounded on $[0,T] \times \R^N$ and
 \begin{equation} \label{max16}
 \| u(t)\|_{\infty} \le \|u_0 \|_{\infty} \, + \,
 t \| h\|_{T, \infty},\;\;\; t \in [0,T],
\end{equation}
where $\| h\|_{T, \infty}
 = \sup_{t\in (0,T)} \|h(t)\|_{\infty}$.
\end{lemma}
\begin{proof}
 Let us consider, for $\epsilon>0$,
$$
\phi_{\epsilon}(t,x) = u(t,x)-\vep \, \vfi(t,x)
  - \frac\vep{T-t} - \| u_0\|_{\infty} -
  t  \| h\|_{T,\infty},
 $$  where $\vfi$ is the
Lyapunov function which may be assumed to satisfy \rife{lyap0}. If
for any $\epsilon
>0$, we have $ \phi_{\epsilon}(t,x) \le 0$, then, letting $\epsilon
\to 0^+$,
  we deduce $u(t , x) \le \| u_0\|_{\infty} +
  t \| h\|_{T,\infty}$.
  Arguing by contradiction, suppose that, for some $\epsilon >0$,
   $
\max\limits_{[0,T)\times \RN} \, \phi_{\epsilon} (t,x) \, >0$.
 Let $(t_{\epsilon}, x_{\epsilon})$ be a point
   where   this maximum is attained. Note
   that $t_{\epsilon} \in (0,T)$.  By the
 definition of subsolution we have
$$
   \| h\|_{T,\infty} + \frac\vep{(T- t_{\epsilon})^2}
+ (\partial_t - A_{t_{\epsilon}} )
\big( \vep \, \vfi
   \big)(t_{\epsilon}, x_{\epsilon}) \le h(
   t_{\epsilon}, x_{\epsilon}) - V(t_{\epsilon}, x_{\epsilon})u(t_{\epsilon}, x_{\epsilon}),
$$
which  yields   a contradiction. Indeed   $(\partial_t -
A_{t_{\epsilon}} ) \varphi(t_{\epsilon}, x_{\epsilon}) \ge 0$ and
 $u(t_{\epsilon}, x_{\epsilon}) > 0$ imply
$$
\| h\|_{T,\infty} + \frac\vep{(T- t_{\epsilon})^2} \le
h(t_{\epsilon}, x_{\epsilon}).
$$
 Applying the same reasoning to $-u$, which
  is a subsolution to \rife{eq1} with $h$ replaced by $-h$,
   we conclude.
\end{proof}

In the next result we will assume that  the potential $V: Q_T \to \R_+$ satisfies, for
every $x, y \in \R^N$, $|x-y| \le 1$, $t \in (0,T)$, \be
\label{pot}
  |V(t,x) - V(t,y)| \le { k_0} +
  { k_1} |x-y| \cdot\, \max \big \{ V(t,x), V(t,y) \big \},
\ee
for some non-negative constants $k_0$ and $k_1$.
This condition allows us to consider
unbounded potentials and   operators like
 $\triangle u - (|x|^2 + |x|^{\alpha}) u$, $\alpha \in (0,1)$.

\begin{theorem}\label{PWc}
Assume that
 \rife{q1}, \rife{lyap}
 hold true, and that  there exists  a non-negative
  $g\in C(0, 1)$ such that $\int_0^1 g(s)ds<\infty$ and
 \rife{pw} holds.
 Assume that  $h \in  C( Q_T)\cap L^\infty(Q_T)$ and that
 $V \in C(Q_T)$ is non-negative and satisfies \rife{pot}.
  Let  $u\in C(\bar Q_T)$
  be a
viscosity solution of \rife{eq1} such that
  $u_0 : = u(0, \cdot)$ is bounded on $\R^N$ and
 moreover $u$ is $oo_{\infty}(\varphi)$ in $\bar Q_T$,
   where $\vfi$ satisfies \rife{lyap}.

 Then $ u(t) \in W^{1,\infty}(\RN)$,
 for every $t\in (0,T)
$, and there exists $C_0 = C_0(\lambda, g)$, such that
\be\label{grad}
\|   D u(t)\|_\infty  \leq   C_0\left\{  \left[\frac1{\sqrt{t \wedge 1}}+ \sqrt{t \wedge
1} \, k_0 + k_1\right]
\left( \|u_0\|_{\infty} +  (T \wedge \frac{3}{2}t) \| h\|_{T,
\infty} \right)+  \sqrt{t \wedge
1} \, \| h\|_{T, \infty}\right\}.
\ee
\end{theorem}
\vskip1em

Let us briefly compare  this
result  with  \cite[Theorem 3.4]{PW}, where the assumption
\rife{pw} was introduced and a similar gradient estimate was proved
but using a probabilistic approach.  In particular, the concept of
solution of \rife{eq}
 used there is the probabilistic one
 given through the expectation of  $u_0(X_t)$ where $X_t$ is the diffusion process associated to the
 %solution of the martingale problem for the
operator $A_t$ in \rife{At}  and starting from $x \in \R^N$ (see
Appendix A for more details). Similarly,  assumption \rife{lyap} is
replaced there by the condition that $X_t$ is not explosive (observe
that the existence of a
 Lyapunov function implies
non-explosion of the diffusion process,  see e.g. \cite[Chapter
10]{SV}). Up to such changes due to the different settings, Theorem
\ref{PWc} provides an analytical proof of \cite[Theorem 3.4]{PW}
with a few more generality, since in the latter paper the
coefficients were not supposed to be time-dependent, there was no
source term $h$ and the potential  $V$ was supposed to be either bounded
 or Lipschitz continuous, in which case a stronger  hypothesis on
 $g$ was required.

 We point out that  \rife{grad} implies
$$
\|   D u(t)\|_\infty \lesssim C\, \frac{\|u_0\|_\infty}{\sqrt{t}
}\qquad \hbox{as $t\to 0^+$}.
$$
On the other hand,  when $u_0=0$,
 \rife{grad} implies that,  for small $t$,
$$
\|   D u(t)\|_\infty \leq C \sqrt{t} \, \| h\|_{T, \infty}
$$
 which generalizes the
 classical gradient estimate for solutions to the Cauchy problem
involving the
 inhomogeneous heat equation with zero initial condition.

\begin{proof}[Proof of Theorem  \ref{PWc}.]
Using Lemma \ref{max}, we know
that $u$ is bounded and moreover
\be\label{stima-osc}
\omega(t,u):= {{\rm osc}}_{(\frac t2,\, T\wedge \frac{3}{2}t ) }(u)
 \le 2 \sup_{s \in (\frac t2,\, T\wedge \frac{3}{2}t )}\| u(s)\|_{\infty} \le
 2 \|u_0 \|_{\infty} \, + \,
 2(T\wedge \frac{3}{2}t ) \| h\|_{T, \infty}.
\ee In particular, if $V=0$ the conclusion is a direct consequence
of  Theorem \ref{PW}.

\smallskip In the case
$V \not =0$, we follow  the proof  of Theorem \ref{PW}
 (with the same
notations) until we obtain (see \rife{finegiro}) that
$$
\begin{array}{c} 2C_0(\hat t-t_0) \leq K \left(
4\la\,f''(|\hat x-\hat y| ) + f'(|\hat x-\hat y|)\, g(|\hat x-\hat
y|)\right) +\omega_{0,\de}(h)
\\
\m
 + \big( V(\hat t, \hat y)u(\hat t, \hat y) -
 V(\hat t, \hat x)u(\hat t, \hat x) \big),
\end{array}
$$
which implies, using \rife{ode1} and since $\omega_{0,\de}(h)\leq 2\| h\|_{T, \infty} $,
\be\label{finegiroV}
2C_0(\hat t-t_0) + K \leq  2\| h\|_{T, \infty}
 + \big( V(\hat t, \hat y)u(\hat t, \hat y) -
 V(\hat t, \hat x)u(\hat t, \hat x) \big)\,.
\ee
We concentrate on the last term and set
$$
J(\hat  t, \hat x, \hat y) =
 V(\hat t, \hat y)u(\hat t, \hat y) -
 V(\hat t, \hat x)u(\hat t, \hat x)  .
 $$
 Assume that $V(\hat t , \hat y) \ge V(\hat t , \hat x) $.
 Set\footnote{In the case that $V(\hat t , \hat x) \ge V(\hat t
, \hat y) $, we use the identity $J(\hat  t, \hat x, \hat y) =
  (V(\hat t, \hat y)-V(\hat t, \hat x))
   u(\hat t, \hat y) +
 (u(\hat t, \hat y) -
u(\hat t, \hat x))   V(\hat t, \hat x)$ and then we proceed in a
similar way. }
$$
J(\hat  t, \hat x, \hat y) =
  V(\hat t, \hat y)\big(u(\hat t, \hat y) -
  u(\hat t, \hat x)\big)
+
 u(\hat t, \hat x) \big(V(\hat t, \hat y)-
  V(\hat t, \hat x) \big).
  $$
  Since $u(\hat t, \hat y) -
u(\hat t, \hat x)< -  K f(|\hat x- \hat y|) $, which is a
consequence of \rife{abs11},  using  \rife{pot}, we get
 $$
  J(\hat  t, \hat x, \hat y)
  \le -K  V(\hat t , \hat y )  f(|\hat x- \hat y|)
  +  \| u(\hat t )\|_{ \infty} \left( k_0+  k_1
   |\hat x- \hat y| \, V(\hat t, \hat y)\right) .
 $$
On account of \rife{fde0}, we deduce
 $$
 J(\hat  t, \hat x, \hat y) \le -K\,\frac{\de}{8\la}  V(\hat t , \hat y )  |\hat x- \hat y|
  +  \| u(\hat t )\|_{ \infty} \left( k_0+  k_1  |\hat x- \hat y| \, V(\hat t, \hat y)\right) .
 $$
 Therefore, if $K\geq k_1\| u(\hat t )\|_{ \infty} \frac{8\la}\de$ we get
 $$
  J(\hat  t, \hat x, \hat y) \le k_0\, \| u(\hat t )\|_{ \infty}
 $$
 so that \rife{finegiroV} implies
$$
   2C_0(\hat t-t_0) + K \le   2\| h\|_{T, \infty}  +
    k_0  \, \| u(\hat t )\|_{ \infty}.
  $$
  Finally, we obtain here a contradiction if we choose
$$
K \ge  \max \left\{ \frac{4 \omega_{0,\de}(u) }{t_0} + 2\| h\|_{T, \infty}  + k_0  \, \| u(\hat t )\|_{ \infty}\,, k_1\| u(\hat t )\|_{ \infty} \frac{8\la}\de\,,
  \, \frac{\omega_{0,\de}(u) }{f(\de)}\right\}\,.
$$
Using \rife{fde1}   and setting $\de= \sqrt {t_0} \wedge 1$ we
find
$$
 |u(t_0,x)-u(t_0,y)| \leq
 K f(|x-y|) \leq K f'(0) |x-y|\leq c_2 K \delta |x-y|
 $$
 $$
 \leq C_0   \Big(
\frac{4\omega_{0,\de}(u)}
 {\sqrt{t_0\wedge1}} + \sqrt{t_0\wedge1} \left(\| h\|_{T, \infty}  + k_0 \| u(\hat t )\|_{ \infty}\right) + 8\la k_1\| u(\hat t )\|_{ \infty} +8\la  \frac{\omega_{0,\de}(u)}
 {\sqrt{t_0\wedge1}} \Big)\,   |x-y|.
$$
Recall that  \rife{stima-osc} implies  $
\omega_{0,\de}(u)\leq  2 \|u_0 \|_{\infty} \, + \,
 2(T\wedge \frac{3}{2}t_0 ) \| h\|_{T, \infty}$, hence we conclude
$$
\begin{array}{rl}
 |u(t_0,x)-u(t_0,y)| &\leq \left\{C(\la,g) \left(\frac1{\sqrt{t_0\wedge 1}}+ k_0 \sqrt{t_0\wedge1} + k_1\right)\left[ \|u_0 \|_{\infty} \, + (T\wedge \frac{3}{2}t_0 ) \| h\|_{T, \infty}\right]\right.
 \\
 \m
 \quad & \left.+ C(\la,g)\, \sqrt{t_0\wedge1}  \, \| h\|_{T, \infty}\right\}|x-y|\,.
 \end{array}
$$
\end{proof}

\begin{remark} \label{bert} {\em  It is not difficult to modify
the previous proof and obtain   gradient estimates, assuming instead
of \rife{pw}, the following more general assumption
$$
\begin{array}{c}
\frac1{|x-y|}\left( \|\sigma(t,x)-\sigma(t,y)\|^2+
(b(t,x)-b(t,y))\cdot(x-y) \right)\\ \\ \leq g(|x-y|) +
 k_3 |x-y|\max(V(t,x),
V(t,y)),
\end{array}$$
for $x,y \in \R^N$ such that $|x-y| \le \delta_1$, $t \in (0,T)$,
for some  $\delta_1, k_3 >0$. This condition is comparable with the
assumption $Db(x) h \cdot h \le (s_1 + s_2 V(x))|h|^2$, $x,h \in
\R^N$, for some $s_1, s_2 >0$ used in \cite{BF}. According to
\cite{BF} the last hypothesis together
 with $|DV(x)| \le c_0 + c_1 V(x)$, $x \in \R^N$, and $q_{ij} $
  Lipschitz continuous
  imply gradient estimates  \rife{grad} for
the autonomous non-degenerate Cauchy problem with $h=0$.}
\end{remark}

\begin{remark}
{\em  Theorem \ref{PWc} can be  also
  interpreted as
an a priori estimate on  classical bounded solutions  to the
Cauchy problem involving $A_t$. On this respect (see \cite[Section
2]{KLL} and \cite[Section 4]{MPW} for the autonomous case),
 if we assume Hypothesis \ref{hy1}, existence of a Lyapunov function
  $\varphi$
 and the fact that  coefficients $q_{ij}$, $b_i$  belong to
   $C^{\alpha/2, \alpha}_{loc}(\bar Q_T)$, for some $\alpha \in (0,1)$,
 then it is
  well-known that there exists a unique {\it bounded  classical
 solution}
 to the  Cauchy problem \rife{eq1} with $h=V=0$ and $u_0$ which
 is continuous
 and bounded on $\R^N$.
 }
\end{remark}

We finish the section with a variant of Theorem \ref{PWc}
establishing H\"older continuity of the solution. In this case we may assume that the potential $V$ satisfies, for
every $x, y \in \R^N$, $|x-y| \le 1$, $t \in (0,T)$,
\be
\label{potalpha}
  |V(t,x) - V(t,y)| \le { k_0} +
  { k_1} |x-y|^\alpha \cdot\, \max \big \{ V(t,x), V(t,y) \big \},
\ee
for some $\alpha\in (0,1)$.

\begin{proposition}\label{hol-lin1}
 Assume that
 \rife{q1}, \rife{lyap}
 hold true, and that  \rife{pw} holds for some $g$  satisfying \rife{gg}.
 Assume that $h \in  C(Q_T)\cap L^\infty(Q_T)$ and that
 $V \in C( Q_T)$ is nonnegative and satisfies \rife{potalpha}.
  If   $u\in C(\bar Q_T)$ is a viscosity solution of \rife{eq}
 as in Theorem \ref{PWc}, then there exists $C_0=C_0(\lambda,\alpha,g)$ such that
\begin{align} \label{hol1}
\nonumber |u(t,x)-u(t,y)| & \leq
  C_0 \left\{
\Big[ \frac{1}{({t \wedge 1})^{\alpha/2}} + ({t \wedge 1})^{1- \alpha/2}\, k_0+ k_1\Big]   \big( \|u_0\|_{\infty} +  (T \wedge \frac{3}{2}t) \| h\|_{T,
\infty} \big) \right.
\\
 & \left. + ({t \wedge 1})^{1- \alpha/2}  \| h\|_{T,
\infty}\right\} \, |x-y|^{\alpha},
\end{align}
 for all $ x,y \in \R^N, \; |x-y|\le 1,$  $t\in (0,T) $,
where $ C_0 $ only depends on $\alpha$, $\lambda$ and the
modulus of continuity of $sg(s)$.
\end{proposition}
\proof   Now we follow the proof of Proposition \ref{hol-lin}. We
fix $\de_0 \in (0,1)$ such that $sg(s)<2\la (1-\alpha)$ if $s<\de_0$
and then we consider the set $\triangle (t_0, \delta)$ with $\delta
< \delta_0$. We obtain
$$
2C_0(\hat t-t_0) \leq
\alpha \,K\, |\hat x-\hat y|^{\alpha-2}  \big( 4\la\,(\alpha-1) +
|\hat x-\hat y| g(|\hat x-\hat y|)\big) + \omega_{0,\de}(h)+ J(\hat t , \hat x, \hat y),
$$
which implies, since $|\hat x - \hat y| < \delta_0$,
$$
2C_0(\hat t-t_0) \leq
2\la\,\alpha(\alpha-1)\,K\, |\hat x-\hat y|^{\alpha-2}   + 2 \| h\|_{T,
\infty}+ J(\hat t , \hat x, \hat y)\,.
$$
We estimate last term as in
Theorem \ref{PWc} and using \rife{potalpha} we get
$$
 J(\hat  t, \hat x, \hat y) \le -K\,|\hat x-\hat y|^\alpha  V(\hat t , \hat y )
  +  \| u(\hat t )\|_{ \infty} \left( k_0+  k_1  |\hat x- \hat y|^\alpha \, V(\hat t, \hat y)\right) .
$$
In particular, if $K\geq k_1\,\| u(\hat t )\|_{ \infty} $ we conclude
$$
2C_0(\hat t-t_0) \leq
2\la\,\alpha(\alpha-1)\,K\, |\hat x-\hat y|^{\alpha-2}   + 2 \| h\|_{T,
\infty}+ k_0 \,  \| u(\hat t )\|_{ \infty} \,,
$$
and we obtain a contradiction if
$$
 K \ge \frac{1}{2\la\,\alpha(1- \alpha) \delta^{\alpha -2}}
 \Big(2 \| h\|_{T,
\infty}+ k_0 \,  \| u(\hat t )\|_{ \infty} + \frac{4\omega_{0,\de}(u)}{t_0 }\Big).
$$
Putting together all the conditions required on $K$, we need to choose $K$ such that
$$
K\geq \max\left\{ \frac{1}{2\la\,\alpha(1- \alpha) \delta^{\alpha -2}}
 \Big(2 \| h\|_{T,
\infty}+ k_0 \,  \| u(\hat t )\|_{ \infty} + \frac{4\omega_{0,\de}(u)}{t_0 }\Big)\,,\, k_1\,\| u(\hat t )\|_{ \infty} \,,\, \frac{\omega_{0,\de}(u)}{\de^\alpha} \right\}\,.
$$
Choosing $\de= \sqrt {t_0}\wedge \de_0$, and using the
 bounds on $u$ (see  \rife{max16}), we can choose
$$
\begin{array}{rl}
\dys
K= & C_0\left\{\left[\frac1{(\sqrt{t_0\wedge 1})^{\frac\alpha 2}} +(\sqrt{t_0\wedge 1})^{1-\frac\alpha 2} \, k_0+k_1\right] \left(  \|u_0\|_{\infty} \,+ \, (T \wedge \frac{3}{2}t_0) \| h\|_{T,
\infty}\right)\right.
\\
\m & \qquad \dys \left. + (\sqrt{t_0\wedge 1})^{1-\frac\alpha 2} \,
\| h\|_{T, \infty}\right\},
\end{array}
$$
for some  $C_0= C_0(\la, g,\alpha)$. In this way we have proved \rife{hol1}.
  \qed

%\vskip 1mm

 We finish the section by noting that all
  the previous results can be further  improved in dimension
$N=1$. Indeed in  this case we only need that $q$ is uniformly
positive and we do not need any ``control'' for $ \| \sigma (t,x) -
\sigma(t,y)\|. $  When $N=1$
Theorems \ref{PW}, \ref{PWc},
Propositions \ref{hol-lin} and
\ref{hol-lin1}
hold  replacing \rife{pw} with
 \be\label{pw2} (b(t,x)-b(t,y))\cdot(x-y) \leq g(|x-y|) |x-y|,
 \quad
\,x,y\,:\, 0<|x-y|\leq 1, \; t \in (0,T).
\ee
 As an example, we only formulate an analogue of Theorem \ref{PW}.
 Let us point out
 that in \cite{BF}
 it is considered   a one-dimensional operator ${A_t} = A$,
$
 {A}u= u'' + b_0 (x)u',
$  where $b_0 $ does not
 satisfy \rife{pw2} and
for which  uniform gradient estimates for the associated parabolic
Cauchy problem do not hold.

\begin{proposition}\label{PW2} Let $N=1$.
 Assume that
 \rife{q1},
  \rife{lyap}
 hold true, and,
  in addition, that there exists
  $g\in C(0, 1; \R_+)$ such that $\int_0^1 g(s)ds<\infty$
  and \rife{pw2} holds.
  Let  $u\in C(Q_T)$
  be a
viscosity solution of \rife{eq}. Moreover, suppose that $u$ and $h $
 have bounded oscillation (see (\ref{osci})).

 Then the function $ u(t)$  is Lipschitz continuous and estimate \rife{general} holds.
\end{proposition}
\proof We follow the  argument of the proof of Theorem \ref{PW} and
arrive at the  identity \rife{pre-opt}. Now note that  tr$\left( A(\hat t,
\hat x,\hat y)\right) -  A(\hat t,
 \hat
x,\hat y)\hat p\cdot \hat p =0.$
Using that
 $A(\hat t,\hat x,\hat y)\hat p\cdot
 \hat p \geq 4\la$, we obtain
$$
\begin{array}{c}
 (q(\hat t,\hat x) +q(\hat t,\hat y)- 2 c(\hat t,
\hat x,\hat y))D^2\psi(\hat x-\hat y) \leq
  4\la\, f''(|\hat x-\hat y|)\,
\end{array}
$$
and so (cf. \rife{equa})  $ 2C_0(\hat t-t_0) $ $\le  4\la K
\,f''(|\hat x-\hat y| )$ $ + \left( b(\hat t ,\hat x)-b(\hat t ,\hat
y)\right)\cdot D\psi(\hat x-\hat y) $ $ + \omega_{0,\de}(h) . $ Using
the assumption  \rife{pw2} we can continue as in the previous proof
and get the assertion. \qed

\subsection{Unbounded data: estimates on the oscillation of the solutions}

In this section we generalize the  gradient estimate \rife{grad}
for the operator \rife{eq} to the case in which the data
$u_0$, $h$ are not necessarily bounded but, more generally, have bounded
oscillation (see \rife{osci} and \rife{lgro}). In
particular, we cover the case of (possibly unbounded) Lipschitz and
H\"older continuous $u_0$.

 With respect to the assumptions in Theorem \ref{PWc},
  we need    to assume
  an additional  control at infinity on the function
  $g$ appearing in   \rife{pw}, namely that $g \in C(0, +\infty)$
 and $g(r)$ is $O(r)$ as
 $r\to +\infty$.
We stress that this condition is  always satisfied whenever $\sigma(t,x)$ has bounded oscillation and $b(t,x)$ verifies $(b(t,x)-b(t,y))\cdot(x-y)   \leq C |x-y|^2$, for any $x,y\,\in
\R^N, $ $ t \in [0,T]$ (notice that in this case $\varphi= 1+ |x|^2$ is a
 Lyapunov function).
 However, the operator  in \rife{ex1} also satisfies our condition at infinity for $g$.
 Roughly speaking, $g = O(r)$  as
 $r\to +\infty$ should be regarded as a control on
the superlinear growth of the coefficients which is not compensated by the interaction drift--diffusion.

In order to bound the oscillation of the solutions,
 we will need  the following growth type estimates
which may be of independent interest. This is why we give them  in a
rather sharp form (distinguishing  between
  the H\"older and the Lipschitz case, as in the previous
subsections, according to the behaviour of $g$  as
 $r \to 0^+$).

 The result explains that  if
 a viscosity solution $u$ of \rife{eq} is $o_{\infty}(\varphi)$ on $\bar Q_T$
 and, moreover,
 $u_0$, $h$ have bounded oscillation,  then $u(t)$  also has  bounded
 oscillation with a precise control which may imply, in particular,
  the  conservation of the Lipschitz and H\"older continuity
   of $u_0$.
 We mention that a related result for
 equations like $\partial_t u + F(\nabla u, \nabla^2 u )$
  is given in   \cite[Proposition 2.3]{GGIS}.

\begin{lemma}\label{osc-u} Assume that \rife{q1}  and
  \rife{lyap} hold true and that \rife{pw} holds in $Q_T$ with some
  non-negative  $g\in C(0, + \infty)$ such that $g(r)$ is $O( r) $ as $r \to + \infty$. In addition, assume that $u_0$, $h$ satisfy
\be\label{osc-uo} |u_0(x) - u_0(y)| \le k_0 + k_\alpha |x-y|^\alpha+
k_{1} |x-y|^{},\;\; x, y \in \R^N, \ee
\be\label{osc-h} |h(t,x) - h(t,y)| \le h_0 + h_\alpha |x-y|^\alpha+
h_{1} |x-y|^{},\;\; x, y \in \R^N, \, t\in (0,T). \ee
 for some $\alpha \in (0,1)$, with $k_0, k_{\alpha}, k_1,
  h_0, h_{\alpha}, h_1 \ge 0$. Let $u \in
C\left( \bar Q_T\right)$ be a viscosity solution of \rife{eq} such
that  $u$ is $o_{\infty}(\varphi)$ in $\bar Q_T$. Then we have:

\hh {\bf (i)} If $rg(r)\to 0$ as $r\to 0^+$, we have
$$
| u (t,x)-u (t,y)|
  \le     k_0 + (K+ L M t) |x-y|^\alpha+ (k_1+ M t(h_1+k_1)) |x-y|\, ,\;\; x, y \in
\R^N, \, t\in [0,T],
$$
 where
$K= \max(\frac{h_0}{2\alpha\la(1-\alpha)}, k_\alpha, k_1)$, $L=
\max(h_0,h_\alpha, h_1, k_\alpha, k_1)$ and $M=M(g,\la,\alpha,T)$.

\hh {\bf (ii)} If $g\in L^1(0,1)$ (and, in case $k_\alpha\neq0$, if
also $rg(r)\to 0$ as $r\to 0^+$), we have
$$
| u (t,x)-u (t,y)|
  \le
 k_0 + (k_\alpha+L \,k_\alpha\, t)
|x-y|^\alpha+ c_0(\max\{h_0,h_\alpha, k_1\} + ML t) |x-y|^{},\;\; x,
y \in \R^N, \, t\in [0,T],
$$
where $L= L(g , \la,\alpha,T )$, $M= \max(h_0, h_\alpha, h_1, k_1)$,
$c_0= c_0(g , \la,\alpha )$.
\end{lemma}

\begin{remark} \label{FGP}{\em
Let us comment on the technical form of estimates (i) and (ii)
above. Such form is meant to show that, if $k_0=0$,  then (i) or
(ii) imply the conservation of the H\"older, respectively Lipschitz,
continuity from $u_0$ to $u(t)$. To our knowledge, no similar results are available in
the literature unless the  coefficients $q_{ij}$ and $b_i$ are  assumed to
be Lipschitz continuous.

In particular, notice that if $k_0=k_1=0$ and if $h_1=0$, estimate
(i) implies a global H\"older estimate for $u(t)$, and in this case
$u(t)$ grows sublinearly at infinity. Similarly, (ii) shows that if
$k_0=k_\alpha=0$, then $u(t)$ is globally Lipschitz continuous. The
reader should keep in mind that here we are not dealing with the
regularizing effect, but with merely conservation of the growth and
continuity properties from $u_0$ to $u$. On the other hand, the
regularizing effect will follow coupling together the above
oscillation estimates with Theorem \ref{PW} and will be stated
later (see Theorem \ref{grow}).
}
\end{remark}

\begin{proof} According to Remark \ref{bene}, there is
no loss of generality in taking $M=0$ in \rife{lyap}. Let us
consider the set
\begin{align} \label{dl}
\Delta = \{ (t,x,y)\in (0,T')\times \RN\times \RN\, \},
\end{align}
where   $T' \in [0,T]$   will  be chosen later, and  the function
\begin{align} \label{phi}
 \Phi_{\epsilon}(t,x,y)= u (t,x)-u (t,y)
  -   f(t,|x-y|)  -
\vep \, (\vfi(t,x)+\vfi(t,y))
  - \frac\vep{T' -t}\,, \;\; \epsilon >0,
\end{align}
where
$$
f(t,r)= (k_0+at) + (\beta+ bt) r^\alpha+ (\gamma+ct) (r+\tilde
f(r)),
$$
with $a,b,c, \beta, \gamma\geq 0$ and $\tilde f(r) \in
C^2(0,\infty)$ is   a nondecreasing concave function with $\tilde
f(0)=0$, to be fixed later. As usual, we wish to prove that,
independently of $\vep$,
 \be\label{goal4}
\Phi_{\epsilon}(t,x,y)\leq 0, \quad
 (t,x,y)\in \Delta\,,
 \ee
and we argue by contradiction, assuming that
$\sup\limits_{\Delta}\Phi_{\epsilon}(t,x,y)>0$. Since
$\vfi$ blows-up at infinity and
$u$ is $o_{\infty}(\varphi)$,  we have
$$
\Phi_{\epsilon} \to -\infty \qquad \hbox{as $|x|\to\infty$ or
$|y|\to \infty$,}
$$
hence $\Phi_{\epsilon}$ has a global maximum at some point
 $(\hat t, \hat x,\hat y )\in \overline \Delta$
and clearly  $\hat x\neq \hat y$ since the maximum is positive, and
$\hat t< T'$. Assuming that
 \be \label{gro}
  \beta\geq k_\alpha\,,\quad \gamma\geq k_1,
 \ee
we deduce that the maximum cannot be reached at $t=0$ because,
thanks to \rife{osc-uo} and \rife{gro}, we have
$$
 u_0 (x)-u_0 (y)
  \le  f(0,|x-y|),\;\;\; x, y \in \R^N.
$$
 It follows that $(\hat t, \hat x,\hat y )$
 is also a  local maximum in which $\Phi_{\epsilon}$
 is positive. Henceforth, proceeding
 like in Theorem \ref{pw} (using \rife{lyap0}, and
that $f' \ge 0$ and $f''\le 0$ (we have set $\partial_x f = f'$
 and $\partial_{xx}^2 f = f''$), we
obtain
 $$
 \begin{array}{c}
  \frac\vep{(T'- \hat t)^2}+
(a+ b |\hat x-\hat y|^\alpha+c (|\hat x-\hat y|+ \tilde f(|\hat
x-\hat y|))
 \\
 \le
  \left(
4\la\,f''(\hat t,|\hat x-\hat y| ) + f'(\hat t,|\hat x-\hat y|)\,
g(|\hat x-\hat y|)\right)
  +
 |h(\hat t ,\hat x)-h(\hat t ,\hat y)|.
 \end{array}
$$
Setting $ \hat r = |\hat x - \hat y|$, using the expression of
$f(t,r)$ and \rife{osc-h}, we get
\begin{equation}\label{12}
\begin{array}{c}
 \frac\vep{(T'-\hat t)^2}+  (a+b\hat r^\alpha+ c({\hat r}+\tilde f({\hat r})))
   \\
 \m
 \leq    (\beta+b \hat t)\,\alpha\left[4\la(\alpha-1)
 {\hat r}^{\alpha-2}+
{\hat r}^{\alpha-1} g( {\hat r})\right] + (\gamma+ c \hat t) g(\hat r) +
(\gamma +c \hat t) \left[ 4\la \tilde f''(\hat r)+ g(\hat r)\tilde
f'(\hat r)\right]
\\
\m \quad + h_0+ h_\alpha {\hat r}^\alpha+ h_1 {\hat r}.
\end{array}
\ee We now distinguish two cases  according to the assumption on
$g(s)$ near $s=0$.

\hh \textbf{(i)} Assume  that $g(s)s\to 0$ as $s\to 0$. We take
here $\tilde f=0$ and we assume that $ b\geq c$, $\beta\geq \gamma
$; in particular we obtain
$$
\begin{array}{c}
  (\beta+b \hat t )\,
  \alpha\left[4\la(\alpha-1)  {\hat r}^{\alpha-2}+   {\hat r}^{\alpha-1} g( {\hat r})\right] + (\gamma+ c\hat t) g({\hat r})
  \\
  \m
  \leq (\beta+b \hat t)\left\{  \alpha\left[4\la(\alpha-1)  {\hat r}^{\alpha-2}+   {\hat r}^{\alpha-1} g( {\hat r})\right] +   g({\hat r}) \right\}.
  \end{array}
$$
Since $sg(s)\to 0$ as $s\to 0$,  there exists $r_0<1$ (only
depending on $g$, $\la$, $\alpha$) such that \be\label{r0}
 \alpha\left[4\la(\alpha-1)  s^{\alpha-2}+
s^{\alpha-1} g( s)\right] +  g(s)\leq 2\alpha\la(\alpha-1)
s^{\alpha-2}, \qquad  s\in (0,r_0)\,.
 \ee
Since $g(s)=O(s)$ as $s\to +\infty$, for $s\geq r_0$ there exists a
constant $L_0$ such that $g(s)\leq L_0s $, hence
 $$
(\beta+b \hat t)\,\alpha\left[4\la(\alpha-1)  s^{\alpha-2}+
s^{\alpha-1} g( s)\right] + (\gamma+ c\hat t) g(s)\leq  (\beta+b
\hat t)\,\alpha L_0 \, s^{\alpha}   +  (\gamma+ c\hat t) L_0 s, \;\;
 s \ge r_0.
 $$
Therefore, we conclude from \rife{12} (where $\tilde f=0$)   that
 $$
 \frac\vep{(T'-\hat t)^2}+   (a+b {\hat r}^\alpha+ c {\hat r})
 \leq     -2\alpha\la(1-\alpha) \beta\, {\hat r}^{\alpha-2}\chi_{r<r_0}+(\beta+b \hat t)\,\alpha L_0 \, {\hat r}^{\alpha}
   +  (\gamma+ c\hat t) L_0 {\hat r} + h_0+ h_\alpha {\hat r}^\alpha+ h_1 {\hat r}\,.
 $$
We choose $T'= 1/{2L_0}$,  so that $L_0 t\leq \frac12$ and we deduce
 \be\label{pre-abc}
 \frac\vep{(T'-\hat t)^2}+
 a+ (\frac12 b-L_0 \beta ) {\hat r}^\alpha+ (\frac12 c-L_0 \gamma ) {\hat r}
 \leq    -2\alpha
 \la(1-\alpha) \beta\, {\hat r}^{\alpha-2}\chi_{r<r_0}+ h_0+ h_\alpha {\hat r}^\alpha+ h_1 {\hat r}\,.
\ee
 Since $h_0 \leq h_0 {\hat r}^{\alpha-2}{ \chi_{(0,r_0)}(\hat r) } + \frac1{r_0^\alpha}\, h_0 {\hat r}^\alpha
 \chi_{r\geq r_0}(\hat r)$, we obtain
 $$
 \begin{array}{c}
  \frac\vep{(T'-\hat t)^2}+
   a+ \left(\frac12 b-L_0
   \beta\right) {\hat r}^\alpha+  (\frac12 c-L_0 \gamma ) {\hat r}
 \\
 \m
 \qquad  \qquad  \qquad
 \leq  (h_0 -2\alpha\la(1-\alpha)
\beta)  {\hat r}^{\alpha-2}\chi_{(0,r_0)}(\hat r)+ (\frac1{r_0^\alpha}h_0+
h_\alpha) {\hat r}^\alpha+ h_1 {\hat r}\,.
 \end{array}
 $$
Here we choose $a=0$, $\beta\geq \frac{h_0}{2\alpha\la(1-\alpha)}$,
$b\geq   2(L_0 \beta + 2\frac{\max\{h_0, h_\alpha\}}
 {r_0^{\alpha}})$,
$\gamma=k_1$, $c=2(L_0\gamma+ h_1)$ and we get a contradiction.

Recalling that $\beta\geq \gamma$, $b\geq c$
 and \rife{gro} were used before, we
have just proved that \be\label{esp} u (t,x)-u (t,y) \le
    k_0 + (\beta + b t) |x-y|^\alpha+(k_1+ ct) |x-y|  ,
\qquad \hbox{$t \in [0,T']$ and $x,y \in \R^N$,}
\ee where $\beta=
\max(\frac{h_0}{2\alpha\la(1-\alpha)}, k_\alpha, k_1)$, $c=2(L_0
k_1+ h_1)$ and
$$
b=   \max \left\{ 2\left(L_0 \beta + 2\frac{\max\{h_0, h_\alpha\}}{r_0^{\alpha}}\right) \,,\,
2(L_0 k_1 + h_1) \right\}\,.
$$
In particular we have $u (T',x)-u (T',y) \le
    k_0 + (\beta + b T')  |x-y|^{\alpha}+
    2(k_1+ h_1 T') |x-y|,$ $  x,y \in \R^N.
$ Since $T'$ only depends on $\la,\alpha$, $g$, and not on the data
$h$, $u_0$,  we can restart the same method on the interval $I= [T',
(2T' \wedge T)]$ (here it is enough to use $
f(t,r)= k_0 +  b_1 t \, r^\alpha+ c_1 t r,
$
for suitable constants $b_1$ and $c_1$).

 Iterating this argument we
obtain the assertion on the whole $[0,T]$.  In particular, we have
proved, for any $x,y \in \R^N$,
\be\label{fin-osc} u (t,x)-u (t,y) \le
    k_0 +  (K +L M t) |x-y|^\alpha+ (k_1+ M (k_1+h_1) t) |x-y|,
\ee
 where  $K= \beta =
  \max(\frac{h_0}{2\alpha\la(1-\alpha)}, k_\alpha, k_1)$,
$L=  \max(h_0,h_\alpha, h_1, k_\alpha, k_1 )$ and
$M=M(g,\la,\alpha,T)$.

\vskip 1mm Note that an alternative estimate can be obtained after
\rife{pre-abc} if we choose $a= h_0$, $b\geq 2(L_0\beta+  h_\alpha
)$ and $c= 2(L_0\gamma+ h_1)$. In this way we get the estimate
(recall that $\beta\geq \gamma=k_1$ and $b\geq c$)
 \be\label{opt-01}
|u (t,x)-u (t,y)| \le (k_0+h_0 t) + (\max\{k_\alpha,k_1\}+ b t)
|x-y|^\alpha+  (k_1+  c t) |x-y|,
 \ee for every $t \in [0,T']$ and
$x,y \in \R^N$, where   $b=  2(L_0\max\{k_\alpha,k_1\} +  \max\{
h_\alpha,h_1\})$ and $c=2(L_0 k_1+ h_1)$. Such a choice better
points out the continuity as $t\to 0$ but, in case $k_0=0$,
\rife{fin-osc} is preferable to show the continuity properties of
$u$.

\hh \textbf{(ii)} Assume now in addition that $g\in L^1(0,1)$. Since
$g(r)=O(r)$ as $r\to + \infty$, we may deduce that there exist $m>0$
and  a continuous function $\tilde g(r)$ such that
$$
g(r)\leq m r+ \tilde g(r)\qquad \forall r>0\,, \quad \tilde g\in
L^1(0,\infty)\,.
$$
Without loss of generality we can also assume that
$1\leq \tilde g(r)$ if $r\leq 1$, hence  taking $\gamma\geq \max\{h_0,h_\alpha\}$ implies
$$
h_0+h_\alpha r^\alpha+ h_1 r \leq (h_0+ h_\alpha) \tilde g(r)+
(h_0+h_\alpha+ h_1)r \leq 2\gamma \tilde g(r)+ (h_0+h_\alpha+h_1)r ,
\;\; r>0.
$$
Therefore, we deduce from \rife{12} \be \label{aggiunta}
 \begin{array}{c}
 \frac\vep{(T'-\hat t)^2}+  (a+b{\hat r}^\alpha+ c({\hat r}+\tilde f({\hat r})))
 \le      (\beta+b \hat t)\,\alpha\left[4\la(\alpha-1)  {\hat r}^{\alpha-2}+   {\hat r}^{\alpha-1} g( {\hat r})\right]
 \\
 \m
 \quad + (\gamma+ c\hat t)(1+\tilde f'({\hat r})) m {\hat r}  + (\gamma +c \hat t) \left[ 4\la \tilde f''({\hat r})+ \tilde g({\hat r})\tilde f'({\hat r})+ 3\tilde g({\hat r})\right]  + (h_0+h_\alpha+ h_1){\hat r}\,.
\end{array}
\ee We define now $\tilde f(r)$ as the solution of the ODE
$$
\begin{cases}
4\la\, \tilde f''(r)+ \tilde g(r)\tilde f'(r) + 3\tilde g(r)=0\qquad r\in (0,\infty)\,,
\\
\m
\tilde f(0)=0\,,\,\, \tilde f'(\infty)=0\,,
\end{cases}
$$
which is nothing but
$$
\tilde  f(r)=   3\int_0^r \left(e^{\frac{\tilde
G(\xi)}{4\la}}-1\right)d\xi\,, \qquad \tilde G(\xi)= \int_\xi^\infty
\tilde g(\tau)d\tau.
$$
Moreover,  since $rg(r)=o(1)$ as $r\to 0$,
 we  still use \footnote{ Note that if $k_\alpha=0$
we can take { $\beta =b=0$,} in which case we do not need anymore
the assumption $rg(r)=o(1)$ as $r\to 0^+$.} \rife{r0} and we deduce,
for some positive  constant $L_0$,
$$
 \frac\vep{(T'-\hat t)^2}+ (a+b{\hat r}^\alpha+ c({\hat r}+\tilde f({\hat r})))
 \leq (\beta+ b \hat t) L_0 \,   {\hat r}^\alpha + (\gamma+ c\hat t) L_0 {\hat r}+ (h_0+h_\alpha+h_1){\hat r} \,.
$$
As before choosing $T'=\frac1{2L_0} $ we get
$$
 \frac\vep{(T'-\hat t)^2}+ (a+\frac12 b {\hat r}^\alpha+
  \frac12 c{\hat r})
 \leq \beta  L_0 \,
 {\hat r}^\alpha +
\gamma L_0 {\hat r}+ (h_0 + h_{\alpha} + h_1){\hat r}
$$
and we conclude choosing $a=0$,
 $b=2L_0\beta $, $c= 2(L_0\gamma+ h_0+h_\alpha+ h_1)$.

Recalling the previous conditions on $\beta$ and $\gamma$ ($\beta = k_{\alpha}$, $\gamma  \ge k_1$), we have then obtained that for every $t \in [0,T']$ and $x,y \in \R^N$
$$
|u (t,x)-u (t,y)| \le
    k_0 + k_\alpha(1+ 2L_0 t) |x-y|^\alpha + c_0 (\max\{h_0,h_\alpha, k_1\} + ct)  |x-y| ) ,
$$
where   $c= 2(L_0 \max\{h_0,h_\alpha,k_1\}+ h_0+h_\alpha+h_1)$ and $c_0=1+\tilde f'(0)$ only depends on $g$, $\la$.

Next we iterate the estimate as in the previous case concluding that
$$
|u (t,x)-u (t,y)| \le k_0 +  (k_\alpha + L  k_\alpha  t)
|x-y|^\alpha+ c_0 (\max\{h_0, h_\alpha, k_1\}+ M L t) |x-y|
$$
 where  $L= L(g , \la,\alpha,T )$, $M=
\max(k_1, h_0, h_\alpha, h_1)$ and $c_0=c_0(g,\la,\alpha)$.
 \vskip 2 mm
Note that also in this case  an alternative estimate can be obtained
if we choose $a= h_0$, $b=2(L_0 k_\alpha+ h_\alpha)$ (as before) and
$c= 2(L_0 k_1+ h_1)$. In this way we obtain the different estimate
\be\label{opt-02} |u (t,x)-u (t,y)| \le
    (k_0+h_0 t) + (k_\alpha + b t) |x-y|^\alpha+c_0(k_1+ ct) |x-y|  ,
\ee for every $t \in [0,T']$ and $x,y \in \R^N$.
% where   $b=2(L_0
%k_\alpha+ h_\alpha)$   and $c= 2(L_0 k_1+ h_1)$.
This estimate may be   interesting for small time $t$.
  \end{proof}

\begin{remark}\label{tpicc} {\em
 We explicitly
  point out that we also proved the alternative estimates \rife{opt-01} and \rife{opt-02}, namely that
\be\label{smallt}
|u (t,x)-u (t,y)| \le
    (k_0+h_0 t) + (\beta+ b t) |x-y|^\alpha+ c_0 (k_1+  2(L_0 k_1+ h_1) t) |x-y|
\ee
for every $t \in [0,T']$ and $x,y \in \R^N$,
where
 $L_0$ only
depends on  $g$, $\la$, $\alpha$, $T$ and the constants $\beta$, $b$
are different according to case (i) or (ii); precisely, we have
$\beta= \max\{k_\alpha,k_1\}$ and   $b=  2(L_0\max\{k_\alpha,k_1\} +
\max\{ h_\alpha,h_1\})$ in case (i), while $\beta= k_\alpha$ and $b=
2(L_0 k_\alpha+ h_\alpha)$ in case (ii). }
\end{remark}

\begin{remark} \label{dege} {\em
{\it Under the stronger condition that there exist $M$, $L
\ge 0$
 such that  $g(r) \le M + Lr$ for every $r > 0$,
 the above proof   works even if $\la=0$, i.e., for a degenerate
problem, providing a bound for the oscillation of $u$.}

 More
precisely, if \rife{osc-uo} and \rife{osc-h} hold with $k_\alpha=0$
and $h_\alpha=0$, respectively, then $u$ has bounded oscillation
and \be\label{osc-dege} | u (t,x)-u (t,y)|
  \le  k_0 + C_0  t (h_0+ M(h_1+k_1))  +( k_1+ C_0 t (h_1+k_1)) |x-y|^{}\big),
  \ee
for any $ x, y \in
\R^N$, $ t\in [0,T]$,
where $ C_0= C_0 (T,  L) >0$.

The proof runs as above; taking $b=\beta=0$ and  $\tilde f=0$, and
$\gamma=k_1$,   \rife{12} reads as
$$
 \frac\vep{(T'-\hat t)^2}+  (a+  c{\hat r})
 \leq    (k_1+ c \hat t) g({\hat r})
 + h_0+   h_1 {\hat r}
$$
which implies, since  $g(r) \le M + Lr$, $r>0,$
$$
 \frac\vep{(T'-\hat t)^2}+  (a+  c{\hat r})
 \leq    (k_1+ c \hat t) (M+L{\hat r})
 + h_0 + h_1 {\hat r}\,.
$$
Choosing $T'=\frac1{2L}$ we get
$$
 \frac\vep{(T'-\hat t)^2}+  (a+  c {\hat r})
 \leq     M( k_1 +\frac c{2L}) + (k_1L+\frac12  c) {\hat r}
 + h_0 + h_1 {\hat r}
$$
and we conclude choosing $c=2(h_1+ k_1L)$ and  $a=M( k_1 +\frac c{2L})+ h_0$. We obtain then
$$
| u (t,x)-u (t,y)|
  \le   k_0 + at +  (k_{1}+(2h_1+2 k_1L)t)  |x-y|^{}\big),\;\; x, y \in
\R^N, \, t\in [0,T'],
$$
where $a=M( 2k_1 +\frac {h_1}{L})+ h_0$, and next we extend the
estimate in $[0,T]$. \vskip0.3em Notice that \rife{osc-dege} allows
one to estimate the oscillation of $u$ in terms of the oscillation
of $h$ and $u_0$. If in addition we have $M=0$ (i.e., $g(r) \le Lr$,
$r>0$),  we  also deduce that the Lipschitz continuity can be
preserved if $h_0=k_0=0$; {\it in that case, $u_0$ and $h$ globally
Lipschitz imply that $u(t)$ is globally Lipschitz.}
 }
\end{remark}

Let us observe the following corollary of the above result.

\begin{corollary} Assume that \rife{q1}  and
\rife{lyap} hold true and that \rife{pw} holds in $Q_T$ with some
non-negative  $g\in C(0, + \infty)$ such that $g(r)$ is $O( r) $ as
$r \to +\infty$ and one of the two conditions (i) or (ii) of Lemma
\ref{osc-u}  holds. Let $u \in C\left( \bar Q_T\right)$ be a
viscosity solution of \rife{eq} such that  $u$ is $o_{\infty}(\varphi)$ in
$\bar Q_T$. If data $u_0$, $h$ have bounded oscillation, then $u$
has bounded oscillation and
$$
{{\rm osc}}(u(t)) \leq  K\left\{  {{\rm osc}}(u_0)+
 t  \left[{{\rm osc}}(u_0)+ {{\rm osc}}_{(0,T)}(h) \right]\right\}\,,
$$
for every $t\in [0,T]$, where $K= K(\la, g, \alpha, T)$.
\end{corollary}

Combining Theorems \ref{PW} and Lemma \ref{osc-u}, assuming the
additional control of $g$ at infinity, we obtain a global estimate
for the case that   the initial datum and the source term are possibly unbounded.

\begin{theorem} \label{grow}
 Assume that  \rife{q1}, \rife{lyap}
 hold true, and,
  in addition, that there exists  a non-negative
  $g\in C(0, +\infty)$ such that $\int_0^1 g(s)ds<\infty$, $g(r)r\to 0$ as $r\to 0^+$,
   $g$ is O$(r)$ as $r \to +\infty$ and
 \rife{pw} holds for every $x,y\in \R^N$.

Assume that  $u_0$ and $h$ satisfy \rife{osc-uo} and  \rife{osc-h}.
Let $u\in C(\bar Q_T)$ be a viscosity solution of \rife{eq} which
     is $o_{\infty}(\varphi)$ in $\bar Q_T$,
   where $\vfi$ satisfies \rife{lyap}.
 Then $ u(t)$ is Lipschitz continuous,
 for every $t\in (0,T)
$, and, moreover, there exists $c= c(T, \lambda, g, \alpha)>0$ such that, for $t\in (0,T)$,
 \be \label{grad1} \|   D u(t)\|_\infty \leq  c\left\{ \frac
{k_0} {\sqrt{t \wedge 1}}\, + \frac { k_{\alpha}} {{(t \wedge
1)}^{1/2 - \alpha/2}} + k_1+ (\sqrt{t \wedge 1}) \, (h_0 +  h_{\alpha}(t \wedge 1)^{\alpha/2}+ h_1\sqrt{t \wedge 1})\right\}.
 \ee
\end{theorem}
The above estimate shows the
sharp dependence in $t$ of the Lipschitz constant  of $u(t)$ in
terms of Lipschitz and H\"older constants of $u_0 $ and $h$. For
example, if $h=0$ and $u_0$ is (possibly unbounded and)
$\alpha$-H\"older continuous on $\R^N$ we can set $k_{\alpha}
 = [u_0]_{\alpha}$ ($[u_0]_{\alpha}$ denotes the $\alpha$-H\"older
  constant or the $C^\alpha$-seminorm of $u_0$)
and obtain with $k_0 =k_1=0$
$$ \label{fh}\|   D u(t)\|_\infty \leq
\frac {c_T}  { t^{1/2 - \alpha/2}}\, [u_0]_{\alpha}, \;\;\; t \in
(0,T).
 $$
A similar  estimate under stronger assumptions on the coefficients
of the operator $A_t$ (i.e., $h=0$, $q \ge \lambda I$, $q_{ij}$ are
bounded, regular  and Lipschitz continuous, and  $b$ is possibly
unbounded and H\"older continuous) was recently  proved in
\cite[Lemma 4]{FGP}.

\begin{proof} Let $t_0 \in (0,T)$.
 By assumption, we estimate the oscillation of
$h$ as
$$
 \omega_{0,\de}(h): =
{{\rm osc}}_{(\frac {t_0}{2}\, , \, T\wedge \frac{3}{2}t_0), \de}(h) \leq
h_0+ h_\alpha \de^\alpha+ h_1 \de\,.
$$
By Remark \ref{tpicc}, thanks to the assumptions on $u_0$ and $g$,
we estimate the oscillation of $u$ as
 $$
 \omega_{0,\de}(u) :=
 {{\rm osc}}_{(\frac {t_0}{2}\, ,
\, T\wedge \frac{3}{2}t_0), \de}(u) \leq  C_T \left\{ (k_0+h_0 t_0)
+ (k_\alpha+ h_\alpha t_0) \de^\alpha+  (k_1+  h_1 t_0) \de\right\}
 $$
where  $C_T$ depends on $T, g,\la,\alpha$.  Using
\rife{lip-delta} with $\de=\sqrt{t_0 \wedge 1}$,
 %{lip-delta} after Theorem \ref{pw}
 we deduce
 $$
 Lip(u(t_0))\leq  C_\la\,    \,  \Big(
\frac{C_T \left\{ (k_0+h_0 t_0) + (k_\alpha+ h_\alpha t_0)
\de^\alpha+ (k_1+  h_1 t_0) \de\right\} }
 {\de} + \de (h_0+ h_\alpha \de^\alpha+ h_1 \de)  \Big)
 $$
 and
 %recalling that $\de=\sqrt{t_0\wedge 1}$
 we conclude \rife{grad1}.
\end{proof}

\subsection{Possible extensions to coefficients with more general growth}

We   give here  an  extension of Theorem \ref{PW} to operators with more general growth, like for example
 \be \label{opernew}
\tilde A_t u = (1+ |x|^2)\,  {\rm tr}(q(t,x)D^2u) + b(t,x) \cdot Du,
 \ee
 where  $q(t,x ) \ge \lambda_0 I$ for some $\lambda_0>0$, $q_{ij}(t, \cdot) \in
C^{\alpha}_b (\R^N)$ (uniformly in $t \in [0,T]$)  for some $\alpha
\in (0,1)$ and, say, $b(t, \cdot)$ is uniformly continuous
(uniformly in time). Generalizing \rife{sig6}, we first consider
 positive numbers $\la_{t, x}$ satisfying
 \be \label{fb}
 q(t,x)\xi\cdot\xi\geq  \la_{t, x}\,|\xi|^2, \;\; \,\; (t, x ) \in
Q_T, \;\; \xi \in \R^N,
 \ee
 and such that, for some $\lambda \in (0,1)$, we have   $\la_{t,x} \ge
\lambda$, $(t, x ) \in
Q_T$. We write $\sigma(t,x,y) = \sqrt{q(t,x) - (\lambda_{t,x}
  \wedge \lambda_{t,y})I \, } $
  to denote the symmetric $N \times N$ non-negative
matrix such that
\begin{align} \label{sig}
\sigma^2(t,x,y)=q(t,x)- (\la_{t , x} \wedge \la_{t,y})\, I,\;\;\; t
\in (0,T), \;\; x,y \in  \R^N.
 \end{align}
Similarly, $\sigma(t,y,x) = \sqrt{q(t,y) - (\lambda_{t,x}
  \wedge \lambda_{t,y})I .} $

\begin{theorem}\label{PWplus}
Let $A_t$ be given in \rife{At}. Assume the same hypotheses of
Theorem \ref{PW}  only replacing
 \rife{pw} with the following  weaker condition
 %(see also (\ref{sig}))
\begin{align} \label{prec}
 & \frac1{|x-y|}\Big(  \Big \|\sqrt{q(t,x) - (\lambda_{t,x}
 \wedge \lambda_{t,y}) I \, }
 - \sqrt{q(t,y) - (\lambda_{t,x}
 \wedge \lambda_{t,y}) I \, }
\Big \|^2
 \\ \nonumber  & + (b(t,x)-b(t,y))\cdot(x-y) \Big)
 \leq  {(\la_{t, x}\wedge
\la_{t, y})}
 \, g_0(|x-y|),
\end{align}
for all $x,y\,\in \R^N, \, 0<|x-y|\leq 1,\quad t \in (0,T),$
 for some  $g_0\in
C(0, 1; \R_+) \cap L^1(0,1)$ and some $\la_{t,x}$ satisfying \rife{fb}.
 Let  $u \in C(Q_T)$ be as in Theorem \ref{PW}.
  Then  $ u(t)$  is Lipschitz continuous and, setting
 $ G = e^{\frac14\int_0^1 g_0(s)ds},$
\begin{equation}\label{general13}
 \|   D u(t)\|_\infty \leq  \frac {G(1+ 2\la )}{\lambda \sqrt{t
\wedge 1}} \, \, {{\rm osc}}_{(\frac t2\, , \, T\wedge \frac{3}{2}t )}(u)
\,\, + \, \, \frac{G}{4 \lambda} \sqrt{t \wedge 1}\, \,
{{\rm osc}}_{(\frac t2\, , \, T\wedge \frac{3}{2}t)}(h) \,\, ,\;\; t \in
(0,T).
\end{equation}
\end{theorem}
\noindent Note that, by taking $\la_{t,x} = \lambda$  condition \rife{prec} becomes \rife{pw} with $g=  \lambda g_0$.
\begin{proof}
 According to Remark \ref{bene} there is  no loss of
generality in taking $M=0$ in \rife{lyap}.

As in the proof of Theorem \ref{PW}, we fix $t_0\in (0,T)$ and
consider $\de \in (0,1]$ and
  the open set $\Delta = \Delta (t_0, \de)$ with the same notation.
 Also the function $\Phi_{\epsilon}(t,x,y)$
 is defined as in the previous proof.
Arguing by contradiction, we assume that $
\sup\limits_{\Delta}\Phi_{\epsilon}(t,x,y)>0 $
  and  find that $ u (t, x)-u (t,y)-
 z(t,x,y)$  has a local maximum
  at some $(\hat t,
\hat x,\hat y) \in \triangle$.  Applying  Theorem \ref{key} (see also Remark \ref{rit})
 we arrive again at formula  \eqref{equa}
  with matrices $X$, $Y\in { \mathcal S}_N$
 which satisfy the inequality \rife{matri1}.
 Then we apply Proposition \ref{refe}, this time  with
$$
\sigma_1 = \sigma(\hat  t,\hat x,\hat y) =\sqrt{q(\hat t,\hat x) - (\lambda_{\hat t,\hat x}
 \wedge \lambda_{\hat t,\hat y}) I \, }, \;\;\; \sigma_2 = \sigma(\hat t,\hat y,\hat  x) = \sqrt{q(\hat t,\hat y) - (\lambda_{\hat t,\hat x}
 \wedge \lambda_{\hat t,\hat y}) I \, }
 $$
 and $A = K D^2\psi( \hat x-\hat y)$, $\tilde \lambda =  (\lambda_{\hat t,\hat x}
 \wedge \lambda_{\hat t,\hat y})  $.  On account of \rife{A<C}, and being
 $P$ as in \eqref{pp1}, we obtain from  \eqref{rep3}:
$$
 \begin{array}{c}
{\rm tr}\left( q(\hat t , \hat x) \tilde X -q( \hat  t,
 \hat y)\tilde Y\right)\\ \leq
 4(\lambda_{\hat t,\hat x}
 \wedge \lambda_{\hat t,\hat y}) K\, f''(|\hat x-\hat
y|)+
K \frac{f'(|\hat x-\hat y|)}{|\hat x-\hat y|}\, \|\sigma(\hat
t, \hat x, \hat y)-\sigma(\hat t, \hat y, \hat x)\|^2 ,
 \end{array}
$$
where $\tilde X = X - \vep\,  D^2\vfi(\hat t,\hat x)$,
 $\tilde Y = -Y - \vep\,  D^2\vfi(\hat t,\hat x) $.
Combining this estimate with   \eqref{equa},  we deduce
 $$
 \begin{array}{c}
2C_0(\hat t-t_0)
  \leq K 4(\lambda_{\hat t,\hat x}
 \wedge \lambda_{\hat t,\hat y})\,f''(|\hat
x-\hat y| )
\\
\m + K \frac{f'(|\hat x-\hat y|)}{|\hat x-\hat y|}\left(
\|\sigma(\hat t, \hat x, \hat y)-\sigma(\hat t,\hat y, \hat x)\|^2+ (b(\hat t ,\hat
x)-b(\hat t ,\hat y))\cdot(\hat x-\hat y)\right)   +\omega_{0,\de}(h)
\end{array}
$$
and, using   \rife{prec},
$$
\begin{array}{c}
2C_0(\hat t-t_0) \leq    (\la_{\hat t, \hat x} \wedge
\la_{\hat t, \hat y}) K\big(4 f''(|\hat x-\hat y| ) +  f'(|\hat x-\hat y|)\, g_0(|\hat
x-\hat y|)\big) + \omega_{0,\de}(h) .
\end{array} $$
Choose now $f$ as the  solution  of \rife{ode1}
 with $\lambda=1$ and $g=g_0$.
We get
$$
  (\la_{\hat t, \hat x} \wedge \la_{\hat t, \hat y}) \, K\leq
\omega_{0,\de}(h) + 2C_0(t_0- \hat t) < \omega_{0,\de}(h) + C_0 {t_0},
$$
and since we have chosen $C_0= \frac{4\omega_{0,
\delta}(u)}{t_0^2}$, we deduce
$$
K < \frac{ 4\omega_{0, \delta}(u)}{ (\la_{\hat t, \hat x} \wedge
\la_{\hat t, \hat y})\, t_0} + \frac{\omega_{0,\de}(h) }{
 (\la_{\hat t, \hat x} \wedge \la_{\hat t, \hat y})} \le \frac{4
\omega_{0, \delta}(u)}{ \lambda \, t_0} + \frac{\omega_{0,\de}(h)
}{\lambda}.
$$
Therefore, if $K$ is bigger than this bound we reach the desired
contradiction, and we conclude as in the proof of Theorem \ref{PW} choosing
\be\label{K1}
 K =  \frac{1}{\lambda}
  \left( \frac{4 \omega_{0, \delta}(u) }{t_0}
 +  \omega_{0,\de}(h) +
   \frac{\la \omega_{0, \delta}(u) }{f(\de)}\right)\,.
  \ee
\end{proof}
Next we state a sufficient condition in terms of $q_{ij}$ which implies   \eqref{prec}. Note that the operator in \rife{opernew} satisfies  \rife{pws}.
\begin{corollary} Assume the same hypotheses of
Theorem \ref{PW}  only replacing
 \rife{pw} with the following   condition
\be \label{pws}
\begin{array}{c}
  \frac1{|x-y|}\left(  \frac{1}{ { 2(\la_{t,x} \wedge \la_{t,y}) }}
  \, \|q(t,x)-q(t,y)\|^2+
(b(t,x)-b(t,y))\cdot(x-y) \right)\\ \leq    {(\la_{t, x}\wedge
\la_{t, y})} \,\cdot \, g_0(|x-y|),
\end{array}
\ee
for all $x,y\in \R^N$, $0<|x-y|\leq 1$,  $t \in (0,T)$, for some  $g_0\in
C(0, 1; \R_+) \cap L^1(0,1)$ and some $\la_{t,x}$ satisfying \rife{fb}. Then the same conclusion of Theorem \ref{PWplus} holds.
\end{corollary}
\begin{proof}
We only need to show that  \rife{pws} implies \rife{prec}; indeed,
  we notice that $ \sqrt{q(t,x)-
 \frac{\lambda_{t,x} \wedge \lambda_{t,y}}{2} I}$
 $\ge \sqrt{\frac{\la_{t,x}}{2}} I$, and
  $ \sqrt{q(t,y)-
 \frac{\lambda_{t,x} \wedge \lambda_{t,y}}{2} I}$
 $\ge \sqrt{\frac{\la_{t,y}}{2}} I$. Then, using \rife{frat}, we have
$$
\Big \|   \sqrt{q(t,x)-
 \frac{\lambda_{t,x} \wedge \lambda_{t,y}}{2} I}
 -  \sqrt{q(t,y)-
 \frac{\lambda_{t,x} \wedge \lambda_{t,y}}{2} I}
  \Big\| \le  \frac{1}{\sqrt{2}\sqrt{
\la_{t,x} \wedge \la_{t,y} }} \,\| q(t,x) - q(t,y) \|,
$$
and we obtain \rife{prec} with $\frac{\lambda_{t,x}}2$ and
$\frac{\lambda_{t,y}}2$.
 \end{proof}
Proceeding as in the proof of Theorem \ref{PWplus}, all the results
in Sections 3.1 and 3.2 concerning Lipschitz and H\"older regularity
can be extended replacing condition \rife{pw} with the more general
hypothesis \rife{prec} (clearly, to generalize Proposition
\ref{hol-lin} and obtain \rife{ci1} one has  to require that $s
g_0(s) \to 0$ as $s\to 0^+$). Below we only state a generalization
of Theorem \ref{PWc} concerning regularity of bounded solutions to
Cauchy problems.

On the other hand, we are not able to generalize the results in
Section 3.3 replacing \rife{pw} with  \rife{prec} (in particular, we
cannot extend Lemma \ref{osc-u}).

\begin{corollary}\label{PWc1}
Assume  the same assumptions of Theorem \ref{PWc}, only replacing
\rife{pw} with the more general \rife{prec}.
  Let  $u\in C(\bar Q_T)$
  be a
viscosity solution of the Cauchy problem \rife{eq1} such that
  $u_0 : = u(0, \cdot)$ is bounded on $\R^N$ and
 moreover $u$ is $o_{\infty}(\varphi)$ in $\bar Q_T$,
   where $\vfi$ satisfies \rife{lyap}.

 Then $ u(t)  \in W^{1,\infty}(\RN)$,
 for every $t\in (0,T)
$, and there exists $C_0 = C_0(\lambda, g)$, such that
\be\label{grad3} \|   D u(t)\|_\infty  \leq   C_0\left\{
\left[\frac1{\sqrt{t \wedge 1}}+ \sqrt{t \wedge 1} \, k_0 +
k_1\right] \left( \|u_0\|_{\infty} +  (T \wedge \frac{3}{2}t) \|
h\|_{T, \infty} \right)+  \sqrt{t \wedge 1} \, \| h\|_{T,
\infty}\right\}. \ee
\end{corollary}

 \section{Nonlinear equations}\label{nonlin}

 Here we consider  the fully nonlinear equation
 \be\label{fnl} \partial_t u+
F(t,x,  D u,D^2 u)= h(t,x) \;\; \text{in} \; Q_T,
 \ee
 where {\it $F$ is a real
 continuous function on $ Q_T \times \R^N \times
 {\mathcal S}_N$ and $h$ is a continuous function
 on $ Q_T.$}

\subsection{Regularizing effect in terms of the
oscillation of the solution}

We start with the global  regularizing effect which is proved in
terms of the oscillation of a solution. To this purpose, we
introduce  a set of assumptions which generalize those of the linear
case and cover interesting examples as  the sup/inf of linear
operators, as well as the case of viscous Hamilton-Jacobi equations
or, more generally, nonlinear first order terms.

\begin{hypothesis} \label{pw-nl} There exist
$\lambda>0$, $M \ge 0$, $q >1$, $c_0$, $c_1\ge 0$,
non-negative functions
  $\eta(t,x,y)$, $\omega : [0,1] \to \R_+$   such that $\lim_{s \to
0^+}\omega(s)=0$ and
  $g\in C((0, 1); \R_+ ) \cap L^1(0,1)$ such that
  \be\label{Feq}
  \begin{array}{c}
  F(t, x, \mu(x-y), X)- F(t,y,\mu(x-y),Y)\geq
 - \la {\rm tr}\left(X-Y\right)- \mu |x-y| \, g(|x-y|)\\
\m
 \qquad  -  (\mu | x- y|)^{2}\left(c_0+ c_1(\mu |x-y|^2)^{q-1} \right)\omega(|x-y|) - M
  - \nu \, \eta(t,x,y)\,,
  \end{array}
\ee for any $\mu> 0$, $\nu \ge 0$, \  $ x,y\in \R^N\,,
  \;\; 0<|x-y|\le 1,
  \;\;\; t\in (0,T),\,\quad    X,Y\in {\mathcal S}_N$:
\begin{align} \label{ine}
  \begin{pmatrix}
 X & 0    \\
\noalign{\medskip} 0 & -Y
 \end{pmatrix}
 \leq  \mu \, \begin{pmatrix}
 I & -I    \\
\noalign{\medskip} -I & I
\end{pmatrix}  +  \nu \begin{pmatrix}
 I & 0    \\
\noalign{\medskip} 0 & I
\end{pmatrix}.
%\qquad \hbox{for some matrix $B$: $B\leq \mu I$.}
 \end{align}
\end{hypothesis}

A few comments on Hypothesis \ref{pw-nl} are in order.

\begin{remark}\label{varieoss}\rm
 (i) To check that, in the  linear case,
  \rife{pw} implies  \rife{Feq},  we first multiply
the matrix inequality \rife{ine} by
$$
\begin{pmatrix}
 \sigma_{}(t,x)^2  & \sigma_{}(t,x)\sigma_{} (t,y)    \\
\noalign{\medskip} \sigma_{}(t,y)\sigma_{}(t,x) &
\sigma_{}(t,y)^2
\end{pmatrix}.
$$
Then, taking traces, we deduce
$$
-{\rm tr}\left((q_{}(t,x) X-q_{}(t,y)Y\right)
$$
$$
\geq -\la{\rm tr}\left(X-Y\right)-
\mu {\rm tr}\left( (\sigma_{}(t,x)  -\sigma_{}(t,y))^2  \right)
 -  \nu {\rm tr}
  \left( q_{} (t,x)  + q_{} (t,y)
- 2 \lambda I \right)$$
$$
\geq  -\la{\rm tr}\left(X-Y\right)- \mu \|(\sigma_{}(t,x)
-\sigma_{}(t,y)\|^2 - \nu  \eta(t,x,y)\,,
$$
where $\eta(t,x,y) = {\rm tr}(
q_{} (t,x))   + {\rm tr}(q_{} (t,y))  -
2 \lambda  N.$ By \rife{pw} we obtain
$$
-{\rm tr}\left((q_{}(t,x) X-q_{}(t,y)Y\right)
- \mu
[b_{}(t,x)-b(t,y)]\cdot (x-y)
$$$$
\ge -\la{\rm tr}\left(X-Y\right)-\mu |x-y|g(|x-y|)
 - \nu \eta(t,x,y),
$$
and hence \rife{Feq}.

\smallskip \noindent (ii) Hypothesis \ref{pw-nl} implies  that
$F(t,x,0,X)$ is  elliptic, at least assuming that $g(s)s\to 0$ as
$s\to 0^+$ (which is consistent with requiring that $g$ is
integrable). In that case, taking $x=y$ we get
 \be\label{elli} F(t, x, 0, X)- F(t,x,0,Y)\geq
  - \la {\rm tr}\left(X-Y\right) -M - \nu \eta(t,x,x),
\ee
 for every $X$, $Y \in {\mathcal S}_N$ satisfying the matrix
inequality. One can prove, as in \cite[Remark 3.4]{CIL}, that if
$X\leq Y$, then $X$, $Y+\vep I$
 satisfy such matrix inequality with
 $\mu= \left( 1+\frac{\|Y\|}\vep\right) \|Y\|$
and  $\nu =0$.
 Hence \rife{elli}
is satisfied by $X$, $Y+\vep I$, and letting $\vep\to 0$ shows that
\rife{elli} holds for any $X\leq Y$ with  $\nu =0$.

\smallskip \noindent (iii)
 The terms with
$c_0$, $c_1$ account for possibly nonlinear terms in the function
$F$ which may depend superlinearly on the gradient. In particular,
the term with $c_0$ includes typical terms with (at most) quadratic
growth in the gradient, while the term with $c_1$ includes further
terms with possibly larger growth. Similar kind of structure
conditions are standard in the fully nonlinear framework (see, for
instance,  \cite{IL}, \cite{CIL}, \cite{Ba2}). {\em Let us stress
that, if the function $g(s)$ satisfies $sg(s)\to 0$ as $s\to 0^+$
(as it is in most cases  when $g\in L^1(0,1)$),   we can   assume
that the function $\omega(\cdot)$ which appears in those terms is
only bounded,
 and
therefore it could be dropped being absorbed by the constants
$c_0,c_1$. This modification however requires a refinement of the
proof below, since one needs first to prove a H\"older estimate (of
some order $\alpha$, possibly small)  and then obtain the  Lipschitz
bound in a second step (see also Remark \ref{omega} later).}

\smallskip \noindent (iv)  The above Hypothesis \ref{pw-nl} is not meant to cover general situations of quasilinear operators. It is however possible to extend our approach to such situations (see e.g. \cite{Ba1}, \cite{Chen93} for similar general frameworks), by suitably modifying Hypothesis \ref{pw-nl}. In order to avoid too many additional technicalities, we decided to defer the analysis of the quasilinear case to the next future.
\end{remark}

\begin{remark}\label{cfr-IL}
{\rm Hypothesis \ref{pw-nl}  should be compared with the conditions
under which, {\sl  in a bounded domain $\Omega$,}
 the Lipschitz regularity is proved  in
\cite[Theorem VII.1]{IL}.
Even if   we essentially adopt the method introduced there,
 assumption \rife{Feq}
 %seems to allow for
 is a slightly more general condition.
For simplicity, assume that $F$ only depends on $x$ and $X$.  In that case, it is required in \cite{IL}   that
$$
|F(x,X)- F(y,X)| \leq \mu(|x-y|) \|X\|, \;\; x,y \in\bar \Omega,\;\;
X \in {\mathcal S}_N,
$$
for some non-negative function $\mu(s)$ such that $\frac{\mu}s\in
L^1(0,1)$. When specialized to the linear case, this assumption is
satisfied if $$ \|q(x)-q(y)\|\leq  \mu(|x-y|)\,.$$ On the other
hand, since $q(x) \ge \lambda I$, the previous inequality implies
%if $\sigma$ is bounded and elliptic, this is
%equivalent to
$$
\| \sqrt{q(x) - \lambda/2} - \sqrt{q(y) - \lambda/2}\|\leq C\mu(|x-y|),
$$
which corresponds to  our assumption \rife{pw} with  $\sigma(x) =
\sqrt{q(x) - \lambda/2}$, $b=0$ and $g(s)= \frac{C^2 \, \mu^2}s$.
Since our result only requires that $\frac{\mu^2}s\in L^1$ rather
than $\frac{\mu}s\in L^1$,  there is a small  improvement (ex. we
can afford $\mu(s)=\frac1{(\log s)}$).
 In fact,
 it is known
 that assumptions involving the whole matrix inequality satisfied by
$X$, $Y$ (like Hypothesis
 \ref{pw-nl}) may yield finer results  rather than assumptions made at fixed $X$ (see also
  \cite[Section III.1]{Ba1}).
}
\end{remark}

Let us proceed towards the formulation of  a general result.
 In order to take care of the behaviour at infinity,
we need to make extra assumptions; as  in Section 2,
we require the existence   of Lyapunov type functions, suitably
related to the growth of $F$.
More precisely, we assume

\begin{hypothesis} \label{pert-fi}
 For any $L \ge 0$, $\exists \,\, \varphi=
 \varphi_L \in C^{1,2}(\bar Q_T)\,, \;    \vep_0=
\epsilon_0 (L)> 0\,$:
$$
\begin{array}{l}
\begin{cases}
\vep \partial_t \varphi + F(t,x, p+ \vep D\varphi, X+ \vep
D^2\varphi)- F(t, x,p, X)\geq  0 &
\\
\m \qquad \hbox{for every $(t,x)\in Q_T$, $p\in \RN$: $|p|\leq
L+\vep |D\varphi(t,x)|$, $X\in {\mathcal S}_N$, and every $\vep\leq
\vep_0$} &
 \\
 \m
\varphi(t,x)\to +\infty \quad \hbox{as $|x|\to \infty$, uniformly
for $t\in [0,T]$.} &
\end{cases}
\end{array}
$$
\end{hypothesis}
We refer to Subsection 4.5 for  examples of operators satisfying
Hypotheses \ref{pw-nl} and \ref{pert-fi}, including the case of
Bellman-Isaacs operators, and a discussion concerning the case of
nonlinear first order terms. In particular,  we will see that in the
linear case such assumptions reduce to \rife{pw} and \rife{lyap}
which were made in the previous section. We have then the following
nonlinear version of Theorem \ref{PW}.

\begin{theorem}\label{pw-fully}
 Assume that $F(t,x,p,X)$ satisfies Hypotheses
 \ref{pw-nl}  and   \ref{pert-fi}.
   Let  $u \in C(Q_T)$
  be a
viscosity solution of \rife{fnl}.
 Moreover, assume that $u$ and $h$ have
  bounded oscillation (see
(\ref{osci})). Then   $ u(t) \in
W^{1,\infty} (\R^N)$, $t \in (0,T)$,  and
\begin{align}\label{general1}
 \|   D u(t)\|_\infty \leq  \frac {C}{\sqrt{t \wedge 1}},
\end{align}
where $C$ depends on ${{\rm osc}}_{(\frac t2\, , \, T\wedge
\frac{3}{2}t)}(u)$,  ${{\rm osc}}_{(\frac t2\, , \, T\wedge
\frac{3}{2}t)}(h)$,
  $\lambda, g, M, q$, $c_0 $, $c_1 $ and  $\omega$
 (cf. Hypothesis \ref{pw-nl}).
\end{theorem}
\proof  As in the proof of Theorem \ref{PW}, with the same notation,
 we define
\begin{align*}
\Phi_{\epsilon}(t,x,y) &= u (t,x)-u (t,y)
  -  K f(|x-y|) -
\vep \,  (\varphi(t,x)+\varphi(t,y))
  - C_0 (t-t_0)^2
   - \frac{\epsilon}{T-t}
\\ &=  u (t,x)-u (t,y)
  - z(t,x,y),
  \end{align*}
where $f$ is the solution of \rife{ode1} and  $\varphi(t,x)$ is the Lyapunov function given in Hypothesis
\ref{pert-fi} corresponding to  $L= K f'(0)$.
Arguing by contradiction, we deduce that $\Phi_\vep$
 has  a positive local maximum at
  $(\hat t, \hat x,\hat y) \in \triangle $
  $= \triangle(t_0, \delta)$ provided
\be\label{vecchiecon}
  C_0=\frac{4 \omega_{0,\de}(u)}{t_0^2}\,, \qquad K\geq \frac{\omega_{0,\de}(u)}{f(\de)}\,,
  \ee
  where $\omega_{0,\de}(u) =
 {{\rm osc}}_{(\frac {t_0}{2}\, , \, T\wedge \frac{3}{2}t_0), \de}(u)$. Since $u$ is a viscosity  solution of \rife{fnl},    by Theorem \ref{key} we end up with
$$
 \begin{array}{c}
  \frac{\epsilon}{(T- \hat t)^2} +
 2C_0(\hat t-t_0)+ \vep\,(\partial_t \varphi(\hat t, \hat x)+\partial_t \varphi(\hat t, \hat y))+
 F (\hat t ,\hat x, K D\psi(\hat x-\hat y) + \vep \,  D\varphi(\hat t, \hat x),  X_n)
 \\
 \m
 \quad -F(\hat t ,\hat y,
K D\psi(\hat x-\hat y) -\vep \,  D\varphi(\hat t, \hat y), Y_n)\leq
h(\hat t,\hat x)-h(\hat t,\hat y),
\end{array}
$$
where $\psi(\cdot)= f(|\cdot|)$  and $X_n$, $Y_n$ satisfy  the matrix inequality
\be\label{matri13} -  (n+ c_N \| D^2 z(\hat t,\hat x,\hat y)\|)I
\leq
\begin{pmatrix}
 X_n & 0    \\
\noalign{\medskip}
0 & -Y_n
\end{pmatrix}
 \leq D^2 z (\hat t,\hat x,\hat y)+
 \frac1n \,\left(D^2 z(\hat t,\hat x,\hat y)\right)^2.
\ee
 Therefore we obtain
$$
 \begin{array}{c}
 \frac{\epsilon}{(T- \hat t)^2} +
 2C_0(\hat t-t_0)+F (\hat t ,\hat x, K D\psi(\hat x-\hat y) , X_n- \vep \,  D^2\varphi(\hat t, \hat x))
 \\
 \m
 \quad -F(\hat t ,\hat y,
K D\psi(\hat x-\hat y) , Y_n +\vep \,  D\varphi(\hat t, \hat y))\leq
\omega_{0,\de}(h) + {\mathcal I}_\vep(\varphi(\hat
t,\hat x))+ {\mathcal I}_\vep(\varphi(\hat t, \hat y)),
\end{array}
$$
where $\omega_{0,\de}(h) =
 {{\rm osc}}_{(\frac {t_0}{2}\, , \, T\wedge \frac{3}{2}t_0), \de}(h)$ and where
$$
\begin{array}{c}
{\mathcal I}_\vep(\varphi(\hat t, \hat x))= - \vep\,\partial_t
\varphi(\hat t, \hat x)+ F (\hat t ,\hat x, K D\psi(\hat x-\hat y) ,
X_n- \vep \,  D^2\varphi(\hat t, \hat x))
\\
\m \qquad  -F (\hat t ,\hat x, K D\psi(\hat x-\hat y) + \vep \,
D\varphi(\hat t, \hat x), X_n)
\end{array}
$$
and similarly
$$
\begin{array}{c}
{\mathcal I}_\vep(\varphi(\hat t, \hat y))= - \vep\,\partial_t
\varphi(\hat t, \hat y)+ F(\hat t ,\hat y, K D\psi(\hat x-\hat y)
-\vep \,  D\varphi(\hat t, \hat
y), Y_n) \\
\m \qquad -F(\hat t ,\hat y, K D\psi(\hat x-\hat y) , Y_n +\vep \,
D^2\varphi(\hat t, \hat y)).
\end{array}
$$
 Recall that, since $f$ is increasing and concave, we have
$$
| K D\psi(\hat x-\hat y) | = K f'(|\hat x-\hat y|) \leq K f'(0)\,,
$$
hence, using   Hypothesis \ref{pert-fi} with $L= K f'(0)$, we deduce that ${\mathcal I}_\vep(\varphi(\hat t,\hat x))$, ${\mathcal I}_\vep(\varphi(\hat t,\hat y))\leq 0$.  In this way we get, for $ \epsilon $ sufficiently small (only depending on $K$):
 \be\label{pre-est}
 \begin{array}{c}
 \frac{\epsilon}{(T- \hat t)^2} +
 2C_0(\hat t-t_0)+F (\hat t ,\hat x, K D\psi(\hat x-\hat y) , X_n- \vep \,  D^2\varphi(\hat t, \hat x))
 \\
 \m
 \quad -F(\hat t ,\hat y,
K D\psi(\hat x-\hat y) , Y_n +\vep \,  D\varphi(\hat t, \hat y))\leq
\omega_{0,\de}(h) \,.
\end{array}
\ee
Observe that the matrix inequality \rife{matri13} implies  \be\label{ipot}
\begin{pmatrix}
\tilde X_n  & 0\\
\noalign{\medskip} 0&  - \tilde Y_n
\end{pmatrix}
\leq  K \begin{pmatrix}
 D^2\psi(\hat x-\hat y) & -D^2\psi(\hat x-\hat y)\\
\noalign{\medskip} - D^2\psi(\hat x-\hat y)&  D^2\psi(\hat x-\hat y)
\end{pmatrix} + \frac{1}{n}
(D^2 z(\hat t, \hat x, \hat y))^2
 \ee
where $\tilde X_n = X_n-  \vep\, D^2\varphi (\hat t , \hat x) $ and $\tilde Y_n=Y_n+ \vep\, D^2\varphi(\hat t, \hat y)
$.  Note that
 \begin{align} \label{the}
 \frac{1}{n}
(D^2 z(\hat t, \hat x, \hat y))^2
 \le \frac{\theta_{\epsilon} (\hat t, \hat x, \hat y)} {n}
  \begin{pmatrix}
 I & 0 \\
\noalign{\medskip} 0 &  I
\end{pmatrix},
 \end{align}
 for some positive function $\theta_{\epsilon}$ (independent of $n$).
 Moreover,  using \rife{hesspsi} and the concavity of $f$, we
have
\be \label{the1}
KD^2\psi(\hat x-\hat y) \leq K \frac{f'(|\hat x-\hat y|)}{|\hat
x-\hat y|} \,I.
\ee
Since for matrices $ B, C \in {\mathcal S}_N$,
 the inequality $B \le C$ implies $\begin{pmatrix}
 B & - B\\
  - B &  B
\end{pmatrix}$ $\le \begin{pmatrix}
 C & - C\\
 - C &  C
\end{pmatrix}$, we can use Hypothesis
\ref{pw-nl} with $\mu= K \frac{f'(|\hat x-\hat y|)}{|\hat x-\hat
y|}$ and $\nu = \frac{\theta_{\epsilon} (\hat t, \hat x, \hat y)}
{n}$ to estimate  (recall that $ D \psi(\hat x- \hat y)= K \frac{f'(| \hat x-\hat
y|)}{|\hat x-\hat y|}(\hat x-\hat y) $):
\be\label{xnyn}
\begin{array}{c}
F (\hat t ,\hat x, K D\psi(\hat x-\hat y) , \tilde X_n)
   -F(\hat t ,\hat y,
K D\psi(\hat x-\hat y) , \tilde Y_n )
\\
\m \qquad \geq - \la {\rm tr}\left(\tilde X_n - \tilde Y_n
 \right)  -K  f'(|\hat x-\hat y|) \, g(|\hat x- \hat y|)
\\
\m  \qquad  - \left(K
f'(|\hat x-\hat y|)\right)^{2}\, \left( c_0 + c_1\left(K
f'(|\hat x-\hat y|)\, |\hat x-\hat y|
\right)^{q-1}\right)\omega(|\hat x-\hat y|) - M -  \frac{\theta_{\epsilon}
(\hat t, \hat x, \hat y)} {n}
 \, \eta(\hat t, \hat x, \hat y)\,.
\end{array}
\ee
The bound from below obtained so far allows us to let $n$ go to infinity, using the compactness argument of Remark  \ref{rit}; in this way we may replace $X_n$ and $Y_n$ with  matrices $X$, $Y$ such that
\be \label{fr3}
\begin{pmatrix}
 X-  \vep\, D^2\varphi (\hat t , \hat x) & 0\\
\noalign{\medskip} 0&  - (Y+ \vep\, D^2\varphi(\hat t, \hat y))
\end{pmatrix}
\leq  K \begin{pmatrix}
 D^2\psi(\hat x-\hat y) & -D^2\psi(\hat x-\hat y)\\
\noalign{\medskip} - D^2\psi(\hat x-\hat y)&  D^2\psi(\hat x-\hat y)
\end{pmatrix}
\ee
and we obtain, combining \rife{pre-est} and \rife{xnyn}
$$
\begin{array}{c}
 \frac{\epsilon}{(T- \hat t)^2} +
 2C_0(\hat t-t_0)- \la\,{\rm tr}\left(X- \vep \, D^2\varphi(\hat t, \hat x)-(Y
+\vep \,  D^2\varphi(\hat t, \hat y)) \right)
\\
\m \leq \omega_{0,\de}(h)  + K  f'(|\hat x-\hat y|) \, g(|\hat x- \hat y|)
\\
\m   + \left(K
f'(|\hat x-\hat y|)\right)^{2}\, \left( c_0 + c_1\left(K
f'(|\hat x-\hat y|)\, |\hat x-\hat y|
\right)^{q-1}\right)\omega(|\hat x-\hat y|) + M\,.
\end{array}
$$
Now, since
$\max \Phi_\vep>0$, we deduce that
$$
K f(|\hat x-\hat y|) \leq u(\hat t, \hat x)- u(\hat t, \hat y) \leq
\omega_{0, \delta}(u),
$$
hence we get, since $f$ is concave,
$$
K f'(|\hat x-\hat y|)\, |\hat x-\hat y| \leq \omega_{0,
\delta}(u)\,.
$$
Therefore we obtain
$$
\begin{array}{c}
 \frac{\epsilon}{(T- \hat t)^2} +
 2C_0(\hat t-t_0)- \la\,{\rm tr}\left(X- \vep \, D^2\varphi(\hat t, \hat x)-(Y
+\vep \,  D^2\varphi(\hat t, \hat y)) \right)
\\
\m \leq \omega_{0,\de}(h)  + K  f'(|\hat x-\hat y|) \, g(|\hat x- \hat y|)
\\
\m   + K^2 \left(f'(|\hat x-\hat y|)\right)^2
(c_0+ c_1 \left(\omega_{0, \delta}(u) \right)^{q-1})\omega(|\hat x-\hat y|) + M\,.
\end{array}
$$
By \eqref{fr3}, we can use   \rife{rep1} in Proposition \ref{refe}, with $P=\frac{\hat x-\hat y}{|\hat x-\hat y|}\otimes
 \frac{\hat x-\hat y}{|\hat x-\hat y|}$ and  $A = K
 D^2\psi(\hat x-\hat y) $, and we estimate
 $$
\begin{array}{c}{\rm tr}\left(X- \vep \,  D^2\varphi(\hat t, \hat x)-
(Y +\vep \,  D^2\varphi(\hat t, \hat y)) \right)\leq 4 \,K f''(|\hat
x-\hat y|).
\end{array}
$$
 Hence
\be\label{fg1}
\begin{array}{c}
 \frac{\epsilon}{(T- \hat t)^2} +
 2C_0(\hat t-t_0)- 4 \la\, Kf''(|\hat x-\hat y|)
 \leq \omega_{0,\de}(h)  + K  f'(|\hat x-\hat y|) \, g(|\hat x- \hat y|)
\\
\m   + K^2 \left(f'(|\hat x-\hat y|)\right)^2
(c_0+ c_1 \left(\omega_{0, \delta}(u) \right)^{q-1})\omega(|\hat x-\hat y|) + M \,,
\end{array}
\ee
which implies, since  $f$ is the solution of \rife{ode1},
\be\label{post1}
 %\frac{\epsilon}{(T- \hat t)^2} +
 2C_0(\hat t-t_0)+ K \leq
 \omega_{0,\de}(h) +K^2 \left(f'(|\hat x-\hat y|)\right)^2 (c_0+ c_1 \left(\omega_{0, \delta}(u) \right)^{q-1})\omega(|\hat x-\hat y|)+ M\,.
\ee
Here, if $c_0=c_1=0$ in Hypothesis \ref{pw-nl}, we conclude exactly as in Theorem \ref{PW}, obtaining,  similarly to \rife{preliouv},  that
\be\label{liouv2}
\begin{array}{c}
|u(t_0,x)-u(t_0,y)|\leq  K\, f(|x-y|),   \qquad  |x-y|<\de
\\
\m
\m
\hbox{where $K =\frac{4 \omega_{0,\de}(u) }{t_0} + \omega_{0,\de}(h) + M+
  \, \frac{\omega_{0,\de}(u) }{f(\de)}$\,.}
  \end{array}
\ee Otherwise, due to the presence of  the nonlinear terms $K^2
\left(f'(|\hat x-\hat y|)\right)^2$ $ (c_0+ c_1 \left(\omega_{0,
\delta}(u) \right)^{q-1})\omega(|\hat x-\hat y|)$ we do not continue
exactly as in the proof of Theorem \ref{PW}. We first fix
 $$
 K=  \frac{8 \lambda \, \omega_{0, \delta}(u)}{ \de^2},
  $$
so that $K$ satisfies \rife{vecchiecon}  thanks to \rife{fde1}; we
will choose $\de$ later. Then, recalling \rife{fde2} and the
concavity of $f$,  we estimate \be\label{omega-1}
\begin{array}{c}
K^2\, \left(f'(|\hat x-\hat y|)\right)^2(c_0+ c_1 \left(\omega_{0, \delta}(u) \right)^{q-1})\omega(|\hat x-\hat y|)\leq  C\, K^2\,
   \de^2(c_0+ c_1\left(\omega_{0, \delta}(u)\right)^{q-1})
   \omega(|\hat x-\hat y|)
   \\
   \m
   \qquad \leq C\, K\, \omega_{0, \delta}(u)
  (1+ \left(\omega_{0, \delta}(u)\right)^{q-1}) \omega(|\hat x - \hat y|)\,.
  \end{array}
\ee
Since $\lim_{r \to 0^+}\omega(r)=0$, there exists $\de_0$ (depending on $\omega$ and $\omega_{0, 1}(u)$)  such that   $\delta \le \delta_0$ implies
$$
K^2\, \left(f'(|\hat x-\hat y|)\right)^2 (c_0+ c_1 \left(\omega_{0, \delta}(u) \right)^{q-1}) \omega(|\hat x-\hat y|)\leq \frac K 2.
$$
Therefore we get from \rife{post1}
%(recall that $C_0 = \frac{4 \omega_{0,
%\delta}(u)}{t_0^2}$)
$$
\frac{K}{2} +  \frac{\epsilon}{(T-\hat t)^2}
   \leq  C_0 t_0
 + \omega_{0,\de}(h)+ M\,.
$$
Choosing again  $\de$ small enough we  have  $\frac K 2 \ge  C_0 t_0
 + M    + \omega_{0,\de}(h)$ and we obtain a contradiction. Indeed, recalling the value of $C_0$ (see \rife{vecchiecon}) and that $K= \frac{8 \lambda \omega_{0, \delta}(u)}{ \de^2}$, it is enough to take
$$
\delta \le  \Big( \sqrt{
\frac{ \lambda \,  t_0\, \omega_{0,
\delta}(u)}{ 4 \omega_{0, \delta}(u) + M t_0 + \omega_{0,\de}(h)  t_0 }} \,
\, \Big)\, \wedge \, \delta_0.
$$
This means that we can choose  $\delta =C_3 \sqrt{t_0 \wedge 1}$
for some constant $C_3 = C_3 (\de_0, \omega_{0, \delta}(u)$,
 $\lambda, g, M, q$, $c_0 $, $c_1 $); recall that
  $\omega_{0, \delta}(u) \le {{\rm osc}}_{(\frac t2\, , \, T\wedge
\frac{3}{2}t)}(u)$. Finally, we have proved that
 $$
 |u(t_0, x)-u(t_0, y)|\leq K f'(0) |x-y| \le c\, K \delta
  |x-y|,
  \qquad |x-y|\leq \de\,,
 $$
and the proof is complete.
\qed
\vskip0.4em
Now we consider a similar
regularizing effect but concerning the H\"older continuity  of
solutions (i.e., the extension of Proposition \ref{hol-lin}).  In
this case we still assume the structure condition of Hypothesis
\ref{pw-nl} but replacing the condition $g \in  L^1 (0,1)$
 with  $\lim_{s\to 0^+} s g(s) =0$; moreover,
 we do not need any more the modulus of continuity $\omega(|x-y|)$ in
the superlinear terms (see also Remark \ref{varieoss} (iii)).

\begin{hypothesis} \label{pw-nl-hol}
There exist $\lambda>0$, $M \ge 0$, $q >1$,  $c_0$, $c_1\ge
0$, non-negative functions
  $\eta(t,x,y)$  and
  $g\in C((0, 1); \R_+ )$ satisfying $\lim_{s\to 0^+} s g(s) =0$, such that
  \be\label{Feq2}
  \begin{array}{c}
F(t, x, \mu(x-y), X)- F(t,y,\mu(x-y),Y)\geq
 - \la {\rm tr}\left(X-Y\right)- \mu |x-y| \, g(|x-y|)\\
\m
 \qquad  -  (\mu | x- y|)^{2}\left(c_0+ c_1(\mu |x-y|^2)^{q-1} \right)- M
  - \nu \, \eta(t,x,y)\,,
  \end{array}
\ee for any $ \mu>0$, $\nu \ge 0$, \ $x,y\in \R^N\,,
  \;\; 0<|x-y|\le 1,
  \;\;\; t\in (0,T),\,
   \,\,  X,Y\in {\mathcal S}_N\,$ satisfying \rife{ine}.\end{hypothesis}

We also need to replace Hypothesis \ref{pert-fi} with the following
stronger one, which accounts for the fact that solutions will
possibly  have unbounded gradient.
\begin{hypothesis} \label{pert-fi-bis}
$\exists \,\, \varphi \in C^{1,2}(\bar Q_T)\,, \;  \,\epsilon_0
>0 \,$:
$$
\begin{array}{l}
%\exists \,\, \varphi \in C^2([0,T)\times \RN)\,, \;  \,\epsilon_0 >0
%\,:\, \;
%\\
%\m
\begin{cases}
\vep \partial_t \varphi + F(t,x, p+ \vep D\varphi, X+ \vep
D^2\varphi)- F(t, x,p, X)\geq  0 &
\\
\m \qquad \hbox{for every $(t,x)\in Q_T$, $p\in \RN$, $X\in
{\mathcal S}_N$, and every $\vep\leq \vep_0$} &
 \\
 \m
\varphi(t,x)\to +\infty \quad \hbox{as $|x|\to \infty$, uniformly
for $t\in [0,T]$.} &
\end{cases}
\end{array}
$$
\end{hypothesis}

\begin{proposition}\label{hol-fully}
 Assume that $F$ satisfies Hypotheses
 \ref{pw-nl-hol} and   \ref{pert-fi-bis}.
   Let  $u \in C(Q_T)$
  be a
viscosity solution of \rife{fnl}.
 Moreover, suppose that $u$ and $h$ have
  bounded oscillation.
  Then   $ u(t) $ satisfies,
 for any $\alpha \in (0,1)$,
\begin{align} \label{ci11}
 |    u(t, x) - u(t,y)| \leq  \frac {C}{{(t \wedge 1)^{\alpha/2}}}
 |x-y|^{\alpha},
\end{align}
 for all $x, y \in \R^N$, $|x-y| \le 1$, $t \in (0,T)$,
where $C$ depends on $\alpha$,  ${{\rm osc}}_{(\frac t2\, , \, T\wedge
\frac{3}{2}t)}(u)$,  ${{\rm osc}}_{(\frac t2\, , \, T\wedge
\frac{3}{2}t)}(h)$,
  $\lambda, g, M, q$,  $c_0 $ and  $c_1$.

Moreover, if we replace $\lim_{s\to 0^+} s g(s) =0$ with the
condition
 $\limsup_{s \to 0^+} s g(s) <
4 \lambda$, then there exists some  $\alpha = \alpha ($
${{\rm osc}}_{(\frac t2\, , \, T\wedge \frac{3}{2}t)}(u)$, ${{\rm osc}}_{(\frac
t2\, , \, T\wedge \frac{3}{2}t)}(h),$
  $\lambda, g, M, q, c_0, c_1) \in
(0,1)$ such that
 \rife{ci11} holds.
\end{proposition}
\proof  We proceed as in the proof of Theorem \ref{pw-fully} (see
also the proof of Proposition \ref{hol-lin}).
 The main difference is that we consider $f(s) = s^{\alpha}$
and take $\delta \le \delta_1 <1$ such that
 $sg(s)<2\la (1-\alpha)$
if $s<\de_1$. We also fix   $K = \frac{\omega_{0,
\delta}(u)}{\delta^{\alpha}}$ and we use Hypothesis \ref{pert-fi-bis} to get rid of the terms involving the Lyapunov function.
 Since $f$ is increasing
  and concave, we arrive again at the inequality \rife{fg1}, which
implies
$$
\begin{array}{c}
 \frac{\epsilon}{(T- \hat t)^2}+
 2C_0(\hat t-t_0) \leq
\alpha \,K\, |\hat x-\hat y|^{\alpha-2}  \big( 4\la\,(\alpha-1) +
|\hat x-\hat y| g(|\hat x-\hat y|)\big) + \omega_{0,\de}(h)
\\
\m
\quad
+ K^2 \left(f'(|\hat x-\hat y|)\right)^2(c_0+ c_1 \left(\omega_{0, \delta}(u)
    \right)^{q-1}) + M\,.
    \end{array}
 $$
Since $|\hat x - \hat y| < \delta \le \delta_1$, we get
$$
 \alpha \,K\, |\hat x-\hat y|^{\alpha-2}  \big( 4\la\,(\alpha-1) +
|\hat x-\hat y| g(|\hat x-\hat y|)\big) \leq 2\la\,\alpha(\alpha-1)
\,K\, |\hat x-\hat y|^{\alpha-2}.
$$
Using the precise choice of $K$ and that $|\hat x-\hat y|<\de$, we
have
 \be\label{aleps} K^2\, \left(f'(|\hat x-\hat y|)\right)^2
\leq \,  K \omega_{0, \delta}(u) \frac1{\de^\alpha} \alpha^2 |\hat
x-\hat y|^{2\alpha -2} \leq \,  K \omega_{0, \delta}(u) \alpha^2
|\hat x-\hat y|^{\alpha -2} .
\ee
 Therefore, if $\alpha$ is small
enough (eventually depending on the oscillation of $u$), we have
\be\label{aleps2} K^2 \left(f'(|\hat x-\hat y|)\right)^2(c_0+ c_1
\left(\omega_{0, \delta}(u)
    \right)^{q-1}) \leq  \lambda \alpha (1 - \alpha
)K\, |\hat x-\hat y|^{\alpha-2}.
\ee
It follows that
$$
\frac{\epsilon}{(T- \hat t)^2} + \la\,\alpha(1- \alpha)K\, |\hat x-\hat y|^{\alpha-2}\le  2C_0(t_0-\hat t)+ \omega_{0,\de}(h) +  M\,.
$$
Recalling the values of $C_0$ and $K$, and since $|\hat x-\hat y|<\de$,  we obtain then
$$
\frac{\epsilon}{(T- \hat t)^2} + \la\,\alpha(1- \alpha) \,
\delta^{\alpha -2} \frac{\omega_{0, \delta}(u)}{\delta^{\alpha}}
\le  \frac{4 \omega_{0,\de}(u)} {t_0} +  \omega_{0,\de}(h)  +  M\,.
$$
We can choose now  $\de$ small enough so that  $\la\,\alpha(1-
\alpha) \,  \delta^{-2} {\omega_{0, \delta}(u)}  \ge
\omega_{0,\de}(h) +  \frac{\omega_{0, \delta}(u)} {t_0}  +  M$ and we get a
contradiction.  To this purpose, it is enough to consider
 \be
\label{rrr} \delta \le  \Big( \sqrt{
%\frac{1}{2k_1}
\, \frac{ \lambda \alpha (1-\alpha) \, t_0\,
\omega_{0, \delta}(u)}{ 4 \omega_{0, \delta}(u) + M t_0 +
\omega_{0}(h)  t_0 }} \, \, \Big)\, \wedge \, \delta_1. \ee
Continuing as in the proof of Theorem \ref{pw-fully} we obtain
 $$
 |u(t_0, x)-u(t_0, y)|\leq
 \frac{\omega_{0, \delta}(u)}{\delta^{\alpha}}  |x-y|^{\alpha} ,
  \qquad |x-y|\leq \de\,.
 $$
Taking into account  \rife{rrr} we find the assertion, which is now
proved for $\alpha$ sufficiently small. We repeat now the same
scheme with any $\alpha<1$; knowing that \rife{ci11} holds at least
for $\alpha$ small, we estimate the oscillation of $u$ as
$$
\omega_{0,\de}(u)\leq   \frac {C}{{(t_0 \wedge 1)^{\alpha_0/2}}}
 \de^{\alpha_0}
$$
for some (small) $\alpha_0>0$, possibly depending
 on the oscillation  $\omega_{0,1}(u)$. Note that
  $\delta \le C_3 \sqrt{t_0 \wedge 1}$
for some constant $C_3 = C_3 (\omega_{0, \delta}(u)$,
 $\lambda, g, M, q$, $c_0 $, $c_1 $, $\alpha$).

  Now \rife{aleps}
implies
$$
K^2\, \left(f'(|\hat x-\hat y|)\right)^2   \leq \,  K \omega_{0,
\delta}(u) \alpha^2 |\hat x-\hat y|^{\alpha -2} \leq \, K  \frac
{\de^{\alpha_0}}{{(t_0 \wedge 1)^{\alpha_0/2}}} \, \alpha^2|\hat
x-\hat y|^{\alpha -2} \,,
$$
and we deduce again that \rife{aleps2} holds true
 provided $ \frac {\de}{\sqrt{t_0 \wedge 1}}$ is sufficiently small. We conclude then as before and obtain the estimate \rife{ci11} for every $\alpha<1$.

To prove the last statement, we set $\gamma= \limsup_{s \to 0^+} s
g(s)$ and proceed similarly.  First we choose $\alpha \in (0,1)$
such that $ \sigma = 4\la\,(1- \alpha) - \gamma >0$. Then we obtain,
for $|\hat x - \hat y| < \delta \le \delta_1$,
$$
 \alpha \,K\, |\hat x-\hat y|^{\alpha-2}  \big( 4\la\,(\alpha-1) +
|\hat x-\hat y| g(|\hat x-\hat y|)\big) \leq - \sigma \,K\, |\hat
x-\hat y|^{\alpha-2}.
$$
Next,  we proceed as before and   obtain the
 assertion.
 \qed

\begin{remark}\label{3e}{\rm  We stress that  if  $c_0=c_1=0$ in Hypothesis \ref{pw-nl} or  in  Hypothesis \ref{pw-nl-hol}, then
we obtain the same estimates as in the linear case (see e.g.  \rife{liouv2} in the proof of Theorem \ref{pw-fully}). In particular, in this case  estimate
 \rife{general1} can be precised as
\begin{align}\label{general111}
& \|   D u(t)\|_\infty \leq  \frac {\hat c_1}{\sqrt{t \wedge 1}} \, \,
\omega(t,u)+ \, \, \hat c_2
\sqrt{t \wedge 1}\, \, (\omega(t,h)+M)\, ,\;\; t \in (0,T),
\end{align}
which is of the same kind as \rife{general}.
 Similarly we have, as in \rife{lip-delta}:
 \be \label{llip}
 {\rm Lip}(u(t_0))\leq  C_\la\, \delta   \,  \Big( \frac{\omega_{0,\de}(u)}
 {t_0} + \omega_{0,\de}(h) + M + \frac{\omega_{0,\de}(u)}
 { \delta^2} \Big),
\ee where $C_\la$ only depends on
$\la$ and $g$.

Similarly, if  $c_0=c_1=0$ in Hypothesis \ref{pw-nl-hol} then estimate \rife{ci11}
becomes as \rife{ci1} (with $\omega(t,h)$ replaced by
 $\omega(t,h) +M$).
}
\end{remark}

\subsection{The case of bounded data and solutions}

As a first application of  Theorem \ref{pw-fully} and Proposition
\ref{hol-fully} to the Cauchy problem
\begin{align}\label{cauful}
\begin{cases} \partial_t u +
F(t,x,  D u,D^2 u)= h(t,x)
\\ u(0,\cdot) =u_0,
\end{cases}
\end{align}
%when the initial datum $u_0$ is continuous and bounded on %$\R^N,$
we   consider the easier case in which $u_0$ is bounded on $\R^N$,
and $h$ is bounded on $Q_T$. Recall that a viscosity solution to
\rife{cauful} is a function $u \in C(\bar Q_T)$ which is a
 viscosity solution to \rife{fnl} and satisfies $u(0, \cdot) = u_0$.
 We also assume that
 \be\label{F0} h_1(t,x) := F(t,x,0,0) \;\;
 \text{is bounded }\,
\; \hbox{on $ Q_T.$}
 \ee
The following  lemma, based on the
Lyapunov function $\varphi$ (cf. Lemma \ref{max}), ensures the
boundedness of solutions.

\begin{lemma} \label{maxfully}
Assume  \rife{F0} and Hypothesis \ref{pert-fi}. Let $h\in C(Q_T)\cap L^\infty(Q_T)$.
  Let $u\in C\left( \bar Q_T \right)$
  be a viscosity solution of  \rife{cauful} such that
  $u_0 : = u(0, \cdot)$ is bounded on   $\R^N$ and moreover $u$ is $o_{\infty}(\varphi_0)$ in
  $\bar Q_T$,
   where $\varphi_0$ satisfies Hypothesis \ref{pert-fi} for  $L=0$.  Then
 $u$ is bounded on $[0,T] \times \R^N$ and, setting $\tilde h = h- h_1 $,
 \begin{equation} \label{max1}
 \| u(t)\|_{\infty} \le \|u_0 \|_{\infty} \, + \,
 t \| \tilde h \|_{T, \infty},\;\;\; t \in (0,T),
\end{equation}
where $\|\tilde  h \|_{T, \infty}
 = \sup_{t\in [0,T]} \|\tilde h(t)\|_{\infty}$.
\end{lemma}
\begin{proof}
We note that
  $u$ is a viscosity solution of
 \eqref{cauful} with $F$ replaced by $\tilde F= F(t,x,p,X)-F(t,x,0,0)$ and $h$ replaced by $h - h_1$.
 Moreover,
 $\tilde F (t,x,0,0)=0.$
Let us consider, for $\epsilon>0$,
$$
f_{\epsilon}(t,x) = u(t,x)-\vep \, \varphi_0(t,x)
  - \frac\vep{T-t} - \| u_0\|_{\infty} -
  t  \| \tilde h\|_{T,\infty}.
 $$ If, for any $\epsilon >0$, we have $
f_{\epsilon}(t,x) \le 0$, then, letting $\epsilon \to 0^+$,
  we deduce $u(t , x) \le \| u_0\|_{\infty} +
  t \| \tilde h\|_{T,\infty}$.
  Arguing by contradiction, suppose that, for some $\epsilon >0$,
   $
\sup\limits_{[0,T)\times \RN} \,  f_{\epsilon} (t,x) \, >0$. Since $u=o_{\infty}(\vfi_0)$ this sup is a maximum attained at some point $(t_{\epsilon}, x_{\epsilon})$. Note
   that $t_{\epsilon} \in (0,T)$ and
  so   $(t_{\epsilon}, x_{\epsilon})$
     is a local maximum.  By
 definition of subsolution we have
\be \label{dd1}
   \| \tilde h\|_{T,\infty} + \frac\vep{(T- t_{\epsilon})^2}
+ \vep \partial_t \varphi_0 (t_{\epsilon}, x_{\epsilon}) +
\tilde F(t_{\epsilon}, x_{\epsilon}, \vep  D\varphi_0(t_{\epsilon},
x_{\epsilon}), \vep D^2\varphi_0(t_{\epsilon}, x_{\epsilon})) \le  \tilde h(
   t_{\epsilon}, x_{\epsilon}).
\ee Using Hypothesis \ref{pert-fi}  with $p=0$ and $X=0$, and
\rife{F0}, we deduce
$$\| h\|_{T,\infty} + \frac\vep{(T- t_{\epsilon})^2}
\le  \tilde h(t_{\epsilon}, x_{\epsilon})
$$
which  is   a contradiction. To obtain the opposite inequality,
i.e., $ u(t,x) \ge -\| u_0\|_{\infty} -
  t \| \tilde h\|_{T,\infty}$, we introduce
  $$
g_{\epsilon}(t,x) = u(t,x)+ \vep \, \varphi_0(t,x)
  + \frac\vep{T-t} + \| u_0\|_{\infty} +
  t  \| \tilde h\|_{T,\infty},
$$
and   we suppose that, for some $\epsilon
>0$,
   $
\min\limits_{[0,T)\times \RN} \, g_{\epsilon} (t,x) \, <0$.
 Let $(t_{\epsilon}, x_{\epsilon})$ be a point
   where   this minimum is attained.
   Considering now $u$ as a supersolution, we get
$$
 -  \| \tilde h\|_{T,\infty} - \frac\vep{(T- t_{\epsilon})^2}
- \vep \partial_t \varphi_0 (t_{\epsilon}, x_{\epsilon}) +
\tilde F(t_{\epsilon}, x_{\epsilon}, - \vep  D\varphi_0(t_{\epsilon},
x_{\epsilon}), - \vep D^2\varphi_0(t_{\epsilon}, x_{\epsilon}))  \ge
\tilde h(
   t_{\epsilon}, x_{\epsilon}).
$$
Using Hypothesis \ref{pert-fi}  with $p=-\vep D\varphi_0
 (t_{\epsilon}, x_{\epsilon}) $ and $X=-\vep
D^2\varphi_0(t_{\epsilon}, x_{\epsilon})$, and using \rife{F0}, we obtain again a contradiction, and then we conclude.
\end{proof}

Using Lemma \ref{maxfully}, we immediately obtain from Theorem
\ref{pw-fully} and Proposition \ref{hol-fully}:

\begin{corollary}
\label{fulbound}
 Assume  \rife{F0} and suppose that
   $h\in C(Q_T)\cap L^\infty(Q_T)$.
  Assume also that $u_0$ is bounded on $\R^N$.
   Let  $u\in C(\bar Q_T)$
 be a
 viscosity solution of \rife{cauful} which is also   $o_{\infty}(\varphi_0)$ in
 $\bar Q_T$ ($\varphi_0$ satisfies Hypothesis \ref{pert-fi} for  $L=0$). We have the following statements, where $\tilde h = h- h_1$.

\smallskip  i) If Hypotheses  \ref{pw-nl} and \ref{pert-fi}
 hold, then
  $ u(t) \in W^{1,\infty} (\R^N)$, $t \in
(0,T)$,  and
 \rife{general1} holds with
 a constant $C$ depending on $\| u_0 \|_{\infty}$,
   $ \| \tilde h\|_{T, \infty}$,
  $\lambda, g, M, q$,  $c_0$, $c_1$ and  $\omega$
 (cf. Hypothesis \ref{pw-nl}).

\smallskip ii)
 If Hypotheses \ref{pw-nl-hol} and \ref{pert-fi-bis} hold, then
 $u(t)$
 verifies   \rife{ci11},
 for any $\alpha \in (0,1)$,
   with a constant $C$  which depends   on
  $\alpha$,  $ \| u_0\|_{\infty}$, $\| \tilde h\|_{T, \infty},$
  $\lambda, g, M, q$, $c_0$ and $c_1$.

Moreover, if in Hypothesis \ref{pw-nl-hol} the function
$g$ only satisfies $\limsup_{s \to 0^+} s g(s) <
4 \lambda$, then there exists some $\alpha = \alpha (
\|u_0\|_{\infty}$, $\| \tilde h\|_{T, \infty},$
  $\lambda, g, M, q, c_0, c_1) \in
(0,1)$ such that \rife{ci11} holds.
\end{corollary}

\begin{remark}\label{3e1}{\rm Similarly to
Remark \ref{3e}  if  $c_0=c_1=0$ in Hypotheses \ref{pw-nl} and
\ref{pw-nl-hol} then the previous global estimates of i) and ii)
become like \rife{grad} and \rife{hol1} respectively. }
\end{remark}

\subsection{Bounded data  and more general  equations containing
   $u$ }
The previous result can be  proved in   more generality  in order to provide a
nonlinear version of  Theorem \ref{PWc} and Proposition
\ref{hol-lin1}. Recall that  these results  also contain the case of
unbounded potential terms $V$ which is not covered in Corollary
\ref{fulbound}.

 {\it Let us only show how to generalize (i) in Corollary \ref{fulbound} to the case when $F$ depends also on $u$, namely for the Cauchy problem
\be\label{cauu}
\partial_t u+ F(t,x,u,Du,D^2 u)=0 \qquad \hbox{in $Q_T$,}\quad u(0)=u_0\,.
\ee
More precisely, we show how  to extend Theorem \ref{PWc} to  the present  nonlinear setting.}

\smallskip  For simplicity   we  assume that  $h=0$ and $F(t,x,0,0,0)=0$, $(t,x) \in Q_T$.
Hypothesis  \ref{pw-nl} can   be generalized as follows: there exists $\lambda>0$, $q >1$, $c_0, c_1,  k_0  \ge 0$, $\ga_R, M_R \ge 0$,
non-negative functions
  $\eta(t,x,y)$, $V(t,x)$,  $\omega : [0,1] \to \R_+$   such that $\lim_{s \to
0^+}\omega(s)=0$ and
  $g\in C((0, 1); \R_+ ) \cap L^1(0,1)$ such that
  \be\label{Feq11}
  \begin{array}{c}
  F(t, x, r, \mu(x-y), X)- F(t,y,s,\mu(x-y),Y)\geq
 - \la {\rm tr}\left(X-Y\right)- \mu |x-y| \, g(|x-y|)
  \\
  \m
   -  (\mu | x- y|)^{2}\left(c_0+ c_1(\mu |x-y|^2)^{q-1} \right)\omega(|x-y|)
   \\ \m + \ga_R (V(t,x)\vee V(t,y))(r-s)- M_R  \big(1 +
  { k_0} |x-y| ( V(t,x) \vee V(t,y))\big)
  - \nu \, \eta(t,x,y)\,,
  \end{array}
\ee for any $\mu> 0$, $\nu \ge 0$, \  $ x,y\in \R^N$ such that  $0<|x-y|\le 1,
  $ $t\in (0,T)$,  $r,s\in \R$ such that $- R \le s \le r \le R$, $R>0$, and $    X,Y\in {\mathcal S}_N$ which satisfy
  \rife{ine}.

One can easily check that \rife{Feq11} is satisfied in the linear
case considered in Theorem \ref{PWc}.
Finally, Hypothesis \ref{pert-fi} has to be generalized as follows:
 for any $L \ge 0$, $\exists \,\, \varphi= \varphi_L \in C^2([0,T)\times \RN)\,, \;    \vep_0=
\epsilon_0 (L)> 0\,$:
\be \label{lyp}
\vep \partial_t \varphi + F(t,x,r, p+ \vep D\varphi, X+ \vep
D^2\varphi)- F(t, x, r,p, X)\geq  0,
\ee
for every $(t,x)\in Q_T$, $p\in \RN$ such that  $|p|\leq
 L+\vep |D\varphi(t,x)|$, $X\in {\mathcal S}_N$,
 $ r \in \R$, and every $0<\vep\leq
\vep_0$, with
 $\varphi(t,x)\to +\infty$ as $|x|\to \infty$, uniformly
for $t\in [0,T]$.

\smallskip To treat the Cauchy problem \rife{cauful} when $F$ depends also on $u$, we first note that,  replacing Hypotheses \ref{pw-nl} and
\ref{pert-fi} respectively with \rife{Feq11} and \rife{lyp}, one can
still prove Lemma \ref{maxfully}. To this purpose, the only
modification is that one uses \rife{lyp} together with \rife{monou}
and $F(t,x,0,0,0)=0$ in order to deduce that
$$
\vep \partial_t \varphi_0 (t_{\epsilon}, x_{\epsilon}) +
F(t_{\epsilon}, x_{\epsilon}, u(t_{\epsilon},
x_{\epsilon}), \vep  D\varphi_0(t_{\epsilon},
x_{\epsilon}), \vep D^2\varphi_0(t_{\epsilon}, x_{\epsilon})) \geq 0,
$$
and, of course, a similar argument for the lower bound.

Then, we observe that  the proof of Theorem \ref{pw-fully} still works with small changes under condition \rife{Feq11}. Recall  that here we are assuming that $u_0$ is bounded
 and so by Lemma \ref{maxfully} the solution $u$ is also bounded, which makes
 possible such  a variation. Indeed, we use \rife{Feq11} with $R= \| u_0\|_{\infty} + T \| h\|_{T, \infty}$, $r = u(\hat t ,\hat x)$ and $s=u(\hat t ,\hat y)$. With the same notations of Theorem   \ref{pw-fully},  we obtain now, instead of \rife{xnyn}, the following inequality:
$$
\begin{array}{c}
F (\hat t ,\hat x,u(\hat t ,\hat x),  K D\psi(\hat x-\hat y) , \tilde X_n )
   -F(\hat t ,\hat y, u(\hat t ,\hat y),
K D\psi(\hat x-\hat y) , \tilde Y_n )
\\
\m \geq - 4 \la\, Kf''(|\hat x-\hat y|)
 -K f'(|\hat x-\hat y|) \,
g(|\hat x- \hat y|) -K^2 \left(f'(|\hat x-\hat y|)\right)^2(c_0+ c_1 \left(\omega_{0, \delta}(u) \right)^{q-1})\omega(|\hat x-\hat y|)
\\ \m+
 \gamma_R (V(\hat t,\hat x)\vee V(\hat t, \hat y))( u(\hat t,\hat x)-u(\hat t,\hat y))  - M_R (1+{ k_0}   |\hat x- \hat y| ( V(\hat t,\hat x) \vee V(\hat t,\hat y)))
 \\
\m \quad -
 \frac{2N \lambda \,
 \theta_{\epsilon} (\hat t, \hat x, \hat y)}{n}  -
\frac{\theta_{\epsilon} (\hat t, \hat x, \hat y)} {n}
 \, \eta(\hat t, \hat x, \hat y).
\end{array}
$$
We  proceed letting $n$ tend to infinity and using \rife{ode1}.
Since $( u(\hat t,\hat y)-u(\hat t,\hat x))< -  K f(|\hat x - \hat
y|)$,
 using
 \rife{fde0} we end up by replacing  \rife{post1} with the following:
$$
\begin{array}{c}
 \frac{\epsilon}{(T- \hat t)^2} +
 2C_0(\hat t-t_0)\leq
 -K
   +K^2 \left(f'(|\hat x-\hat y|)\right)^2
(c_0+ c_1 (2R)^{q-1})\omega(|\hat x-\hat y|) \\
\m-
 \gamma_R \, K \frac{\delta}{8\lambda} (V(\hat t,\hat x)\vee V(\hat t, \hat y))+ M_R (1+{ k_0}   |\hat x- \hat y| ( V(\hat t,\hat x) \vee V(\hat t,\hat y)))\,.
\end{array}
$$
Choosing $K \ge  \frac{k_0 M_R}{\ga_R} \frac{8 \lambda}{\delta}$, we
drop the terms with the potential $V$
%we get $[R k_1  -   K \frac{\delta}{4\lambda}]\, |\hat x - \hat y|\, ( V(\hat t,\hat x) \vee V(\hat t,\hat y))) \le 0$
and next we conclude as in the proof of Theorem \ref{pw-fully}.
%}
%\end{remark}

\subsection{Unbounded data: estimates on the oscillation of the solutions}

We consider now the Cauchy problem \rife{cauful} in the case where
the initial datum $u_0$ is not necessarily bounded. The   first goal
is to obtain  estimates on the oscillation of  viscosity solutions,
so that Theorem \ref{pw-fully} or Theorem \ref{hol-fully} can
provide a full estimate only depending on the data. To this purpose
we will assume that $h$ and $u_0$ have bounded oscillation and we
will need to modify Hypotheses \ref{pw-nl} or \ref{pw-nl-hol} by
requiring some additional condition when $|x-y|$ is large. Hence
Hypotheses \ref{pw-nl} and \ref{pw-nl-hol} are modified in
the following way.

 \begin{hypothesis}\label{pw-nl1}
There exist $\lambda>0$, $M_0,M_1 \ge 0$, $q\in (1,2)$, $c_0$, $c_1
\ge 0$, non-negative functions $\eta(t,x,y)$, $\omega \in
C(\R_+;\R_+)$ which is bounded and satisfies $\lim_{s\to 0^+}
\omega(s)=0$ and
  $g\in C((0, + \infty); \R_+ )\cap L^1(0,1)$ which is O(r)
  as $r \to +\infty$,  such that
$$
\begin{array}{c}
F(t, x, \mu(x-y), X)- F(t,y,\mu(x-y),Y)\geq
 - \la {\rm tr}\left(X-Y\right)- \mu |x-y| \, g(|x-y|)\\
\m
 \qquad  -  (\mu | x- y|)^{2}\left(c_0+ c_1(\mu |x-y|^2)^{q-1} \right)\omega(|x-y|) - M_0-M_1|x-y|
  - \nu \, \eta(t,x,y)\,,
\\
\m \m
  \text{for any} \; \mu>0,\;  \nu \ge 0\,, \;\;   x,y\in \R^N\,,
  \;\;\; t\in (0,T),\,
   \,\,  X,Y\in {\mathcal S}_N\,\,\hbox{satisfying \rife{ine}.}
\end{array}
$$
\end{hypothesis}

 \begin{hypothesis}\label{pw-nl1-hol}
There exist $\lambda>0$, $M_0,M_1 \ge 0$, $q\in (1,2)$,
 $c_0$, $c_1 \ge 0$,  a non-negative
functions
  $\eta(t,x,y)$ and
  $g\in C((0, + \infty); \R_+ )$ which is O(r)
  as $r \to +\infty$ and satisfies $g(s)s\to 0$ as $s\to 0^+$,  such that
$$
\begin{array}{c}
F(t, x, \mu(x-y), X)- F(t,y,\mu(x-y),Y)\geq
 - \la {\rm tr}\left(X-Y\right)- \mu |x-y| \, g(|x-y|)\\
\m
 \qquad  -  (\mu | x- y|)^{2}\left(c_0+ c_1(\mu |x-y|^2)^{q-1} \right) - M_0-M_1|x-y|
  - \nu \, \eta(t,x,y)\,,
\\
\m \m
  \text{for any} \; \mu> 0,\;  \nu \ge 0\,, \;\;   x,y\in \R^N\,,
  \;\;\; t\in (0,T),\,
   \,\,  X,Y\in {\mathcal S}_N\,\,\hbox{satisfying \rife{ine}.}
\end{array}
$$
\end{hypothesis}

 \vskip1em

\noindent
As a first step, we need the following growth   estimates
which generalize Lemma \ref{osc-u}.  This result only requires
Hypothesis \ref{pw-nl1-hol}. On the other hand, Theorem
\ref{finnonl} will also use
 Hypothesis \ref{pw-nl1}.

\begin{lemma} \label{osc2}
Assume that $u_0$, $h$ satisfy \rife{osc-uo} and \rife{osc-h},
 respectively.   There exists   $L$ (depending on
  $T, g, \lambda,
\alpha,  q,$ $c_0$, $c_1$, $M_0$, $M_1$ and $k_i, h_i$, for $i=0,\alpha,1$) such that
 if
 $u \in C\left( \bar Q_T\right)$ is a viscosity solution of
\rife{cauful}  and  $u$ is $o_{\infty}(\varphi_L)$ in $\bar Q_T$  (see Hypothesis \ref{pert-fi}), then we have:

(i) If  Hypotheses \ref{pw-nl1-hol} and  \ref{pert-fi-bis} hold true, then $u$ has bounded
oscillation and there exists $C_T$  such that
\be\label{CT}
| u (t,x)-u (t,y)|
  \le   C_T \big( 1+ |x-y|\big),\;\; x, y \in
\R^N, \, t\in [0,T], \;\; \ee where $C_T$ depends on $T, g, \lambda,
\alpha,  q,$ $c_0$, $c_1$, $M_0$, $M_1$ and $k_i, h_i$, for $i=0,\alpha,1$.

\smallskip (ii) If Hypothesis \ref{pw-nl1-hol} holds replacing the condition $
\lim_{s\to 0^+} g(s)s=0$ with $g\in L^1(0,1)$, and if Hypothesis
\ref{pert-fi} holds true, then $u$ has bounded oscillation and
satisfies \rife{CT}.

\smallskip (iii) Assume   Hypothesis \ref{pw-nl1-hol}
 with $c_0=c_1=0$ and Hypothesis
\ref{pert-fi}. Then $u$ satisfies the same estimate in (i)
 of Lemma  \ref{osc-u}. If Hypothesis
\ref{pert-fi} holds and we have,
 in addition, that $g \in L^1(0,1)$, then $u$ satisfies also
 the  estimate
  (ii) in Lemma  \ref{osc-u}. Moreover,  estimate
\rife{smallt} holds for $u$.
\end{lemma}
In order to explain the last statement, notice that if Hypothesis
\ref{pw-nl1-hol} holds with $c_0=c_1=0$, this means, roughly
speaking, that the nonlinear terms in the function $F$ have at most
linear growth with respect to $Du$. As mentioned in (iii), in this
case the conclusion of the above Lemma is much stronger since it
coincides with the conclusion of Lemma \ref{osc-u}. In particular,
we have conservation of the Lipschitz, or H\"older, continuity from
$u_0$ to $u(t)$ if $c_0=c_1=0$ (see also Remark \ref{FGP}).

 \proof
 We essentially follow the  proof of Lemma  \ref{osc-u}. For
$T'\leq T$, let us set
\begin{align} \label{dl1}
\Delta = \{ (t,x,y)\in (0,T')\times \RN\times \RN \}
\end{align}
and define the function
\begin{align} \label{phi1}
    \Phi_{\epsilon}(t,x,y)= u (t,x)-u (t,y)
  -   f(t,|x-y|)  -
\vep \, (\vfi(t,x)+\vfi(t,y))
  - \frac\vep{T' -t}\,, \;\; \epsilon >0,
\end{align}
where  $C_0$, $T'>0$   will  be chosen later and  where
$$
f(t,r)= (k_0+at) + (\beta+ bt) r^\alpha+ (\gamma+ct) (r+\tilde
f(r)),
$$
where $a,b,c, \beta, \gamma\geq 0$ and $\tilde f(r) \in
C^2(0,\infty)$ is   a nondecreasing concave function with $\tilde
f(0)=0$, to be fixed later. We  prove that, independently on $\vep$,
it holds
$$
\Phi_{\epsilon}(t,x,y)\leq 0, \quad
 (t,x,y)\in \Delta\,.
$$
Arguing by contradiction, as in Lemma \ref{osc-u}  we deduce
 that $\Phi_{\epsilon} $ has a
local (and global) positive maximum
  at some point
 $(\hat t, \hat x,\hat y )\in \overline \Delta$
and clearly  $\hat x\neq \hat y$ since the maximum is positive, and
$\hat t< T'$.
 Assuming that
 \be \label{gro3}
  \beta\geq k_\alpha\,,\quad \gamma\geq k_1,
 \ee
we deduce that the maximum cannot be reached at $t=0$  thanks to
\rife{osc-uo}. Since $u$ is a  solution,  proceeding as usual and
setting $\psi(t,x-y)= f(t,|x-y|)$, we end up with
$$
 \begin{array}{c}
  \frac{\varepsilon}{(T'- \hat t)^2} +
 (a+ b |\hat x-\hat y|^\alpha+c (|\hat x-\hat y|+ \tilde f(|\hat x-\hat y|))
 \\ + \vep\,(\partial_t \varphi(\hat t, \hat x)+\partial_t \varphi(\hat t, \hat y))+
 F (\hat t ,\hat x,  D\psi(\hat t, \hat x-\hat y) + \vep \,  D\varphi(\hat t, \hat x), X)
 \\
 \m
 \quad -F(\hat t ,\hat y,
D\psi (\hat t, \hat x-\hat y) -\vep \,  D\varphi(\hat t, \hat y),
Y)\leq h(\hat t,\hat x)-h(\hat t,\hat y),
\end{array}
$$
where $X$, $Y$ satisfy \rife{matri1} with $z= f(t, |x-y|)+\vep \,
(\varphi(t,x)+\varphi(t,y))
  + \frac\vep{T' -t}$.
  Therefore we obtain
\be \label{pp}
 \begin{array}{c}
 \frac{\epsilon}{(T'- \hat t)^2} +
 (a+ b |\hat x-\hat y|^\alpha+c (|\hat x-\hat y|+ \tilde
 f(|\hat x-\hat y|)) \\   +F (\hat t ,\hat x, D\psi(t,\hat x-\hat y) ,
X- \vep \,  D^2\varphi(\hat t, \hat x))
 \\
 \m
 \quad -F(\hat t ,\hat y,
 D\psi(t,\hat x-\hat y) , Y +\vep \,  D^2\varphi(\hat t, \hat y))\leq
h(\hat t,\hat x)-h(\hat t,\hat y) + {\mathcal I}_\vep(\varphi(\hat
t,\hat x))+ {\mathcal I}_\vep(\varphi(\hat t, \hat y)),
\end{array}
\ee where
$$
\begin{array}{c}
{\mathcal I}_\vep(\varphi(\hat t, \hat x))= - \vep\,\partial_t
\varphi(\hat t, \hat x)+ F (\hat t ,\hat x,  D\psi(t,\hat x-\hat y)
, X- \vep \,  D^2\varphi(\hat t, \hat x))
\\
\m \qquad  -F (\hat t ,\hat x,  D\psi(t,\hat x-\hat y) + \vep \,
D\varphi(\hat t, \hat x), X)
\end{array}
$$
and
$$
\begin{array}{c}
{\mathcal I}_\vep(\varphi(\hat t, \hat y))= - \vep\,\partial_t
\varphi(\hat t, \hat y)+ F(\hat t ,\hat y, D\psi(t,\hat x-\hat y)
-\vep \,  D\varphi(\hat t, \hat
y), Y) \\
\m \qquad -F(\hat t ,\hat y,  D\psi(t,\hat x-\hat y) , Y +\vep \,
D^2\varphi(\hat t, \hat y)).
\end{array}
$$
Since $f(t,\cdot)$ is increasing and concave, as in the proof of
Theorem \ref{pw-fully} we can use \rife{matri1} and Hypothesis
\ref{pw-nl1-hol} with $\mu =
 \frac{f'(\hat t, |\hat x- \hat y|)}{|\hat x - \hat y|}$
  (we have set $\partial_x f = f'$
 and $\partial_{xx}^2 f = f''$)
%(\beta+ bt) \alpha {\hat r}^{\alpha-2}+ (\gamma+ct) (1+\frac{\tilde
%f'(\hat r)} {\hat r})$ and $\hat r = \hat x-\hat y$)
and we get
$$
\begin{array}{c}
F (\hat t ,\hat x, D\psi(t,\hat x-\hat y) , X- \vep \,
D^2\varphi(\hat t, \hat x))
   -F(\hat t ,\hat y,
 D\psi(t,\hat x-\hat y) , Y +\vep \,  D^2\varphi(\hat t, \hat y))
\\
\m \geq - \la {\rm tr}\left(X- \vep \,
 D^2\varphi(\hat t, \hat x)-(Y +\vep \,  D^2\varphi(\hat t, \hat
y)) \right)- f'(\hat t, |\hat x-\hat y|) \, g(|\hat x- \hat y|)
\\
\m \quad -\left(f'(\hat t, |\hat x-\hat y|)\right)^2\left( c_0+ c_1
\left(f'(\hat t, |\hat x-\hat y|)|\hat x-\hat y|\right)^{q-1}\right)
- M_0-M_1|x-y|
 - \frac{\theta_{\epsilon} (\hat t, \hat x, \hat y)} {n}
 \, \eta(\hat t, \hat x, \hat y),
 \end{array}
 $$
 for some positive quantity
$\theta_{\epsilon} (\hat t, \hat x, \hat y)$. As in the proof of
Theorem \ref{pw-fully} we estimate
 $$
\begin{array}{c}{\rm tr}\left(X- \vep \,  D^2\varphi(\hat t, \hat x)-
(Y +\vep \,  D^2\varphi(\hat t, \hat y)) \right)\leq 4 \, f''(\hat
t, |\hat x-\hat y|) +
 \frac{2N \theta_{\epsilon} (\hat t, \hat x, \hat y)}{n}\,.
\end{array}
$$
Therefore, setting $\hat r= | \hat x-\hat y|$ and using the growth
of $h$,   we deduce \be\label{pre-abc2}
 \begin{array}{c}
 \frac{\epsilon}{(T'- \hat t)^2} +
 (a+ b \hat r^\alpha+c (\hat r+ \tilde
 f(\hat r)) \leq  4 \,\la  f''(\hat t,\hat r)+ f'(\hat t, \hat r) \, g(\hat r)
 \\
\m \quad + \left(f'(\hat t, \hat r)\right)^2\left( c_0+ c_1
\left(f'(\hat t,\hat r)\hat r\right)^{q-1}\right)  + M_0+M_1 \hat r
+ h_0+ h_\alpha \hat r^\alpha+ h_1 \hat r
 \\
\m +\frac{\theta_{\epsilon} (\hat t, \hat x, \hat y)} {n}
 \, \left(2N\la+ \eta(\hat t, \hat x, \hat y)\right)
+ {\mathcal I}_\vep(\varphi(\hat t,\hat x))+ {\mathcal
I}_\vep(\varphi(\hat t, \hat y)).
\end{array}
\ee We split henceforth the proof in the two cases (i) and (ii).

\vskip1em \noindent
{\bf (i)} Assume that $g(r)r\to 0$ as $r\to 0^+$
and that Hypothesis \ref{pert-fi-bis} holds. Thanks to this latter
assumption
 we have ${\mathcal I}_\vep(\varphi(\hat t,\hat x))$, $ {\mathcal
I}_\vep(\varphi(\hat t, \hat y))\leq 0$. We take $\tilde f =0$;
dropping the latter non-positive  terms and letting $n \to \infty$,
and using the precise form of $f(t,r)$, we find from \rife{pre-abc2}
%with  $b \ge c$ and $\beta
%\ge \gamma$
$$
\begin{array}{c}
 \frac\vep{(T'-\hat t)^2}+
(a+b\hat r^\alpha+ c{\hat r})
   \\
 \m
 \leq    (\beta+b \hat t)\,\left[4\la \alpha (\alpha-1)
 {\hat r}^{\alpha-2}+
 \alpha {\hat r}^{\alpha-1} g( {\hat r})\right]
 %(\beta + c \hat t)g(\hat r) +
+ (\gamma + c \hat t)g(\hat r)
  \\ \m  + \left(f'(\hat t, \hat r)\right)^2\left( c_0+ c_1 \left(f'(\hat t,\hat r)\hat r\right)^{q-1}\right)+ M_0  + h_0+ h_\alpha
{\hat r}^\alpha+( h_1+M_1) {\hat r}.
\end{array}
$$
Henceforth, if $c_0=c_1=0$, one can follow the proof of Lemma
\ref{osc-u} only replacing $h_0$ with $h_0+M_0$ and $h_1$ with $h_1+M_1$, and obtaining the
same kind of estimate. Assume instead that $c_0$ or $c_1$ are
positive. We take $b \ge c$ and $\beta\ge \gamma$; as in the proof
of Lemma \ref{osc-u}, we have that there exist $r_0<1$ and $L_0>0$
such that
$$
\begin{array}{c}
 (\beta+b \hat t)\,\left[4\la \alpha (\alpha-1)
 {\hat r}^{\alpha-2}+
 \alpha {\hat r}^{\alpha-1} g( {\hat r})\right] +
 (\gamma + c \hat t)g(\hat r) \\
 \m
 \leq
 -2\al(1-\al)\la \beta {\hat r}^{\al-2}\chi_{\{r<r_0\}}+
 (\beta+b \hat t)\al L_0\hat r^\alpha+ (\ga+ c\hat t)L_0 \hat r
 \end{array}
 $$
so that, if
we fix $T'= \frac1{2L_0}$, we have $L_0 \hat t \le 1/2$ and we deduce
$$
\begin{array}{c}
  \frac\vep{(T'-\hat t)^2}+
 a+ (\frac12 b-L_0 \beta ) {\hat r}^\alpha+ (\frac12 c-L_0 \gamma )
  {\hat r}
 \leq  - 2 \beta  \al(1-\al)\la \hat r^{\al-2}\chi_{\{r<r_0\}}  \\
\m \quad + \left(f'(\hat t, \hat r)\right)^2\left( c_0+ c_1
\left(f'(\hat t,\hat r)\hat r\right)^{q-1}\right) + M_0 + h_0+h_\alpha
\hat r^\alpha+( h_1+M_1) \hat r.
\end{array}
$$
We estimate now the superlinear term. If $\hat r\leq r_0$, we have
$f'(\hat t,\hat r)\hat r\leq C$,  where $C$ depends on  $r_0$,
$\beta$, $\gamma$, $b,c$, $L_0$ (since $T'$ depends on $L_0$). Then if
$\hat r \le r_0$
$$
\left(f'(\hat t, \hat r)\right)^2\left( c_0+ c_1 \left(f'(\hat
t,\hat r)\hat r\right)^{q-1}\right) \leq C_2 \hat r^{2\alpha-2},
$$
for some $C_2=C_2(r_0, \beta, \ga, b,c, L_0, c_0, c_1)$;
 so, for some  $r_1<r_0$, we have
$$
\left(f'(\hat t, \hat r) \right)^2\left( c_0+ c_1 \left(f'(\hat
t,\hat r)\hat r\right)^{q-1}\right) \leq\beta  \al(1-\al)\la \hat
r^{\al-2} \quad \hbox{ if $ \hat r\leq r_1$}\,.
$$
If instead  $\hat r >r_1$, we have
$$
\left(f'(\hat t, \hat r)\right)^2 \left( c_0+ c_1 \left(f'(\hat
t,\hat r)\hat r\right)^{q-1}\right) \leq C \hat r^{q-1}
$$
for a possibly different constant $C$; hence, since $q<2$, Young's
inequality implies
$$
\left(f'(\hat t, \hat r)\right)^2 \left( c_0+ c_1 \left(f'(\hat
t,\hat r)\hat r\right)^{q-1}\right) \leq \frac14 c \hat r+ K
$$
where $K= K(\alpha, r_0, \beta, \ga, L_0, b, c, c_0, c_1)$. We conclude that
$$
\begin{array}{c}
  \frac\vep{(T'-\hat t)^2}+
 a+ (\frac12 b-L_0 \beta ) {\hat r}^\alpha+ (\frac14 c-L_0 \gamma )
  {\hat r}
 \leq  K+ M_0 +
h_0+h_\alpha \hat r^\alpha+ ( h_1+M_1) \hat r\,.
\end{array}
$$
Here one chooses $c\geq  4(L_0 \gamma +h_1+M_1)$, $b\geq 2(L_0 \beta
+h_\alpha)$ and, lastly,  $a\geq K+M_0+h_0$ and the conclusion
follows. The estimate is therefore proved in $[0,T']$ for $T'= \frac
1{2L_0}$. Since $L_0$ only depends on $\la, g, \alpha$, the argument
can be iterated and yields the global estimate   on $[0,T]$.

\smallskip \noindent {\bf (ii)}   Assume that $g(r)\in L^1(0,1)$ and that
 Hypothesis \ref{pert-fi} holds. Up to using Young's inequality, we can suppose that
$k_\alpha=0$; in particular, this allows us to take $b=\beta=0$,
so that  \rife{pre-abc2} implies
$$
\begin{array}{c}
 \frac{\epsilon}{(T'- \hat t)^2} +
 (a+c (\hat r+ \tilde
 f(\hat r)) \leq  (\ga+ c \hat t)[ 4 \,\la  \tilde  f''(\hat r)+
   g(\hat r)(1+ \tilde f'(\hat r)) ]
 \\
\m \quad + \left((\ga + c\hat  t) (1+ \tilde f'(\hat
r))\right)^2\left( c_0+ c_1 \left((\ga + c\hat  t) (\hat r+ \tilde
f'(\hat r)\hat r)\right)^{q-1}\right)  + M_0 + h_0+ h_\alpha \hat
r^\alpha+ (h_1+M_1) \hat r
 \\
\m +\frac{\theta_{\epsilon} (\hat t, \hat x, \hat y)} {n}
 \, \left(2N\la+ \eta(\hat t, \hat x, \hat y)\right)
+ {\mathcal I}_\vep(\varphi(\hat t,\hat x))+ {\mathcal
I}_\vep(\varphi(\hat t, \hat y)).
\end{array}
$$
We proceed then as in
 Lemma \ref{osc-u},  choosing,
  $\tilde f(r)=3\int_0^r \left(e^{\frac{\tilde
G(\xi)}{4\la}}-1\right)d\xi$, where $\tilde G(\xi)= \int_\xi^\infty
\tilde g(\tau)d\tau$. In the same way we obtain, for some constant
$L_0=L_0(g, \lambda)>0$ (cf. \rife{12} and \rife{aggiunta}):
$$
\begin{array}{c}
 \frac{\epsilon}{(T'- \hat t)^2} +
 (a+c (\hat r+ \tilde
 f(\hat r)) \leq  (\ga+ c\hat t)L_0\hat r
 + (M_0+ h_0+ h_\alpha + h_1+M_1 )\hat r
 \\
\m \quad + \left((\ga + c\hat t) (1+ \tilde f'(\hat
r))\right)^2\left( c_0+ c_1 \left((\ga + c\hat t) (r+ \tilde f'(\hat
r)\hat r)\right)^{q-1}\right)
 \\
\m +\frac{\theta_{\epsilon} (\hat t, \hat x, \hat y)} {n}
 \, \left(2N\la+ \eta(\hat t, \hat x, \hat y)\right)
+ {\mathcal I}_\vep(\varphi(\hat t,\hat x))+ {\mathcal
I}_\vep(\varphi(\hat t, \hat y)),
\end{array}
$$
and then, choosing $T'= \frac 1{2L_0}$ and since $t\leq T'$ and
$\tilde f \ge 0,$
$$
\begin{array}{c}
 \frac{\epsilon}{(T'- \hat t)^2} +
 (a+(\frac 12 c -\ga L_0) \hat r ) \leq    (M_0+ h_0+ h_\alpha + h_1 +M_1)\hat r
 \\
\m \quad + \left((\ga + ct) (1+ \tilde f'(\hat r))\right)^2\left(
c_0+ c_1 \left((\ga + ct) (r+ \tilde f'(\hat r)\hat
r)\right)^{q-1}\right)
 \\
\m +\frac{\theta_{\epsilon} (\hat t, \hat x, \hat y)} {n}
 \, \left(2N\la+ \eta(\hat t, \hat x, \hat y)\right)
+ {\mathcal I}_\vep(\varphi(\hat t,\hat x))+ {\mathcal
I}_\vep(\varphi(\hat t, \hat y)).
\end{array}
$$
Since  there exists $\tilde l>0$ such that $\tilde f'(r)\leq \tilde
l$ for every $r>0$, we have $|D \psi(\hat t,\hat x-\hat y)|= (\ga+
c\hat t) (1+\tilde f'(\hat r))\leq L$, where $L= (\ga+ \frac
c{2L_0})(1+\tilde l)$. Therefore, we use Hypothesis \ref{pert-fi}
with such $L$
and we deduce that last two terms are non-positive and can be
neglected; letting also $n\to \infty$ we get
$$
\begin{array}{c}
 \frac{\epsilon}{(T'- \hat t)^2} +
 (a+(\frac 12 c -\ga L_0) \hat r ) \leq
 (M_0+ h_0+ h_\alpha + h_1 +M_1)\hat r
 \\
\m \quad +\left((\ga + c\hat  t) (1+ \tilde f'(\hat
r))\right)^2\left( c_0+ c_1 \left((\ga + c\hat  t) (r+ \tilde
f'(\hat r)\hat r)\right)^{q-1}\right) \,.
\end{array}
$$
If $c_0=c_1=0$, we take $a=0$ and we obtain the same conclusion as
in Lemma \ref{osc-u}, with $k_\alpha=0$. If $c_0$ or $c_1$ are
positive, we observe that, being   $q<2$, we
have
$$
\left((\ga + c\hat t) (1+ \tilde f'(\hat r))\right)^2\left( c_0+ c_1
\left((\ga + c\hat  t) (\hat  r+ \tilde f'(\hat r)\hat
r)\right)^{q-1}\right) \leq K+ \frac 14 c \hat r
$$
for some constant $K= K(\ga, c, L_0, \tilde f, c_0, c_1)$.
Then we obtain
$$
\begin{array}{c}
 \frac{\epsilon}{(T'- \hat t)^2} +
 (a+(\frac 14 c -\ga L_0) \hat r )
 \leq    (M_0+ h_0+ h_\alpha + h_1+M_1 )\hat r + K,
\end{array}
$$
and we conclude as before choosing $a$ and $c$ sufficiently large.

\smallskip \noindent
{\bf (iii)} Since $c_0=c_1=0$, it follows exactly as in the  proof
of Lemma \ref{osc-u}. \qed

As a consequence of Lemma \ref{osc2},  we have now
conditions under which solutions have bounded oscillation and
Theorem \ref{pw-fully} can be applied.  We deduce a complete regularity result.

\begin{theorem} \label{finnonl}  Assume  that $h$ and $u_0$ have bounded
oscillation  in $Q_T$ and in $\R^N$, respectively.
 There exists $L>0$   as in Lemma \ref{osc2}
 such that if
  $u\in C(\bar Q_T)$ is a  viscosity solution of \rife{cauful} and $u$ is $o_{\infty}(\varphi_L)$ in $\bar Q_T$ (see Hypothesis \ref{pert-fi}) then
%is also   o$(\varphi)$ in $\bar Q_T$.
we have the following
statements.

\smallskip  i) If Hypotheses  \ref{pw-nl1} and \ref{pert-fi}
 hold, then
  $ u(t)  \in W^{1,\infty} (\R^N)$, for all $t \in
(0,T)$,  and
 \rife{general1} holds with
 a constant $C$ depending  on ${{\rm osc}}(u_0)$,  ${{\rm osc}}_{(0\, , \, T)}(h)$,
  $\lambda, g, M, q$, $c_0$, $c_1$, $M_0$, $M_1$, $T$
  and  $\omega$.

\smallskip ii) If Hypotheses \ref{pw-nl1-hol} and \ref{pert-fi-bis} hold,
then  $u(t) $
 is $\alpha$-Holder continuous on $\R^N$,
 for any $\alpha \in (0,1)$,
   and  \rife{ci11} holds with
  a constant $C$ depending on
  $\alpha$,  ${{\rm osc}}(u_0)$,  ${{\rm osc}}_{(0\, , \, T)}(h)$,
  $\lambda, g, M, q$, $c_0$,  $c_1$, $M_0$, $M_1$ and $T$.
\end{theorem}
\proof  Applying Lemma \ref{osc2} we deduce that $u$ has bounded
 oscillation and its oscillation is estimated in terms of   ${{\rm
osc}}(u_0)$,  ${{\rm osc}}_{(0\, , \, T)}(h)$, $T$ besides  the
usual constants. Then, we apply either Theorem \ref{pw-fully} or
Proposition \ref{hol-fully} to conclude the Lipschitz, respectively
H\"older, estimate. \qed
\begin{remark}\label{omega} {\rm
We stress that statement (i) still holds true if the function
$\omega$ in Hypothesis \ref{pw-nl1} is only assumed to be bounded,
provided the  function $g(s)$ satisfies $g(s)s\to 0$ as $s \to 0^+$
(or even simply $\limsup_{s \to 0^+} s g(s)< 4 \lambda$) and
Hypothesis \ref{pert-fi} is replaced by the stronger   Hypothesis
\ref{pert-fi-bis}. In this case one first prove that $u(t) $
 is $\alpha$-Holder continuous for some
 $\alpha\in (0,1)$,
and then uses this information to estimate the right-hand side of
\rife{omega-1} in the proof of Theorem \ref{pw-fully}. Notice that
this preliminary H\"older estimate can be obtained by only requiring
that $\limsup_{s \to 0^+} s g(s)< 4 \lambda$, as in Proposition
\ref{hol-fully}.}
\end{remark}

The following corollary provides a nonlinear version of Theorem
\ref{grow}. It can be applied, for instance,  to Bellman-Isaacs type equations.

\begin{corollary} \label{fin2}
 Assume that $u_0$, $h$ satisfy \rife{osc-uo} and \rife{osc-h},
respectively. Assume  Hypothesis \ref{pw-nl1} with $c_0=c_1=0$ and
  $g$ satisfying  $sg(s) \to 0$ as $s \to 0^+$. Assume also  Hypothesis
\ref{pert-fi}.

There exists $L>0$   as in Lemma \ref{osc2}
 such that if
  $u\in C(\bar Q_T)$ is a  viscosity solution of \rife{cauful} and $u$ is $o_{\infty}(\varphi_L)$ in $\bar Q_T$ (see Hypothesis \ref{pert-fi})
 % $u$ is o$(\varphi)$ in $\bar Q_T$.
 then $ u(t)\in W^{1,\infty}(\R^N)$, $t\in (0,T)
$, and there exist $c= c_T(T, \lambda, g, \alpha)>0$ such that, for $t\in (0,T)$,
 \be \label{grad1-nl}
 \begin{array}{c}
 \|   D u(t)\|_\infty \leq  c_T\left\{ \frac
{k_0} {\sqrt{t \wedge 1}}\, + \frac { k_{\alpha}} {{(t \wedge
1)}^{1/2 - \alpha/2}} + k_1\right.
\\
\m
\qquad
\qquad \left. + (\sqrt{t \wedge 1}) \, (h_0 +M_0+  h_{\alpha}(t \wedge 1)^{\alpha/2}+ (h_1+M_1)\sqrt{t \wedge 1})\right\}.
\end{array}
 \ee
\end{corollary}
\begin{proof} We already know by Lemma \ref{osc2}
 that $u$ has  bounded oscillation. We
follow the proof of Theorem \ref{pw-fully} and we arrive at
\rife{llip}. Then  using  (iii) in Lemma \ref{osc2},  we can
conclude similarly to the proof of Theorem \ref{grow}.
\end{proof}

\subsection{Local Lipschitz continuity}

The approach developed so far can also provide estimates and
regularity results concerning the local Lipschitz continuity. We
give an example in the following theorem.
\begin{theorem}\label{liploc}  Assume that Hypothesis
\ref{pw-nl} holds true and in addition that, for every compact set
$S\subset \R^N$  we have \be\label{loc-cont}
\begin{array}{c}
 |F(t,x,p,X)- F(t,x,q,Y)|\leq C_S  ( 1+(|p|^m+|q|^m ) )\{ |p-q| +   \|X-Y\|\},
 \\
 \m
\qquad\qquad  x\in S\,,\; \,p,q\in \R^N \,,\, t\in
(0,T)\,.
\end{array}
\ee for some constants $C_S>0$ and $m<2$.

 Let  $u\in C(Q_T)$
  be a
 viscosity solution of  equation
 \rife{fnl} with $h \in C(Q_T)$.  Then $  u(t)  \in
 W^{1,\infty}_{\rm loc}(\RN)$ for every $t\in (0,T) $ and,     for any
ball  $B_R(x_0)$,  there exists $M_0$  such that
$$
\|   D u(t)\|_{L^{\infty}(B_R(x_0))} \leq
 \frac {M_0}{\sqrt{t  \, \wedge
1} }\,,
$$
where $M_0= (g, \lambda, q, \omega, c_0, c_1, M, x_0,  N, R,
\|u\|_{L^\infty((\frac t2,T\wedge \frac{3}{2}t)\times B_{2R}(x_0))}
, \|h\|_{L^\infty((\frac t2,T\wedge \frac{3}{2}t)\times
B_{2R}(x_0))})$.
\end{theorem}
\proof Let us fix $x_0\in \R^N$, $t_0>0$, $\de>0$, $r>0$ and
consider the set
$$
\Delta =
 \Big \{ (t,x,y)\in (0,T)\times B_r(x_0)\times  B_r(x_0)\,:\,
|x-y|<\de\,,\,\,  \frac{t_0}2 < t < \frac32 t_0 \wedge T \Big\}
$$
and the function
$$
 \Phi_{\epsilon}(t,x,y)= u (t,x)-u (t,y)
  -  K \psi(x-y) - L |x-x_0|^2
  - C_0 (t-t_0)^2\, - \frac{\epsilon}{T-t},
$$
where $K$, $C_0$, $L$ will   be chosen later and
$\psi(x-y)=f(|x-y|)$ with $f$ as in \rife{effe}.  As usual, we
claim that
 \be\label{goal}
\Phi_{\epsilon}(t,x,y)\leq 0, \qquad
 (t,x,y)\in \Delta\,,
 \ee
 and, arguing by contradiction, we suppose that
 \be
\sup\limits_{\Delta}\Phi_{\epsilon}(t,x,y)>0\,. \ee
 In the sequel, in order  to simplify the notation, we write
$\|u\|_\infty:=\|u\|_{L^\infty((\frac{t_0}2, \frac{3}{2}t_0 \wedge
T) \times B_r(x_0))}$ and
$\|h\|_\infty:=\|h\|_{L^\infty((\frac{t_0}2, \frac{3}{2}t_0 \wedge
T) \times B_r(x_0))}$. Since
$$ \Phi_{\epsilon}(t,x,y)<  2\|u\|_\infty
  -  K f(|x-y|)  -L |x-x_0|^2- C_0(t-t_0)^2\,,
$$
then if we choose
$$
L\geq \frac{ 2\|u\|_\infty}{r^2}\,,\qquad C_0\geq
\frac{8\|u\|_\infty}{t_0^2}\,,\qquad K\geq
\frac{2\|u\|_\infty}{f(\de)}
$$
we can exclude that the maximum of $\Phi_{\epsilon}$ in $\overline
\Delta$ be attained when $t=\frac{t_0}2$, $t=\frac32 t_0 \wedge T$,
or when $|x-y|=\de$ or when $|x-x_0|=r$. If $|y-x_0|=r$, then
$|x-x_0|\geq r-\de$ and choosing $r\geq 2\de$ and $L\geq \frac{
8\|u\|_\infty}{r^2}$ we also exclude that the maximum be attained
when $|y-x_0|=r$.

\smallskip On account of the above choices, we deduce that the maximum is attained inside
$\Delta$. Moreover, this maximum being positive, it cannot be
attained when $x=y$. Let $(\hat t, \hat x, \hat y)$ be the point in
which the maximum is achieved.
We proceed now as in the proof of Theorem \ref{pw-nl} obtaining
\be\label{l1}
 \begin{array}{c}
  \frac{\epsilon}{(T- \hat t)^2} +
 2C_0(\hat t-t_0)+  F (\hat t ,\hat x, K D\psi(\hat x-\hat y) +2L(\hat x-x_0), X_n)
 \\
 \m
 \quad -F(\hat t ,\hat y,
K D\psi(\hat x-\hat y) , Y_n)\leq h(\hat t,\hat x)-h(\hat t,\hat y),
\end{array}
\ee
where $\psi(\cdot)= f(|\cdot|)$, and where $X_n$, $Y_n$ satisfy
\be \label{cf5}
\begin{pmatrix}
 X_n-  2LI   & 0\\
\noalign{\medskip} 0&  - Y_n
\end{pmatrix}
\leq  K \begin{pmatrix}
 D^2\psi(\hat x-\hat y) & -D^2\psi(\hat x-\hat y)\\
\noalign{\medskip} - D^2\psi(\hat x-\hat y)&  D^2\psi(\hat x-\hat y)
\end{pmatrix} + \frac{1}{n}
(D^2 z(\hat t, \hat x, \hat y))^2,
\ee
where $D^2 z= D^2 \left(K\psi(  x-  y) + L |  x-x_0|^2 \right)$.
Using \rife{the}   and Hypothesis \ref{pw-nl} we have
\be\label{l1bis}
\begin{array}{c}
F (\hat t ,\hat x, K\,D\psi(\hat x-\hat y) , X_n- 2LI )
   -F(\hat t ,\hat y,
K\,D\psi(\hat x-\hat y) , Y_n ) \geq - \la {\rm tr}\left(X_n- 2L I-Y_n
\right)
\\
\m \quad -K  f'(|\hat x-\hat y|) \, g(|\hat x- \hat y|) - \left(K
f'(|\hat x-\hat y|)\right)^{2}\, [c_0 + c_1\big( K f'(|\hat x-\hat
y|)\, |\hat x-\hat y| \big)^{q-1}] \omega(|\hat x-\hat y|) \\ \m - M
- \frac{\theta_{\epsilon} (\hat t, \hat x, \hat y)} {n}
 \, \eta(\hat t, \hat x, \hat y).
\end{array}
\ee
Note that, in consequence of \rife{l1} and assumption \rife{loc-cont},  the left hand side of the previous inequality is bounded uniformly with respect to $n$. Thus we can apply the compactness argument of  Remark  \ref{rit},  which allows us to pass to the limit as $n$ go to infinity, replacing  $X_n$ and $Y_n$ with  matrices $X$, $Y$ such that  \eqref{cf5} holds without the term $\frac{1}{n}
(D^2 z(\hat t, \hat x, \hat y))^2$. Then we have, from \rife{l1} and \rife{l1bis}
$$
\begin{array}{c}
  \frac{\epsilon}{(T- \hat t)^2} +
 2C_0(\hat t-t_0) - \la {\rm tr}\left(X- 2L I-Y
\right)
\\
\m \quad  \leq K  f'(|\hat x-\hat y|) \, g(|\hat x- \hat y|)+ \left(K
f'(|\hat x-\hat y|)\right)^{2}\, [c_0 + c_1\big( K f'(|\hat x-\hat
y|)\, |\hat x-\hat y| \big)^{q-1}] \omega(|\hat x-\hat y|) \\ \m + M  + h(\hat t,\hat x)-h(\hat t,\hat y) + F (\hat t ,\hat x, K D\psi(\hat x-\hat y)  , X-2LI) -    F
(\hat t ,\hat x, K D\psi(\hat x-\hat y) +2L(\hat x-x_0), X) \,.
\end{array}
$$
On the other hand, since $\max \Phi_\vep>0$, we deduce that
$$
K f(|\hat x-\hat y|) \leq u(\hat t, \hat x)- u(\hat t,\hat y) \leq
2\|u\|_\infty,
$$
hence we get, since $f$ is concave,
$
K f'(|\hat x-\hat y|)\, |\hat x-\hat y| \leq 2\|u\|_\infty\,.
$
Therefore we obtain
$$
\begin{array}{c}
  \frac{\epsilon}{(T- \hat t)^2} +
 2C_0(\hat t-t_0) - \la {\rm tr}\left(X- 2L I-Y
\right)
\\
\m \quad  \leq K  f'(|\hat x-\hat y|) \, g(|\hat x- \hat y|)+ K^2 \left(f'(|\hat x-\hat y|)\right)^2 [ c_0 +
c_1\left(2\|u\|_\infty \right)^{q-1}]\omega(|\hat x-\hat y|)
\\
\m  + M  + 2 \|h\|_\infty   + F (\hat t ,\hat x, K D\psi(\hat x-\hat y)  , X-2LI) -    F
(\hat t ,\hat x, K D\psi(\hat x-\hat y) +2L(\hat x-x_0), X) \,.
\end{array}
$$
By using Proposition \ref{refe},   as in the proof of Theorem \ref{pw-nl}
we estimate ${\rm tr}\left(X- 2L I-Y \right)$ and we conclude that
$$
\begin{array}{c}
  \frac{\epsilon}{(T- \hat t)^2} +
 2C_0(\hat t-t_0)\leq 4\la K f''(|\hat x-\hat y|)
 +K f'(|\hat x-\hat y|) \, g(|\hat x- \hat y|)
 \\
 \m
+  K^2 \left(f'(|\hat x-\hat y|)\right)^2
\,  [ c_0 + c_1\left(2\|u\|_\infty \right)^{q-1}]\omega(|\hat x-\hat
y|) + M +
 2 \| h\|_{\infty}
\\
\m +  F (\hat t ,\hat x, K D\psi(\hat x-\hat y)  , X-2LI) -    F
(\hat t ,\hat x, K D\psi(\hat x-\hat y) +2L(\hat x-x_0), X)\,.
\end{array}
$$
Using \rife{ode1}
 we get
 \be\label{l2}
\begin{array}{c}
  \frac{\epsilon}{(T- \hat t)^2} +
 2C_0(\hat t-t_0)\leq - K + c_\la K
 [ c_0 +
c_1\left(2\|u\|_\infty \right)^{q-1}]\omega(|\hat x-\hat y|) + M +
2\|h\|_\infty \\
\m +  F (\hat t ,\hat x, K D\psi(\hat x-\hat y)  , X-2LI) -    F
(\hat t ,\hat x, K D\psi(\hat x-\hat y) +2L(\hat x-x_0), X)\,.
\end{array}
\ee
We fix henceforth $L= \frac{ 8\|u\|_\infty}{r^2}$,
$C_0=\frac{8\|u\|_\infty}{t_0^2}$ and $ K=\frac{\kappa
\|u\|_\infty}{\de^2}$ for some $\kappa = \kappa (g, \lambda)$ so
that $ K\geq \frac{2\|u\|_\infty}{f(\de)}$. We will later choose
$\de$ suitably small (this implies in turn that  $K$ is sufficiently
large).
The last term in \rife{l2} is estimated using assumption
\rife{loc-cont} with  $S= \overline B_r(x_0)$. Since $|KD\psi(\hat
x-\hat y)|$ $ \leq C\frac{\|u\|_{\infty}}\de $, with $C= C(\lambda,
g)$, we deduce that there exists some constant $c= c(N, L , \lambda,
g, S)$ such that
$$
\begin{array}{c}
F (\hat t ,\hat x, K D\psi(\hat x-\hat y)  , X-2LI) -    F (\hat t
,\hat x, K D\psi(\hat x-\hat y) +2L(\hat x-x_0), X) \leq
\\
\m \quad \leq c\left\{ (1+\left(\frac{\|u\|_{\infty}}\de+
Lr\right)^m L( |\hat x-x_0|  + 1)  \right\}\,.
\end{array}
$$
Recalling that $K\simeq  \frac{\|u\|_{\infty}}{\de^2}$, and $ L$ and
$ |\hat x-x_0|$ are only estimated in terms of $r$, we get
$$
F (\hat t ,\hat x, K D\psi(\hat x-\hat y)  , X-2LI) -    F (\hat t
,\hat x, K D\psi(\hat x-\hat y) +2L(\hat x-x_0), X) \leq
C(\|u\|_\infty, r) K^{\frac m2}.
$$
We obtain then from \rife{l2}
$$
\begin{array}{c}
  \frac{\epsilon}{(T- \hat t)^2} +
 2C_0(\hat t-t_0)\leq - K + c_\la
K\left(2\|u\|_\infty \right)^{q-1}\omega(|\hat x-\hat y|) + M +
2\|h\|_\infty
 +
 C(\|u\|_\infty, r) K^{\frac m2}.
 \end{array}
$$
Then, since $m<2$ and $\omega(0)=0$, we choose
$\de$ small (depending also on $r$, $\| h\|_{\infty}$ and
$\|u\|_\infty$) so that
$$
 \frac{\epsilon}{(T- \hat t)^2} +
 2C_0(\hat t-t_0)\leq - \frac K2
$$
and we conclude taking $\de$ eventually smaller so that $K> 2C_0
t_0$ (i.e., $\delta  \le C_3 \sqrt{t_0 \wedge 1}$).
% for some constant
%$C_3 = C_3 (\omega_{0, \delta}(u)$,
% $\lambda, g, M, q$, $c_0 $, $c_1 $).
 In this way we have proved \rife{goal}, which
implies, if we take $x=x_0$ and $t=t_0$   (and since $Kf(|x-y|)\leq
Kf'(0)|x-y|\leq \tilde c\frac{\|u\|_\infty}{\de}|x-y|$)
$$
u(t_0,x_0)-u(t_0,y)\leq   \tilde c |x_0-y| ,\;  \; \, |x_0-y|<\de\,.
 $$
% Possibly replacing
%$\tilde c $ with $\tilde c \vee 2$ the same inequality also holds if
%$y\in B_r(x_0)$ and $|y-x_0| \ge \de$.
%
%\smallskip
If we take now $x$, $y$ belonging to
some ball $B_R(x_0)$, then
%$y\in B_{2R}(x)\subset B_{3R}(x_0)$ and
we easily extend a similar estimate to $x$, $y$.
%Therefore, in particular,
%$$
%|u(t_0,x)-u(t_0,y)|\leq C(q_{3R}+  b_{3R} + h_{3R}\sqrt{t_0}\wedge
%1\wedge R) \frac{\|u\|_{L^\infty((\frac{t_0}2, \frac{3}{2}t_0 \wedge
%T) \times B_{3R}(x_0))} } {\sqrt{t_0}\wedge 1\wedge R}\, |x-y|.
%$$
\qed

\section{Examples and  applications}

In this section we consider examples of operators to which the
previous results can be applied, in particular checking the
Hypotheses \ref{pw-nl} and \ref{pert-fi}.

\vskip 1mm

\subsection{Bellman-Isaacs operators} Let us consider  the
case of Bellman-Isaacs equations appearing  in  stochastic control
problems or game theory (see \cite{Kr80}, \cite{Lions}, \cite{Fle-Sou}, \cite{Fle-Soner}, \cite{DaLio-Ley}, \cite{K} and the
references therein) where, for instance,
$$
F(t,x,Du, D^2 u)=
\inf\limits_{\beta\in \mathcal B}\sup\limits_{\alpha\in \mathcal A}\, \left\{-{\rm
tr}\left(q_{\alpha,\beta}(t,x)D^2u\right)-
 b_{\alpha,\beta}(t,x)\cdot   D u  -
f_{\alpha,\beta}(t,x)\right\}\,.
$$
As  a preliminary assumption, we suppose that $F(t,x,p,X)$ is finite for every $(t,x)\in Q_T$, $p\in \R^N$, $X\in {\mathcal S}_N$. For example, this is certainly true if, for some $\beta\in {\mathcal B}$,
$$
 %\inf\limits_{\beta\in \mathcal B}\,
    \sup\limits_{\alpha\in
 \mathcal A}\,
 \left\{{\rm tr}\left(q_{\alpha,\beta}(t,x) \right)
 +  |b_{\alpha,\beta}(t,x)| + |f_{\alpha,\beta}(t,x)| \right \}
  < \infty\,.
  $$
We  show now that   Hypotheses \ref{pw-nl} and \ref{pert-fi} are
satisfied provided that the following conditions hold.

\smallskip
{\rm \hh
(i) $q_{\al,\beta}(t,x)= \la\,I+\sigma_{\al,\beta}(t,x)^2 $,
 for some $\lambda>0$, where the coefficients
  $\sigma_{\al,\beta}(t,x)$ and
$b_{\alpha,\beta}(t,x)$ are continuous on $  Q_T$, uniformly in
$\alpha \in  {\mathcal A}$ and $\beta \in {\mathcal B}$, and
 satisfy \rife{pw},  uniformly in
 $\al\in {\mathcal A}$, $\beta\in \mathcal B$.
 \hh
(ii)  $f_{\al,\beta}$ are
 continuous and have bounded
 oscillation on $\bar Q_T$ (uniformly in
  $\al $ and $\beta$).
\hh
(iii) For any $(t,x) \in  Q_T$,  we have
$$
 \sup\limits_{\beta\in \mathcal B}\,\sup\limits_{\alpha\in
 \mathcal A}\,
 \left\{{\rm tr}\left(q_{\alpha,\beta}(t,x) \right)
\right \}
  < \infty,
  $$
 \hh
 (iv)  There exists $\varphi \in C^{1,2}(\bar Q_T)  $ such that
$\varphi\to + \infty$ as $|x|\to \infty$ (uniformly in $[0,T]$) and
$$
\partial_t \varphi +  \left\{ - {\rm tr}\left(q_{\al,\beta}
(t,x)D^2\varphi\right)- b_{\alpha,\beta}(t,x)\cdot   D \varphi
\right\}\geq 0 \quad \hbox{in $Q_T$}, \; \mbox{for every $\alpha\in
\mathcal A, \beta\in \mathcal B$.}
$$
To check Hypothesis \ref{pw-nl},
 we multiply
the matrix inequality \rife{ine} by
$$
\begin{pmatrix}
 \sigma_{\al,\beta}(t,x)^2  & \sigma_{\al,\beta}(t,x)\sigma_{\al,\beta} (t,y)    \\
\noalign{\medskip} \sigma_{\al,\beta}(t,y)\sigma_{\al,\beta}(t,x) &
\sigma_{\al,\beta}(t,y)^2
\end{pmatrix},
$$
then, taking traces, we deduce
\be \label{llo}
-{\rm tr}\left((q_{\al,\beta}(t,x) X-q_{\al,\beta}(t,y)Y\right) \ee
$$\geq -\la{\rm tr}\left(X-Y\right)- \mu {\rm tr}\left(
(\sigma_{\al,\beta}(t,x)  -\sigma_{\al,\beta}(t,y))^2  \right)
 - \nu {\rm tr}\left(
\sigma_{\al,\beta}^2 (t,x)  + \sigma_{\al,\beta}^2 (t,y) \right) \quad
$$$$\geq  -\la{\rm tr}\left(X-Y\right)- \mu \|(\sigma_{\al,\beta}(t,x)
-\sigma_{\al,\beta}(t,y)\|^2 - \nu {\rm tr}\left( q_{\al,\beta} (t,x)  + q_{\al,\beta} (t,y)
- 2 \lambda I \right) ,
$$
for every $\al\in \mathcal A$, $\beta\in \mathcal B$. Therefore
$$
\begin{array}{c}
\left\{-{\rm tr}\left(q_{\al,\beta} (t,x) X\right)- \mu b_{\alpha,\beta}(t,x)\cdot
(x-y)  - f_{\alpha,\beta}(t,x)\right\} \geq
\\
\m \geq \left\{-{\rm tr}\left(q_{\al,\beta}(t,y)Y\right)- \mu
b_{\alpha,\beta}(t,y)\cdot (x-y)  - f_{\alpha,\beta}(t,y)\right\}
\\
\m
 -
\la{\rm tr}\left(X-Y\right)-\mu \left[ \|(\sigma_{\al,\beta}(t,x)
-\sigma_{\al,\beta}(t,y)\|^2+ (b_{\alpha,\beta}(t,x)-b_{\alpha,\beta}(t,y))\cdot (x-y) \right]
\\ \m - (f_{\alpha,\beta}(t,x)-f_{\alpha,\beta}(t,y))  - \nu \eta(t,x,y)\end{array}
$$
where
$$
  \eta(t,x,y) = \sup\limits_{\beta\in \mathcal B}\sup\limits_{\alpha\in \mathcal A}\big\{
    {\rm tr}(
q_{\al,\beta} (t,x))   + {\rm tr}(q_{\al,\beta} (t,y)) \big\} -
2 \lambda  N.
 $$
 Note that $\eta(t,x,y)$ is finite thanks to (iii).
Thus we get
$$
\begin{array}{c}
\left\{-{\rm tr}\left(q_{\al,\beta} (t,x) X\right)- \mu b_{\alpha,\beta}(t,x)\cdot
(x-y)  - f_{\alpha,\beta}(t,x)\right\} \geq
\\
\m \geq \left\{-{\rm tr}\left(q_{\al,\beta} (t,y)Y\right)- \mu
b_{\alpha,\beta}(t,y)\cdot (x-y)  - f_{\alpha,\beta}(t,y)\right\}
\\
\m
 -\la{\rm tr}\left(X-Y\right)-\mu |x-y|g(|x-y|)
-M - \nu \eta(t,x,y),
\end{array}
$$
for every $  x,y \in
 \R^N$, $|x-y| \le 1$, $t \in (0,T)$.

Taking the $\sup$ on $\alpha$ and the $\inf$ on $\beta$ implies
 that
   $F$
satisfies Hypothesis \ref{pw-nl}. In order to see
 that  Hypothesis \ref{pert-fi} is
satisfied, we use (iv) which implies
$$
\begin{array}{c}
\left\{- {\rm tr}\left(q_{\al,\beta}(t,x)(X+\vep D^2\varphi\right)-
b_{\alpha,\beta}(t,x)\cdot  (p+\vep  D \varphi)  -
f_{\alpha,\beta}(t,x)\right\}
\\
\m =  \left\{-{\rm tr}\left(q_{\al,\beta}(t,x)X\right)-
b_{\alpha,\beta}(t,x)\cdot  p  - f_{\alpha,\beta}(t,x)\right\}
\\
\m -\vep \left\{{\rm tr}\left(q_{\al,\beta}(t,x)D^2\varphi\right)+
b_{\alpha,\beta}(t,x)\cdot   D \varphi\right\}
\\
\m \quad \geq \left\{-{\rm tr}\left(q_{\al,\beta}(t,x)X\right)-
b_{\alpha,\beta}(t,x)\cdot  p  - f_{\alpha,\beta}(t,x)\right\}- \vep
\partial_t \varphi
\end{array}
$$
and taking $\sup_\alpha$ and $\inf_\beta$ on both sides we deduce that Hypothesis \ref{pert-fi} is satisfied.

\smallskip
Let us note that a sufficient condition which implies  (iv)
 (hence Hypothesis \ref{pert-fi}) is that
  there exists $C\ge 0$ such that
\be \label{NN}
{\rm tr}\left(q_{\al,\beta}(t,x) \right)+ b_{\alpha,\beta}(t,x)\cdot x
%\|q_{\al,\beta}(t,x) \| + |b_{\alpha,\beta}(t,x)|^2 + |f_{\alpha,\beta}(t,x)|^2
 \le C(1+ |x|^2),
\ee
 $(t,x ) \in  Q_T$,
  $\alpha\in \mathcal A, \beta\in \mathcal B$.
Indeed, in this case   (iv) is satisfied with $\varphi (t,x) =
e^{Mt}(1+ |x|^2)$ for some suitable $M>0$.

\vskip 1mm We point out that  if $F$ consists  only of
$\sup\limits_{\alpha\in \mathcal A}\{ \cdot \}$ (i.e., we have a
Bellman operator), then Lipschitz continuity of solutions of
\eqref{fnl} holds even if $q_{\alpha}$ are degenerate, assuming
Lipschitz continuity of coefficients  (see \cite{YZ} and \cite{BCQ}
for the joint Lipschitz continuity in $(t,x)$).

\subsection{Nonlinear first order terms}
Here we consider the case  when the
nonlinearity only concerns the first order terms, namely   the
equation \be\label{eqnl}
\partial_t u -{\rm
tr}\left(q(t,x)D^2u\right)=H(t,x,Du)+  h(t,x)\qquad \hbox{in
$Q_T$,} \ee
where $H$ is continuous on $  Q_T \times
\R^N$, $q$ is continuous on $ Q_T$ and $h $ is continuous on
$  Q_T$ and has  bounded oscillation. We also assume that
$$
H(t,x,0)=0\,.
$$
 In  general, one can reduce to this case whenever   $H(t,x,0)$ has bounded oscillation.
We assume the ellipticity condition \rife{q1}, and we suppose that
$q(t,x)$ and $H(t,x,p)$
%\,:Q_T\times \R^N\to \R$
satisfy:
\be\label{h1}
\begin{array}{c}
\hbox{ $\exists$  a non-negative function
  $g\in C(0, 1)\cap L^1(0,1)$ such that}
  \\
  \m
  \quad
\mu \|\sigma(t,x)-\sigma(t,y)\|^2+ \big( H(t,x,
\mu(x-y))-H(t,y,\mu(x-y)\big) \leq
\\
\m \quad \leq \mu |x-y| \, g(|x-y|)
  +  (\mu | x- y|)^{2}(c_0+c_1(\mu |x-y|^2)^{q-1})\omega(|x-y|)
  + M,
 \end{array}
\ee  for every $\mu >0,$
 $x,y\in \R^N$
 such that  $0<|x-y|\leq 1$ and every $t\in (0,T)$,
where $\omega$ is some function such that $\lim_{r \to 0^+}\omega(r)
 =0$,  $q>1$, $c_0, c_1, M \ge 0$ and  $\sigma(t,x)
  = \sqrt{q(t,x) - \lambda I}$.
    %satisfies \rife{sig}.
  \vskip0.4em
Assumptions \rife{q1} and \rife{h1} imply that Hypothesis
\ref{pw-nl} holds true for the operator
$$
F(t,x,p,X)=
-{\rm tr}\left(q(t,x)X\right)-H(t,x,p).
$$
Indeed, whenever
$X$, $Y$ satisfy \rife{ine},
we deduce
$$
{\rm tr}(q(t,x)X-q(t,y)Y)\leq \la {\rm tr}(X-Y) + \mu
\|\sigma(t,x)-\sigma(t,y)\|^2+ \nu\left(  {\rm tr}(q(t,x)) + {\rm
tr}(q(t,y)) - 2 \lambda N\right)\,.
$$
Therefore,   using  \rife{h1} we deduce that $F(t,x,p,X)$ satisfies
Hypothesis \ref{pw-nl} with
$\eta(t,x,y)= \big(  {\rm tr}(q(t,x)) $ $ + {\rm tr}(q(t,y))\big) -
2 \lambda N $.
As far as the Lyapunov condition is
concerned, in this example Hypothesis \ref{pert-fi} reads as
follows:

 For any $L >0$, $\exists \,\, \varphi= \varphi_L \in
 C^{1,2}(\bar Q_T)\,, \;    \vep_0=
\epsilon_0 (L)> 0\,$: \be\label{lyap-ex}
\begin{array}{l}
\begin{cases}
 \partial_t \varphi -{\rm tr}(q(t,x)D^2\varphi) -\frac1\vep \left\{ H(t,x, p+ \vep D\varphi) -H(t,x,p)\right\}\geq  0 &
\\
\m \qquad \hbox{for every $(t,x)\in Q_T$, $p\in \RN$: $|p|\leq
L+\vep |D\varphi|$,  and every $\vep\leq \vep_0$} &
 \\
 \m
\varphi(t,x)\to +\infty \quad \hbox{as $|x|\to \infty$, uniformly
for $t\in [0,T]$.} &
\end{cases}
\end{array}
\ee In particular, under the previous assumptions \rife{q1},
\rife{h1} and \rife{lyap-ex},
the
conclusion of Theorem \ref{pw-fully} applies to equation
\rife{eqnl}.

\begin{remark}{\rm
Assume that  $H(t,x,p)$ satisfies,  for every  $(t,x)\in  Q_T$,
every $p$, $\xi\in \R^N$: \be\label{h2} \left| H(t,x, p+ \xi)- H(t,
x,p)- H(t,x,\xi)\right| \leq   \gamma(t,x)|\xi|+ c(t,x) (|\xi|+|p|
)^{q} |\xi| \ee for some non-negative functions $c(t,x)$,
$\gamma(t,x)$, with $q\ge 0$. Observe that the function $c(t,x)$
accounts for the possibly superlinear growth of $H(x,p)$ with
respect to $p$. If we also assume that $c(t,x)$ is bounded and
\be\label{lyap2}
\begin{array}{c}
\exists \,\, \varphi\in C^{1,2}(\bar Q_T)\,, \;    \hat \vep_0> 0\,
\,:\, \; \hbox{ for every   $(t,x)\in   Q_T $ and $\vep\leq
\hat \vep_0$}
\\
\m
\begin{cases}
{\rm tr}\left(q(t,x)D^2\varphi \right)+ \frac1\vep\,  H(t,x,\vep   D \varphi)+ \gamma(t,x)\, |  D \varphi| + \hat \vep_0  |  D \varphi |^{q+1}\leq    \partial_t \varphi\\
\varphi(t,x)\to +\infty \quad \hbox{as $|x|\to \infty$, uniformly in
$t\in (0,T)$,}
\end{cases}
\end{array}
\ee then condition \rife{lyap-ex} is satisfied. Indeed, in this case
one has
$$
\begin{array}{c}
\frac1\vep \left\{ H(t,x, p+ \vep D\varphi) -H(t,x,p)\right\} \leq
\frac1\vep\,  H(t,x,\vep   D \varphi)+ \gamma(t,x)\, |  D \varphi|
\\
\m + c(t,x) (|\vep D\varphi |+|p | )^{q} |D\varphi| \,.
\end{array}
$$
Since, for every $p$: $|p|\leq L+\vep |D\varphi|$, we have
$$
\begin{array}{c}
c(t,x) (|\vep D\varphi |+|p | )^{q} |D\varphi| \leq \vep^q
K(\|c\|_\infty) |D\varphi|^{q+1}+ \frac{\hat \vep_0}2 \, |
D\varphi|^{q+1} + K(L,\|c\|_\infty)
\end{array}
$$
if $\vep$ is small we deduce that
$$
\frac1\vep \left\{ H(t,x, p+ \vep D\varphi) -H(t,x,p)\right\} \leq
\frac1\vep\,  H(t,x,\vep   D \varphi)+ \gamma(t,x)\, |  D \varphi| +
\hat \vep_0 |D\varphi|^{q+1}+ K(L,\|c\|_\infty).
$$
Hence \rife{lyap2} implies
$$
\partial_t \varphi -{\rm tr}(q(t,x)D^2\varphi) -\frac1\vep \left\{ H(t,x, p+ \vep D\varphi) -H(t,x,p)\right\}\geq - K(L,\|c\|_\infty) \,.
$$
Then $\tilde \varphi:= \varphi+ K(L,\|c\|_\infty) t$ satisfies
\rife{lyap-ex}. }
\end{remark}

\vskip0.3em As a model case which can be dealt with, consider the
following equation \be\label{model-nl} \partial_t u -{\rm
tr}\left(q(t,x)D^2u\right)- b(t,x)\cdot   D u+ \Psi(t,x,   D
u)+c(t,x)|  D u|^{q+1}=h(t,x) \quad \hbox{in $Q_T$.} \ee
We deduce then the following

\begin{corollary}\label{mod}
Consider the equation \rife{model-nl}, with $q\leq 1$. Assume that
$q(t,x)$ satisfies \rife{q1} and that $\sqrt{q(t,x) - \lambda I}$
and $b(t,x)$ satisfy  \rife{pw}. Suppose that  $h$ has  bounded oscillation
and let $\Psi(t, x, \xi)$
 be a continuous function on
 $Q_T \times \R^N$ which
 satisfies $\Psi(t,x,0)=0$ and, for $ \xi, \eta \in \R^N$, $x,y \in \R^N$, $ 0 < |x-y|\le 1$, $t\in (0,T)$,
$$
\begin{array}{ll}
(i)\qquad & |\Psi(t,x,\xi)-\Psi(t,x,\eta)|\leq  \gamma(t,x)|\xi-\eta|
\\
\m (ii)\qquad & |\Psi(t,x,\xi)-\Psi(t,y,\xi)|\leq g(|x-y|) \, |\xi|\,,
\end{array}
$$
for some continuous functions $\gamma(t,x)$ on $Q_T$ and $g\in
L^1(0,1) \cap C(0,1 ; \R_+)$. Assume that  $c(t,x)$   is  bounded and continuous on $Q_T$,
uniformly continuous with respect to $x$ (uniformly for $t\in
(0,T))$. Assume in addition that
 \be\label{lyap21}
\exists \,\, \vfi\in C^{1,2}(\bar Q_T)\, \, :\, \;
\begin{cases}
{\rm tr}\left(q(t,x)D^2\vfi\right)+
b(t,x)\cdot   D \vfi +   \, \gamma(t,x) |D \vfi| \leq    \partial_t \vfi \;\; \text{in $Q_T$},  \\
\vfi(t,x)\to +\infty \quad \hbox{as $|x|\to \infty$, uniformly for
$t\in [0,T]$}
\end{cases}
\ee Then the conclusion of Theorem \ref{pw-fully} holds true.
\end{corollary}

\proof Here we have   $H(t,x,p)=b(x)\cdot p+  \Psi(t,x,p) +
c(t,x)|p|^q $. Using \rife{pw} and  the assumption  (ii) on  $\Psi$,  we
immediately deduce that \rife{h1} is satisfied with $c_1=0$ and
$\omega$ being the modulus of uniform continuity of $c(t,\cdot)$.

Given $\vfi$ satisfying \rife{lyap21}, consider now $
\phi\,:= \log (\vfi+ \eta)+Rt$ (with constant $\eta$ such that,
$\vfi+\eta\geq 1$).  Since the
assumptions on $\Psi$ imply
$$
\frac1\vep \left\{ H(t,x, p+ \vep D\varphi) -H(t,x,p)\right\} \leq
b(t,x)\cdot D\vfi+   \gamma(t,x)\, |  D \vfi | + \frac1\vep
 c(t,x) \left\{ |p+\vep D\vfi|^{q+1}- |p|^{q+1}\right\}
$$
we deduce, for $\vep$ small and for every $p$: $|p|\leq L+\vep  |D\phi|$,  and using that $q\leq 1$, that  $\phi$ satisfies
$$
\begin{array}{c}
 \partial_t \phi -{\rm tr}(q(t,x)D^2\phi) -\frac1\vep \left\{ H(t,x, p+ \vep D\phi) -H(t,x,p)\right\}
\\
\m \geq  R+ \frac1{\vfi+\eta}\left(  \partial_t \vfi- {\rm
tr}\left(q(t,x)D^2\vfi (x)\right) - b(t,x)\cdot   D \vfi -
\gamma(t,x) |  D \vfi|\right)
\\
\m \quad  - K(q,\|c\|_\infty) \frac{|D\vfi|}{\vfi+\eta}  \left(L+
2\vep \frac{|D\vfi|}{\vfi+\eta} \right)^q + \frac{q(t,x)  D \vfi \cdot D
\vfi}{(\vfi+\eta)^2}
\\
\m \quad \geq R - \frac\la 2\frac{|D \vfi |^2}{(\vfi+\eta)^2}- \tilde K
+\frac{q(t,x)  D \vfi \cdot D \vfi}{(\vfi+\eta)^2}
\end{array}
$$
for some constant $ \tilde K = \tilde K(q,L, \la, \|c\|_\infty)$. Since $q(t,x)\geq
\la\,I$, choosing $R$ sufficiently large we deduce that
\rife{lyap-ex} is satisfied. Therefore  Theorem \ref{pw-fully} applies.
\qed

Observe that
if $H(t,x,\xi)=b(t,x)\cdot \xi$, then  \rife{lyap21} reduces to
\rife{lyap0}, and we recover the result of the linear case. In fact,
whenever  $\gamma(t,x)$ is also bounded, with the same argument as
above we see that the term $\gamma(t,x)|D\vfi|$ can also be
neglected in \rife{lyap21}. In conclusion, if both the coefficients
$\gamma(t,x)$ and $c(t,x)$ are bounded, and $q\leq 1$, the Lyapunov
condition of the linear case is enough to ensure the same condition
for the nonlinear equation \rife{model-nl}.

\begin{remark}{\rm
It could still be possible to consider the case  when the function
$c(t,x)$ in \eqref{model-nl} is unbounded.
As in Corollary \ref{mod}, one can use  (through a log-transform)
the ellipticity of $q(t,x)$ to control the term with $c(t,x)$. In
particular, whenever \rife{lyap21} is satisfied for some $\vfi$: $|
D \vfi|\to \infty$ as $|x|\to \infty$ (uniformly in $t$),   if we have $c(t,x)= $o $\left( l(t,x)\right)$ as $|x| \to + \infty$, uniformly in $t$, where
$l(t,x) =$ $\inf_{|\xi|=1} q(t,x)\xi \cdot \xi$,
it is still possible to use
the above argument to conclude. }
\end{remark}

\subsection{A Liouville type theorem}
Here we show a related Liouville type theorem
which also extends \cite[Theorem 3.6]{PW} to the  nonlinear setting.

 Let $F: \R^N \times \R^N \times {\mathcal S}_N \to \R$ be a continuous function.
We consider viscosity solutions  $v \in C_b(\R^N)$ to
\be \label{liu1}
F(x,Dv, D^2v) =0 \;\; \text{in } \; \R^N.
\ee
We suppose that $F$ verifies (cf. Hypothesis \ref{pw-nl}):
there exist
$\lambda>0$,  non-negative functions
  $\eta(x,y)$  and
  $g\in C((0, + \infty); \R_+ ) \cap L^1(0,1)$ such that
  \be\label{Feq5}
  \begin{array}{c}
  F(x, \mu(x-y), X)- F(y,\mu(x-y),Y)\geq
 - \la {\rm tr}\left(X-Y\right)- \mu |x-y| \, g(|x-y|)
  - \nu \, \eta(x,y),
  \end{array}
\ee for any $\mu> 0$, $\nu \ge 0$, \  $ x,y\in \R^N,$ matrices    $ X,Y\in {\mathcal S}_N$ which verify \eqref{ine}. Note that if $X=Y =0$, letting $\mu \to 0^+$ and $\nu \to 0^+$ we deduce from \eqref{Feq5} that
$$
F(x, 0, 0)= F(y,0,0),\;\; x, y \in \R^N.
$$
The following result is a slightly more general version  of Theorem
\ref{liu0} stated in the Introduction; here we use the weaker
condition Hypothesis \ref{pert-fi} concerning the Lyapunov function
rather than Hypothesis \ref{L}.
\begin{theorem} \label{liu} Suppose that $F$ satisfies
Hypothesis \ref{pert-fi}  and condition \rife{Feq5}.
\begin{itemize}
\item[(i)] If $g$ satisfies
\be \label{lira} \int_0^{+\infty} e^{-\frac{1}{4 \lambda} \int_0^r
g(s)ds} dr = +\infty, \ee then any bounded  viscosity solution  $v
\in C_b(\R^N)$ of \rife{liu1} is constant.
%\vskip0.5em
\item[(ii)] If $g$ satisfies
\be\label{galpha} \limsup_{s\to \infty} \, g(s) \,s <
4\la(1-\alpha)\,, \ee then any   viscosity solution  $v $ of
\rife{liu1}
 which verifies,
 for some $\alpha \in (0,1)$, $k_0 \ge 0$
 \be \label{frgr}
|v(x) - v(y)|\le k_0 (1+  |x-y|^{\alpha}),\;\;\; x,y \in \R^N, \ee
must be constant.
\end{itemize}
In particular, if $F(x,0,0)$ is not identically zero, and if
\rife{lira} holds true, then \rife{liu1} does not possess any
bounded  viscosity solution, nor it possesses any solution
satisfying \rife{frgr} if \rife{galpha} holds true.
\end{theorem}
If $\alpha=0$, condition   \eqref{galpha}  becomes a  particular
case of \eqref{lira}  (the optimality of the condition $\limsup
g(s)s < 4 \lambda$, in the case of  bounded solutions, was observed
e.g. in \cite{PW10}).

 According to assertion (i),  if $F(x,Dv, D^2 v)= -\triangle v -
b(x) \cdot Dv -1 $, where $b$ satisfies \eqref{pw} with $g$ such
that  \eqref{lira} holds,  then there are no bounded viscosity
solutions to \eqref{liu1}.

The previous theorem can be generalized with a similar proof   to
parabolic equations like  $\partial_t v + F(x,Dv, D^2v)=0$ by
considering viscosity solutions $v$ defined on $(0, \infty) \times
\R^N$.

We refer to \cite{CaCu}  (see also references therein)  for
Liouville  theorems   concerning nonnegative viscosity
supersolutions under assumptions somehow related to \rife{lira}
above (see in particular  \cite[condition (4.2) in Theorem
4.1]{CaCu}).

\begin{proof}  First, one checks easily that $v(t,x)=v(x)$, $t \in (0,T)$, $x \in \R^N,$ is also a viscosity solution to the parabolic equation
$$
\partial_t v + F(x,Dv, D^2v) =0 \;\; \text{in } \; Q_T,
$$
for any $T>0$.
 Then we proceed as in the proof  of Theorem \ref{pw-fully},  with the same notation,
having  fixed $t_0 >0$ and $\de>0$. Observe that assumption  \rife{Feq5} is a stronger version of Hypothesis \ref{pw-nl}, and in this case we obtain inequality \rife{liouv2}, namely
$$
\begin{array}{c}
|v(x)-v(y)|\leq  K\, f(|x-y|),   \qquad  |x-y| \leq \de,
\\
\m
\m
\hbox{where $K =\frac{4 \omega_{0,\de}(v) }{t_0} +  \, \frac{\omega_{0,\de}(v) }{f(\de)}$\,,}
  \end{array}
$$
and   $f$ is the solution of \rife{ode1}, i.e.,
 $f(r) = f_{\delta}(r)$, $r \in [0, \delta]$,
$$
f_\de(r) = \frac1{4\la} \int_0^r e^{-\frac{G(\xi)}{4\la}}\int_\xi^{\de}
e^{\frac{G(\tau)}{4\la}}d\tau d\xi\,, \qquad \hbox{ $G(\xi)=
\int_0^\xi g(\tau)d\tau$}.
$$
% so that
%$K\ge \frac{2 \| v\|_{\infty} 8 \lambda}{\delta^2}$ %implies  $K\geq \frac{2 \| v\|_{\infty}}{f(\de)}$.
 Note that $ f_{\delta}(\delta) \to + \infty$, as  $\delta \to + \infty$.
 Now, letting $t \to + \infty$, we find, for $|x-y|\le \delta$,
\be\label{preliou}
 |v(x)-v(y)|\leq  \frac{\omega_{0,\de}(v) }{f_\de(\de)}\, f_\de(|x-y|),   \qquad  |x-y| \leq \de\,.
\ee
In the following,  let us fix any $x$ and $ y \in \R^N$, we take  $\delta > |x- y|$ and we consider the limit as $\de\to \infty$.

If we are in case (i), $v$ is bounded and \rife{preliou} implies
$$
 |v(x)-v(y)|\leq  2\|v \|_{\infty}\frac{f_\de(|x-y|) }{f_\de(\de)}\,.
$$
Letting $\delta \to +\infty$,
by using L'H\^opital's rule and
thanks to \eqref{lira}, we have
$$
\lim_{\delta \to + \infty}\frac{f_\de(|x-y|) }{f_\de(\de)} = \lim_{\delta \to + \infty} \frac{
    \int_0^{|x-y|} e^{-\frac{1}{4 \lambda}\, \int_0^r g (s)ds} dr}{\int_0^{\de} e^{-\frac{1}{4 \lambda}\, \int_0^r g (s)ds} dr}=0
$$
and so  $v(x)= v(y)$. It follows that $v$ is constant.

 If we are in case (ii), on account of \rife{frgr} \rife{preliou} implies
 $$
 |v(x)-v(y)|\leq  k_0\frac{(1+\de^\alpha)f_\de(|x-y|) }{f_\de(\de)}\,.
$$
 Letting $\delta \to +\infty$, we apply again L'H\^opital's rule: to this purpose, observe that
 $$
 \frac{\de^{\alpha-1} f_\de(|x-y|) }{e^{\frac{1}{4 \lambda}\, \int_0^\de g (s)ds}\int_0^{\de} e^{-\frac{1}{4 \lambda}\, \int_0^r g (s)ds} dr} \leq C
\frac{\de^{\alpha} }{\int_0^{\de} e^{-\frac{1}{4 \lambda}\, \int_0^r
g (s)ds} dr},
 $$
  since $  f_\de(|x-y|)\leq C \,
\de\, e^{\frac{1}{4 \lambda}\, \int_0^\de g (s)ds}$ for some
$C=C(\lambda, x,y)$. Moreover, due to \rife{galpha}, we have
  $$
  g(s) \leq \frac{4\la(1-\alpha)-\vep}s \qquad \forall s\geq s_0,
  $$
  for some $\vep>0$ and some $s_0>0$. We deduce that
  $$
  \int_0^{\de} e^{-\frac{1}{4 \lambda}\, \int_0^r g (s)ds} dr
  \geq c_0 (\de^{\alpha+\frac\vep{4\la}} -1),
  $$
  for some  $c_0= c_0(\alpha, \lambda, \epsilon)>0$, hence
  $$
  \lim_{\delta \to + \infty}  \frac{\de^{\alpha} }
  {\int_0^{\de} e^{-\frac{1}{4 \lambda}\, \int_0^r g (s)ds} dr}=0
  $$
  and so, by L'H\^opital's rule, we conclude that
  $$
  \lim_{\delta \to + \infty}
  \frac{(1+\de^\alpha)f_\de(|x-y|) }{f_\de(\de)} =0.
  $$
 This gives $v(x)=v(y)$ and we conclude.
\end{proof}
We give   an  application to the previous result to Bellman-Isaac's
type operators (see also Section 4.5c). Assume that $F$ in
\eqref{liu1} is of the following type
$$ F(x,Du, D^2 u)= \inf\limits_{\beta\in \mathcal
B}\sup\limits_{\alpha\in \mathcal A}\, \left\{-{\rm
tr}\left(q_{\alpha,\beta}(x)D^2u\right)-
 b_{\alpha,\beta}(x)\cdot   D u \right\},
$$
where  $
 \inf\limits_{\beta\in \mathcal B}\,\sup\limits_{\alpha\in
 \mathcal A}\,
 \left\{{\rm tr}\left(q_{\alpha,\beta}(x) \right)
 +  |b_{\alpha,\beta}(x)|  \right \}
  < \infty\, $ and
  $
  \sup\limits_{\beta\in \mathcal B}\,\sup\limits_{\alpha\in
 \mathcal A}\,
 \left\{{\rm tr}\left(q_{\alpha,\beta}(x) \right)
\right \}
  < \infty,\;\;\; x\in \R^N.
$

\hh  Assume also that $q_{\alpha, \beta}(x) \ge \lambda
  I$, for some $\lambda>0$, uniformly in $\alpha, $ $\beta$, $x \in
  \R^N$, and that $q_{\alpha, \beta}(x)$ and $b_{\alpha, \beta}(x)$
 are continuous on $\R^N$,  uniformly in $\alpha \in  {\mathcal A}$
 and $\beta \in {\mathcal B}$.
 Finally suppose that    $q_{\alpha, \beta}$ are
  $C^1$-functions on $\R^N$, such that
\be \label{cer3}
 2\| D_h \sqrt{q_{\alpha, \beta}}\, (x)\|^2+
(b_{\alpha, \beta}(x+h)-b_{\alpha, \beta}(x))\cdot h \leq 0, \;\;\;
x \in \R^N,\; h \in \R^N, \;\; \alpha \in {\mathcal A}, \, \beta \in
 {\mathcal B}
\ee (for  smooth $N \times N$ matrices $a(x)$, we write $D_h a(x)$
to denote the matrix $(D_h a_{ij}(x))$, where $D_h a_{ij}$ stands
for the directional derivative along  $h \in \R^N$; see also
\eqref{cer}).
  Note that by \eqref{cer3} we deduce that $\varphi(x) = 1+ |x|^2$
 is a Lyapunov function.

 \vskip 2mm  {\sl Under the previous assumptions  we can apply Theorem \ref{liu}  and conclude that all
  viscosity solutions of \rife{liu1} which verify \rife{frgr}
 are  constant.}

 \subsection{A remark on possible  existence results}
 Our  estimates are particularly useful in
cases in which the comparison principle is not known due to the
irregularity of coefficients and so one cannot perform the Perron's
method to get existence of viscosity solutions. We recall (see \cite[Lemma 9.1]{BBL}) that, under fairly general conditions on the function $F$, once a local modulus of continuity is established for the $x$-variable, then it is possible to deduce  a local modulus of continuity for the $t$-variable. In particular, if Lipschitz continuity holds in the $x$-variable, then $\frac12$-H\"older continuity holds in the $t$-variable. Thanks to such  a result, the above Lipschitz, or H\"older estimates which we proved imply  a local uniform equi--continuity for the viscosity solutions, and therefore  a local compactness in the uniform (space--time) topology. As a consequence, suppose that the function $F$ can be  approximated by a  sequence $F_n$ of functions satisfying Hypotheses \ref{pw-nl} and \ref{pert-fi} or Hypotheses \ref{pw-nl-hol} and \ref{pert-fi-bis} uniformly with respect to $n$ and such that viscosity solutions $u_n$ are known to exist for the Cauchy problem
$$
\partial_t u_n + F_n(t,x,D u_n, D^2 u_n)=0,
$$
with $u_n(0)=u_{0n}$ converging locally uniformly to $u_0$. Then, if $u_0\in C_b(\R^N)$ (i.e. $u_0$ is continuous and bounded) we conclude the existence of a bounded viscosity solution $u$ which is Lipschitz, or H\"older, as $t>0$. Similarly, we deduce the existence of a viscosity solution in case $u_0$ is continuous and has bounded oscillation provided the stronger  Hypotheses \ref{pw-nl1} or \ref{pw-nl1-hol} are satisfied.
In particular, in the linear case, as well as in the case of Bellman-Isaacs equations,  we deduce the existence of a viscosity solution by simply approximating the coefficients through a standard convolution regularization.

\appendix

\section{Probabilistic Vs analytic approach}
\label{prob-anal}

Let us spend a few words
 on the comparison between the probabilistic  proof
 of \cite[Theorem 3.4]{PW} and the previous analytical proof of
Theorem \ref{PW}. For simplicity, {\it we assume here that coefficients
 $q$ and $b$
  are independent of time as in \cite{PW}},
namely we consider
$$
A=
 \sum\limits_{i,j=1}^N
q_{ij}(x)\partial^2_{x_i x_j}+ \sum\limits_{i=1}^N
b_i(x)\partial_{x_i}.
% - V(x)f
$$
The probabilistic approach used to prove estimate \rife{est-base} relies on a  so-called coupling method
% In \cite[Theorem 3.4]{PW}  to prove gradient estimates like \rife{est-base}, under the assumptions on $A$ given  in  Theorem \ref{PW},  one uses the probabilistic coupling approach
(see also
\cite{Lin-Rog}, \cite{Chen-Li}, \cite{Cranston},
 \cite{Cranston1}).
 Let us briefly explain  such method.
  Under the assumptions of Theorem \ref{PW},  it is well known that
 there exist unique
 probability measures $\mathbb P^x$ and  $\mathbb P^y$ on
$\Omega_N:= C([0,\infty);\R^N)$ which are martingale solutions for
$A$ (starting from $x \in \R^N$ and $y \in \R^N$ respectively; see
\cite{SV}). In particular,  for $u_0 : \R^N \to \R$ continuous and bounded, this allows to define  the probabilistic solution to the Cauchy problem
\be\label{probcau}
 \begin{cases}  \partial_t u -{\rm tr}\left({ q(x)}D^2u\right)-
{ b(x)}\cdot   D u  =0\qquad
%\hbox{in} \;
\\
u(0, \cdot) = u_0,  \;\;\;
%\text{as}
\end{cases}
\ee
as $$
u(t,x) = \mathbb E^x\left[u_0(x_t)\right], \;\; t\ge0,
$$
where $(x_t)$ denotes the canonical process over $\Omega_N$ with
values in $\R^N$. Note, incidentally,  that this is a viscosity
solution (see \cite{Fle-Soner}). Equivalently, we can also write
 $u = P_t u_0$ where $P_t$ is the diffusion Markov semigroup
 associated to \eqref{probcau}.

A coupling for the measures $\mathbb P^x$ and $\mathbb P^y$ is a probability
measure $\mathbb P^{x,y}$ on $\Omega_{2N}= C([0,\infty);\R^{2N})$ such that the two  marginal
distributions are, respectively, $\mathbb P^x$ and $\mathbb P^y$.
 For any
coupling $\mathbb P^{x,y}$, we have that
 \be \label{ppxy}
  u(t,x)-u(t,y)= \mathbb E^x\left[ u_0(x_t)\right]
-\mathbb E^y\left[u_0(y_t)\right]= \mathbb E^{x,y}\left[
u_0(x_t)-u_0(y_t)\right] \,,\qquad \ee
 where $\mathbb E^{x,y}$
denotes the expectation with respect to $\mathbb P^{x,y}$   and
  $(z_t)$ $=\big((x_t, y_t)\big)$ denotes  the canonical process over
 $ \Omega_{2N} $ with values in $\R^{2N}$.

A Lipschitz estimate for $u(t)$ is then obtained if
we are able to estimate the last term in \rife{ppxy}.
Introducing the coupling time $
T_c:=\inf\{t\geq 0\,:\, x_t=y_t\} $ (with the usual convention $\inf(\emptyset)=\infty$), this
estimate can be obtained by constructing a coupling  $\mathbb P^{x,y}$ which verifies
the property ($\mathbb P^{x,y}$-a.s.):
\be\label{succ} x_t=y_t
\;\;\; \hbox{for all $t\geq T_c$ on $\{T_c < \infty\}$}. \ee
Indeed, if the coupling $\mathbb P^{x,y}$ satisfies the property \rife{succ},
then we have clearly
\be \label{cia3}\mathbb E^{x,y}\left[
u_0(x_t)-u_0(y_t)\right] \leq 2\|u_0\|_\infty\mathbb P^{x,y}(
t<T_c)\,, \ee
 so that, in the end, one is left with the  estimate of
the first hitting time of the diagonal set $\Delta:=\{(x,y)\in
\R^{2N}:\, x=y\}$ for the process $\big((x_t,y_t)\big)$ starting from
$(x,y)\not\in\Delta$.
\vskip0.2em
 It is proved in \cite{PW} that a coupling measure $\mathbb
P^{x,y}$  satisfying \rife{succ} can be  constructed starting  from any martingale
solution  for the   differential operator in $\R^{2N}$
\be\label{op-coup} {\mathcal A}_c:= \sum\limits_{i,j}
\left(q_{ij}(x)\partial^2_{x_i,x_j} +2
c_{ij}(x,y)\partial^2_{x_i,y_j} +q_{ij}(y)\partial^2_{y_i,y_j}
\right)+ \sum\limits_{i}(b_i(x)\partial_{x_i}+ b_i(y)\partial_{y_i}),
 \ee
 where the only requirement is  that $c_{ij}$ are continuous  on $\R^{2N}$ and  the
 %$2N \times 2N$
 matrix
 $\begin{pmatrix}
 q(x) & c(x,y)    \\
\noalign{\medskip} c^* (x,y) &  q(y)
\end{pmatrix}$ is symmetric and nonnegative. Indeed,  if $\tilde {\mathbb P}^{x,y} $ is a martingale solution for $\mathcal A_c$, then one can define on $(\Omega_{2N}, \tilde {\mathbb P}^{x,y}  )$ the process
 $$
x_t' =\begin{cases} x_t,\;\;\; t \le T_c,
 \\ y_t,\;\;\; t> T_c,
 \end{cases}
 $$
and,  using that the martingale problem for $A$ is
well-posed, one can prove  that the process $(x_t)$ and $(x'_t)$ have the same law.
 On  $(\Omega_{2N}, \tilde {\mathbb P}^{x,y}  )$ we define the
 process $\big((x_t', y_t)\big)$.  Its law
is a coupling measure $ {\mathbb P}^{x, y} $
  satisfying also \eqref{succ}.
Note that the martingale problem for $\mathcal A_c$ could
be not well-posed due to the fact that $\mathcal A_c$ is possibly degenerate and has coefficients which are not  locally Lipschitz in general.

\vskip 1mm
Summing up, the probabilistic approach relies on the possibility
that the Lipschitz regularity  can be deduced
in   two main steps:
\begin{itemize}

\item[(i)] construction of a { suitable}
coupling $\mathbb P^{x,y}$  (satisfying \rife{succ} and associated
to ${\mathcal A}_c);$
%\vskip0.3em
\item[(ii)]  estimate of
$\mathbb P^{x,y}( t<T_c)  $ in terms of $|x-y|$.
\end{itemize}
\hh
In particular, if  one has
 \be \label{gth}\mathbb P^{x,y}( t<T_c) \le \frac{k}{\sqrt{t \wedge 1}} |x-y|, \;\;\; t>0
 \ee
then  \rife{ppxy} and \rife{cia3} imply the desired Lipschitz estimate \rife{est-base}.

\smallskip  At the second stage (ii),
 the differential operator
 ${\mathcal A}_c$ (and so the right choice of $c(x,y))$
 play a crucial  role, since
   a quite standard procedure
 (see e.g. \cite{Fr}) to obtain
  estimates on   the hitting time of the diagonal $\Delta$
is through the construction of suitable supersolutions $W$ for
${\mathcal A_{c}}$, i.e.,
$$
%\begin{cases}
\partial_t W - {\mathcal A_c}(W)
  \ge 0 \;\;\;\; \hbox{in \ \ $(0,T)\times (\R^{2N}
\setminus \Delta).$}
$$
Indeed,  estimate  \rife{gth} is proved in \cite{PW} after choosing
\be \label{choice1}
c(  x,  y)= \sigma(    x)\sigma(    y)+\la\left(I- 2
\left(\frac{  x-  y}{|  x-  y|}\otimes\frac{  x-  y}{|  x-  y|}\right) \right),
\ee
 where $\sigma(x)^2 = q(x)-\lambda\, I, $
 and finding suitable supersolutions for the corresponding  ${\mathcal A_c}$.

\begin{remark} {\rm
 The above approach can be also rephrased in terms of  a
Wasserstein distance between  the  transition probabilities
%measures
$p(t,x,\cdot)$ and
$p(t,y,\cdot)$, $t \ge 0,$ $x,y \in \R^N$
 (i.e., $p(t,x,A)= {\mathbb P^x}(x_t \in A)$, for any Borel set $A \subset \R^N$).
%which are the transition probability
%induced by the processes $x_t$ and $x_t$ over $\R^d$.
More precisely if, given two  Borel probability measures $\mu$, $\nu$ in
$\R^N$, we define a distance (see e.g. \cite{Chen-Li}) as
$$
d_W(  \mu,  \nu)= \inf_{Q\in \pi(\mu,\nu)}\,\, \int\int \chi(z,w) dQ(z,w)\,,\qquad \chi(z,w)=\begin{cases} 1 & \hbox{if $z\neq w$}\\
0 & \hbox{if $z=w$},
\end{cases}
$$
where $\pi(\mu,\nu)$ is the set of all couplings of $\mu,\nu$  on $\R^{2N}$, then we
have, being
$\mathbb P^{x,y}$   a  coupling of $\mathbb P^x$ and $\mathbb P^y$,
$$
\mathbb E^{x,y}\left( u_0(x_t)-u_0(y_t)\right) \leq 2\|u_0\|_\infty
d_W(p(t,x,\cdot),p(t,y,\cdot))\,.
$$
Therefore from \rife{ppxy} we get
\be\label{lip-was}
u(t,x)-u(t,y)\leq  2\|u_0\|_\infty \,
d_W(p(t,x,\cdot),p(t,y,\cdot))\,
\ee
(recall that a classical result by  Dobrushin  says that $d_{W}$ is just half of the total variation distance).
In this framework,  the
Lipschitz estimate on $u(t,x)$ is reduced to an estimate of the
Wasserstein distance between $p(t,x,\cdot)$ and $p(t,y,\cdot)$. On the other hand, if one has constructed  a  coupling process  such that the
corresponding measure $\mathbb P^{x,y}$ over $\Omega_{2N}$ satisfies
\rife{succ}, then it follows
$$
d_W(p(t,x,\cdot),p(t,y,\cdot))\leq \mathbb P^{x,y}( t<T_c)
$$
and we end up with \rife{cia3} again.
%$$
%u(t,x)-u(t,y)\leq  2\|u_0\|_\infty \, \mathbb P^{x,y}( t<T_c) \,.
%$$
}
\end{remark}

\vskip0.5em
%To this purpose,
 Let us now rephrase the analytic proof given
in Theorem \ref{PW}  in order to see its close correspondence with the
probabilistic approach.
The key point in such proof is  the \lq\lq maximum principle
for semicontinuous functions\rq\rq\  (Theorem \ref{key}) which is
the heart of the viscosity solutions theory. In the linear
framework, this fundamental result has the following immediate
consequence, which is interesting in its own.

\begin{lemma}\label{visco-coup}
Let $u$, $v$ be respectively a  viscosity sub and supersolution of
\rife{probcau}.
Then, for every open set $\mathcal O\subset \R^{2N}$ and
every  $N \times N$ matrix $c(x,y)$ (with $c_{ij} \in C(\mathcal O))$
 such that
$\begin{pmatrix}
 q(x) & c(x,y)    \\
\noalign{\medskip} c^* (x,y) &  q(y)
\end{pmatrix}$ is nonnegative,
we have that $u(t,x)-v(t,y)$ is a viscosity subsolution of
\be\label{Lc}
 \partial_t z -{\mathcal A}_c(z)=0\qquad \hbox{in $(0,T)\times \mathcal O$}
\ee where ${\mathcal A}_c$ is the operator defined in
\rife{op-coup}.

In particular, if $u$ is a viscosity
solution of \rife{eq}, then $u(t,x)-u(t,y)$ is a viscosity
subsolution of \rife{Lc}.
\end{lemma}
\proof For $\hat t\in (0,T), (\hat x,\hat y)\in \mathcal O$, and
$z\in C^{1,2}\left((0,T)\times \mathcal O\right)$, let $(\hat t,\hat x,\hat y)$ be  a local maximum point of $u(t,x)-v(t,y)-z(t,x,y)$.
Applying Theorem \ref{key}  (see also Remark \ref{rit}), there exist $a,b\in \R$,
$X,Y\in {\mathcal S}_N$ such that $(a,D_x z(\hat t,\hat x,\hat y),X)\in {\overline
P}^{2,+}u(\hat t, \hat x)$, $ (b, - D_y z (\hat t,\hat x,\hat y), Y)\in {\overline
P}^{2,-}v(\hat t,  \hat y)$, $a-b=\partial_t z(\hat t,\hat x,\hat y)$ and
\be\label{XY}
\begin{pmatrix}
 X & 0    \\
\noalign{\medskip} 0 & -Y
\end{pmatrix}
 \leq D^2 z (\hat t,\hat x,\hat y).
\ee  We
proceed then as in the proof of Theorem \ref{PW}; using the
equations of $u$ and $v$,
subtracting the two and since $a-b=\partial_t z(\hat t,\hat x,\hat y)$, we
find inequality \rife{z1}. Now, if   $c(x,y)$ is  any $N \times N$
matrix
  such that
   $\begin{pmatrix}
 q(x) & c(x,y)    \\
\noalign{\medskip} c^* (x,y) &  q(y)
\end{pmatrix}$ is nonnegative,
multiplying both sides of \rife{XY} by this matrix (computed at
$(\hat x,\hat y)$)
and taking traces, we get
$$
\begin{array}{c}
\partial_t  z(\hat t,\hat x,\hat y)
 -{\rm tr}
\left( q( \hat x)D^2_x z(\hat x, \hat y) +q(\hat y) D^2_y z(\hat x, \hat y) -2c(\hat x,\hat y)D^2_{xy} z(\hat x, \hat y)\right)
\\
\m
 \leq  b(\hat x)\cdot D_x
z(\hat x, \hat y) +b(\hat y)\cdot D_y z(\hat x, \hat y)
\end{array}
$$
 that is
$$
\partial_t z(\hat t,\hat x,\hat y)
 - {\mathcal A}_c( z)(\hat t,\hat x,\hat y) \leq 0.
 $$
Since this is true for every $z$ and $(\hat t,\hat x,\hat y)$ being a local maximum of  $u(t,x)-v(t,y)-z(t,x,y)$, this means that $u(t,x)-v(t,y)$  is a viscosity subsolution of \rife{Lc}.
 \qed

The previous lemma offers  a nice interpretation, at least in the
linear setting,  to the use of matrix inequality \rife{XY} which is
usually the essential part of the doubling variable technique for
viscosity solutions. Indeed, it suggests that manipulations of the
matrix inequality \rife{XY} amount to  an optimization among all
possible coupling operators. In this viewpoint, the Lipschitz
estimate for $u$ is reduced to the existence of \lq\lq good\rq\rq
supersolutions for some  coupling operator $\mathcal A_c$. {
Since
 $u(t,x)-u(t,y)$ is a viscosity
subsolution of a family of linear problems, thanks to the comparison
principle for viscosity sub-super solutions, we expect that
$u(t,x)-u(t,y)\leq w(t,x,y)$ for every $w(t,x,y)$ being a
supersolution of some operator   $\mathcal A_c$. This fact could be
shortly rephrased in the following suggestive form closer to
 inequality \rife{lip-was}
$$
u(t,x)-u(t,y)\leq \inf\limits_{{\mathcal A}_c}\, \inf \{ \psi(t,x,y),
\,\,:\,\,\partial_t \psi-{\mathcal A}_c (\psi)\geq 0 \;\;
\}.
$$
Thus, in the analytic approach  the oscillation of $u$  can be
estimated by choosing the best among all possible supersolutions of
such family of operators.} Let us make it more precise as follows.

\begin{lemma}\label{supers} Assume \rife{lyap}  and let  $u\in C(Q_T)$  be a solution of \rife{eq} such that  $u=o_{\infty}(\vfi)$,  where $\vfi$ verifies \eqref{lyap}.  Assume that there exists  some matrix $c(x,y)\in C(\R^{2N}\setminus \Delta; \R^{N^2})$ such that
$\begin{pmatrix}
 q(x) & c(x,y)    \\
\noalign{\medskip} c^* (x,y) &  q(y)
\end{pmatrix}$ is nonnegative,
and  a nonnegative function $\psi\in C^{1,2}((0,T)\times
(\R^{2N}\setminus \Delta))$ such that
$$
\partial_t \psi - {\mathcal A}_c(\psi)\geq 0 \qquad \hbox{in $(0,T)\times
(\R^{2N}\setminus \Delta )$.}
$$
If, for some $\tau\in [0,T)$, we have $\psi(\tau,x,y)\geq
u(\tau,x)-u(\tau,y)$, then
$$
u(t,x)-u(t,y)\leq \psi(t,x,y),\qquad  t\in [\tau,T).
$$
\end{lemma}

\proof  Replacing $u(t,x)$ with $u^\vep(t,x)=u(t,x)-\vep \vfi(t,x)$
and $u(t,y)$ with $u^\vep(t,y)=u(t,y)+ \vep \vfi(t,y)$ we have that
$u^\vep(t,x)-u^\vep(t,y)$ is still a subsolution, and goes to
$-\infty$ when  $|x|\to \infty$ or $|y|\to \infty$, uniformly for
$t\in (0,T)$. Consider the function $\Phi_{\vep}= u^\vep(t,x)-u^\vep(t,y)-
(\psi(t,x,y)+ \frac\vep{T-t})$. Assume by contradiction that
$\sup_{[\tau,T)\times \R^{2N}}\,\Phi_{\vep}>0$. Since $\Phi_{\vep}\to
-\infty$ at infinity  as $t\to T$ too, $\Phi_{\vep}$ has  a maximum at
some point $(\hat t, \hat x,\hat y)$. Since
$u^\vep(t,x)-u^\vep(t,y) \leq 0\leq \psi$ on $\Delta$, it cannot be
$\hat x= \hat y$, and since $ u^\vep(\tau,x)-u^\vep(\tau,y)\leq
\psi(\tau,x,y)$ by assumption, we have $\hat t\in (\tau,T)$.
Therefore $(\hat t, \hat x,\hat y)$ is a local maximum with $(\hat x,\hat y)$ lying in
$\R^{2N}\setminus \Delta$ and by Lemma \ref{visco-coup} we deduce
that
$$
\frac\vep{(T-t)^2}+\partial_t \psi- {\mathcal A}_c(\psi)\leq 0
$$
which gives  a contradiction since $\psi$ is a supersolution. This
proves that $\sup\, \Phi_{\vep}\leq 0$, i.e.,  $u^\vep(t,x)-u^\vep(t,y)\leq
\psi(t,x,y)+ \frac\vep{T-t}$. Letting $\vep\to 0$ we conclude. \qed

\smallskip The above approach represents actually a sort of  analytic
translation of the probabilistic coupling method. Indeed, if Lemma
\ref{visco-coup} gives an analytic counterpart of  the
\rife{lip-was}, Lemma \ref{supers}  gives the way for the estimate
related to (ii).  Indeed, one might expect, roughly speaking, that
$z(t,x,y)= 2 \| u_0\|_{\infty} \mathbb P^{x,y}( t<T_c)$ be  a
supersolution to the Cauchy parabolic problem involving $\mathcal
A_c$ and having $u_0(x) - u_0(y)$ as initial datum at $t=0$, with
the additional property of  being
 \lq\lq minimal\rq\rq\ in the sense that $z(t,x,x)=0$, in  other words that $\mathbb P^{x,y}( t<T_c)$ solves a Cauchy-Dirichlet parabolic problem for $\mathcal A_c$
 with values $1$ at $t=0$ and $0$ at $\Delta$. On the other hand, instead of checking this property on
 $\mathbb P^{x,y}( t<T_c)$, it is enough to build a supersolution
$\psi$ with the desired features. Observe actually that  the
analytic proof of Theorem \ref{PW}  basically follows from the
previous lemmas up to showing that
$$
\psi(t,x,y)=K\, f(|x-y|)+ C_0 (t-t_0)^2
$$
is a supersolution in $(\frac {t_0}2, T)$ for
the coupling $c(x,y)$  given
in \rife{choice1}, for   a suitable choice of $K$, $C_0$ and $f$
satisfying \rife{ode1}.
This is indeed the computational part given
in the proof of Theorem \ref{PW} in Section 3.

To conclude this discussion, we have tried to underline the close
analogy between the probabilistic and analytic approach, though
expressed through different concepts and tools. We stress however
some advantage of the approach through viscosity solutions:  it only
relies on maximum principle, which is a pointwise device, and it
admits natural extensions to \emph{nonlinear} operators as we have
shown in Section 4.

\bigskip

\noindent \textbf{Acknowledgement. }{We would like to thank the
anonymous referee who read  carefully the manuscript and made useful comments and
suggestions to improve the presentation. }


\begin{thebibliography}{99999}

\bibitem[Ba91]{Ba1} G. Barles,
Interior gradient bounds for the mean curvature equation by
viscosity solutions methods, Differential Integral Eq. {\bf 4}
(1991), n.2, 263-275.

\bibitem[Ba08]{Ba2}
G. Barles, $C^{0,\al}$-regularity and estimates for solutions of
elliptic and parabolic equations by the Ishii \& Lions method,
Gakuto International Series, Mathematical Sciences and Applications,
vol. 30 (2008), 33-47.

\bibitem[BBBL03]{BBBL}
G. Barles, S.  Biton, M.  Bourgoing, O. Ley, Uniqueness results for
quasilinear parabolic equations through viscosity solutions'
methods. Calc. Var. Partial Differential Equations 18 (2003), no. 2,
159-179.

\bibitem[BBL02]{BBL}
 G. Barles, S. Biton, O. Ley, A geometrical approach to
the study of unbounded solutions of quasilinear parabolic equations,
Arch. Ration. Mech. Anal. 162 (2002), no. 4, 287-325.

\bibitem[BF04]{BF}  M. Bertoldi, S. Fornaro,
 Gradient
 estimates in parabolic problems with unbounded coefficients, Studia
 Math. 165 (2004), no. 3, 221-254,


\bibitem[BCQ10]{BCQ} R. Buckdahn, P. Cannarsa,
 M. Quincampoix,
Lipschitz continuity and semiconcavity properties of the value
function of a stochastic control problem, NoDEA Nonlinear
Differential Equations Appl. 17 (2010), no. 6, 715-728.

\bibitem[CC95]{CC} L. Caffarelli, X. Cabr\'e,
Fully nonlinear elliptic equations,  Amer. Math. Soc. Colloq. Publ.,
vol. 43, Amer. Math. Soc., Providence, RI, 1995.


\bibitem[Ce96]{Ce0} S. Cerrai,
  Elliptic and parabolic equations in ${\R}^n$ with coefficients
having polynomial growth, Comm. Partial Differential Equations 21
(1996), no. 1-2, 281-317.

\bibitem[Ce01]{Ce} S. Cerrai,  { Second order PDE's  in finite and
infinite dimensions. A probabilistic approach, } { Lectures Notes in
Math.}  1762, Springer Verlag, 2001.

\bibitem[Ch93]{Chen93}
Y. Z. Chen, $C^{1,\alpha}$ regularity of viscosity solutions of
fully nonlinear elliptic PDE under natural structure conditions, J.
Partial Differential Equations 6 (1993), 193-216.

\bibitem[CL89]{Chen-Li}
M. F. Chen, S. F. Li,  Coupling methods for multi-dimensional
diffusion processes,  Ann. Probab. 17 (1989), 151-177.

\bibitem[Cr91]{Cranston}
M. Cranston, Gradient estimates on manifolds using coupling, J.
Funct. Anal. 99 (1991), no. 1, 110-124.

\bibitem[Cr92]{Cranston1} M. Cranston,
{A probabilistic approach to gradient estimates,} Canad. Math. Bull.
35 (1992), 46-55.



\bibitem[CIL92]{CIL}
 M. G. Crandall, H.  Ishii,  P. L. Lions,
 User's guide to viscosity solutions of second order partial
 differential equations. Bull. Amer. Math. Soc. (N.S.)  27  (1992),
 no. 1, 1-67.


\bibitem[CC03]{CaCu} I. Capuzzo Dolcetta, A. Cutr\`i,  Hadamard and Liouville type results for fully nonlinear partial differential inequalities, Commun. Contemp. Math. Vol 5 (2003), 435-448.


\bibitem[DZ02]{DZ}  G. Da Prato, J. Zabczyk, Second Order Partial Differential
Equations in Hilbert Spaces, London Math. Soc. Lecture Note Ser., vol. 293, Cambridge, 2002.


\bibitem[DL06]{DaLio-Ley}  F. Da Lio,  O. Ley,
Uniqueness results for second-order Bellman-Isaacs equations under
quadratic growth assumptions and applications. SIAM J. Control
Optim. 45 (2006), no. 1, 74-106.

 \bibitem[EL94]{EL}
 K. D. Elworthy, X. M. Li,  Formulae for the derivatives of heat
semigroups. J. Funct. Anal. 125 (1994), no. 1, 252-286.

\bibitem[FGP10]{FGP} F. Flandoli, M. Gubinelli, E. Priola, Flow of
 diffeomorphisms for SDEs with unbounded H\"older continuous drift,
 { Bulletin des Sciences Math\'ematiques,}
 134 (2010),  405-422.





\bibitem[FS06]{Fle-Soner}
W. H. Fleming,   H. M.  Soner,    Controlled Markov processes and
viscosity solutions. (Second edition). Stochastic Modelling and
Applied Probability, 25. Springer, New York, 2006.

 \bibitem[FS89]{Fle-Sou} W. H. Fleming, P. E.  Souganidis,  On the existence of value functions of two-player, zero-sum stochastic differential games, Indiana Univ. Math. J. 38 (1989), no. 2, 293-314.


 \bibitem[Fr75]{Fr}
 A. Friedman, Stochastic differential equations and applications.
 Vol. 1, Academic Press, New York, San Francisco, London, 1975.

\bibitem[Fr85]{Fri}   M. I. Freidlin,
Functional Integration and Partial Differential Equations, Princeton
University Press, 1985.




\bibitem[GGIS91]{GGIS}
Y. Giga, S. Goto, H. Ishii, and M. H. Sato, Comparison principle and
convexity preserving properties for singular degenerate parabolic
equations on unbounded domains, Indiana Univ. Math. J. 40 (1991),
443-469.

 \bibitem[HR98]{HR} D. G. Hobson, L. C. G. Rogers,
 Complete models with stochastic volatility, Math. Finance
8 (1998), no. 1, 27-48.

\bibitem[Is89]{Ishii89}  H. Ishii, On uniqueness and existence of viscosity solutions of fully nonlinear second order elliptic PDE's, Comm. Pure Appl. Math. 42 (1989) 15-45.

\bibitem[IL90] {IL}
 H. Ishii, P. L. Lions,
  Viscosity solutions of fully nonlinear
second-order elliptic partial differential equations,  J.
 Differential Equations  83  (1990),  no. 1, 26-78.



\bibitem[Ko09]{K}  J. Kovats,
Differentiability properties of solutions of nondegenerate Isaacs
equations,
 Nonlinear Analysis: Theory, Methods \& Applications,
 71, Issue 12, 2009, 2418-2426.


\bibitem[Kr80]{Kr80}
N. V. Krylov, Controlled Diffusion Processes, Applications of
Mathematics 14,  Springer-Verlag 1980.

\bibitem[KS80]{KS}
 N. V. Krylov, M. V.   Safonov,   A property of the solutions of
 parabolic equations with measurable coefficients. Izv. Akad. Nauk
 SSSR Ser. Mat. 44 (1980), no. 1, 161-175.



\bibitem[KLL10]{KLL} M. Kunze, L.  Lorenzi, A. Lunardi,
 Nonautonomous Kolmogorov parabolic equations with unbounded
 coefficients, Trans. Amer. Math. Soc. 362 (2010),
 no. 1, 169-198.




\bibitem[LR86]{Lin-Rog}
T. Lindvall, C. Rogers, Coupling of multidimensional diffusions by
reflection, Ann. Probab. { 14} (1986), 860-872.



\bibitem[Li83]{Lions} P. L. Lions,
Optimal control of diffusion processes and Hamilton-Jacobi-Bellman equations. II. Viscosity solutions and uniqueness, Comm. Partial Differential Equations 8 (1983), no. 11, 1229-1276.

\bibitem[Lu98]{L1}
 { A. Lunardi,  Schauder theorems for linear elliptic and
parabolic problems with unbounded coefficients in $\R^n$, Studia
Math.  128  (1998),  no. 2, 171-198.}

\bibitem[MPW02]{MPW}
 G. Metafune, D. Pallara, M. Wacker,
 Feller semigroups on $\R^N$,
  Semigroup Forum  65  (2002),  no. 2, 159-205.


\bibitem[PW06]{PW} E. Priola, F. Y. Wang,
  Gradient estimates for diffusion
 semigroups with singular coefficients,
    J. Funct. Anal.  236 (2006),
 no. 1, 244-264.

\bibitem[PW10]{PW10} E.  Priola, F. Y. Wang,
A sharp Liouville theorem for elliptic operators, Atti Accad. Naz.
Lincei Cl. Sci. Fis. Mat. Natur. Rend. Lincei (9) Mat. Appl. 21
(2010),
 no. 4, 441-445.


\bibitem[St74]{St} H. B. Stewart,
 Generation of Analytic Semigroups by Strongly Elliptic Operators,
  Trans. Amer. Math. Soc. 199 (1974),
  141-162.


\bibitem[SV79]{SV}
D. W. Stroock, S. R. S. Varadhan, Multidimensional diffusion processes,
Springer-Verlag, 1979.

\bibitem[YZ99]{YZ} J. Yong, X. Y. Zhou,
Stochastic controls:
Hamiltonian systems and HJB equations, Springer, 1999.

\bibitem[WZ10]{WZ} F. Y. Wang,  T. S. Zhang, {   Gradient
estimates for stochastic evolution equations with non-Lipschitz
coefficients,}
 J. Math. Anal. Appl., 365, 2010,  1-11.

\end{thebibliography}
\end{document}